\documentclass[a4paper,12pt, reqno]{amsart}
\usepackage[utf8]{inputenc}
\usepackage{amsthm}
\usepackage{amsmath}
\usepackage{bm}
\usepackage{enumitem}
\usepackage{amsfonts}
\usepackage{amssymb}
\usepackage{mathtools}
\usepackage{appendix}
\usepackage{mathrsfs}
\usepackage{setspace}
\usepackage{comment}
\usepackage{xcolor}
\usepackage{ulem}
\usepackage{todonotes}
\usepackage{thmtools} 
\usepackage[foot]{amsaddr} 
\usepackage{appendix}
\RequirePackage[a4paper,top=2.54cm,bottom=2.54cm,left=1.90cm,right=1.90cm,%
                headsep=1em,includehead,includefoot]{geometry}
\usepackage[pdfdisplaydoctitle,colorlinks,breaklinks,urlcolor=blue,linkcolor=blue,citecolor=blue]{hyperref} 
\usepackage[nameinlink,capitalise]{cleveref}

\newcommand{\E}{\ensuremath{\mathbb{E}}}
\providecommand{\U}[1]{\protect\rule{.1in}{.1in}}
\newtheorem{theorem}{Theorem}

\newtheorem{definition}[theorem]{Definition}

\newtheorem{lemma}[theorem]{Lemma}

\newtheorem{proposition}[theorem]{Proposition}
\newtheorem{remark}[theorem]{Remark}

\newcommand{\T}{\mathbb{T}}
\newcommand{\R}{\mathbb{R}}

\newcommand{\EE}[1]{\mathbb{E}\left[#1\right]}
\newcommand{\M}{\mathbf{M}}

\renewcommand{\k}{\kappa}
\renewcommand{\U}{\mathcal{U}}

\newenvironment{acknowledgements}{%
  \begin{abstract}
}{%
  \end{abstract}
}

\newcommand{\Div}{\text{div}}
\title[Turbulent stretching of FENE polymers via  stochastic scaling $\&$ singular limits]
{Turbulent stretching of FENE dumbbell polymer model via special stochastic scaling and singular limits}
\author[Federico Butori]{Federico Butori}
\address{Scuola Normale Superiore, Piazza dei Cavalieri, 7, 56126 Pisa, Italy}
\email{\href{mailto:federico.butori at sns.it}{federico.butori at sns.it}}
\author[Yassine Tahraoui]{Yassine Tahraoui}
\address{Scuola Normale Superiore, Piazza dei Cavalieri, 7, 56126 Pisa, Italy}
\email{\href{mailto:yassine.tahraoui at sns.it}{yassine.tahraoui at sns.it}}
\date\today
\begin{document}
	\begin{abstract}
    We investigate  the stretching mechanism of \textit{Finitely Extensible Nonlinear Elastic} (FENE) model of polymers in a random turbulent flow. The turbulent model includes a   dominant  space-scale $\ell\sim N^{-1}$, a dominant time-scale $\tau$, and is white in time. Under suitable scaling assumption,  the polymer density equation, initially a stochastic  Fokker-Planck equation in the presence  of transport-stretching noise, converges weakly  as $N\uparrow \infty$ to a limit deterministic equation with a new extra term, a second order operator. This operator, whose shape has been predicted in the physical literature by other arguments, express a sort of average `turbulent stretching' effect. With respect to other derivation of this effective model, the main novelty of our approach is that the deterministic limit is obtained pathwise, without having to take averages with respect to different realizations of the random flow.
    Next, we consider the limit as $\tau \downarrow 0$ and we identify the stationary distribution of the polymer length. The analysis is carried out in appropriate weighted spaces, which take into account the singularity of the FENE force near the boundary and the no-flux boundary condition, and combines stochastic scaling limit and singular limit techniques.    \end{abstract}
	\maketitle
\tableofcontents

\section{Introduction}

In this work, we are interested in the effect of turbulence on polymers, by using small-scale stochastic modeling  of the velocity of the flow (Kraichnan-Batchelor regime).  Namely, the velocity of the small scales of the  flow is given by some stochastic process, which is  white in time  with some  smooth covariance operator in space. The  kinetic equation of the polymers  is given by the \textit{Finitely Extensible Nonlinear Elastic} (FENE) model.\\

Polymers are complex molecular systems,  which can be modeled using a chain of springs, see \textit{e.g.}  \cite{BriTL09}. They possess many interesting properties, including the ability to change large-scale statistics of the advecting flow. In particular,  the presence of low concentrations of dilute polymers produce  a drag reduction in turbulence and  leads to energy savings \cite{Gyr1995drag}.
 In a turbulent fluid, polymers are usually found in two states called coil and stretched. The coil state is like an  ellipsoidal rolled chain, which may be more or less elongated. The stretched state develops when the chain becomes elongated, more similar to a straight line than a sphere. Turbulent flows may facilitate the appearance of the stretched state. That happens because polymers are deformed and stretched by the gradients of the flow's velocity  and random stretches may lead to a macroscopic increase of the polymer's elongation (see \cite{FalGawVer} for an introduction). When the polymers pass from one state to the other we speak of coil-stretch transition. Understanding the stretching of polymers  by a random flow is  one of the 
 central question in the elastic turbulence theory and we refer \textit{e.g.} to  \cite{Steinberg2021elastic}  for its various industrial applications.\\

There exist several  physical studies about the coil–stretch transition in random flows in the case of the Hookean dumbbell model, for instance \cite{Balk,Celani2005polymer,Chertkov}. In such regimes, as first understood in \cite{Balk} with an argument based on large deviations, the polymers' elongation show a power-law distribution (heavy tails), with certain parameters related to turbulence statistics, which are difficult to identify explicitly. The knowledge of these parameters is very important because it allows us to predict the coil-stretch transition. In other words, this makes it possible to determine the threshold from which the polymer begins to exert an effective feedback on the flow and to cause a drag reduction.
The analysis of the FENE model in turbulent flows is more sparse and we refer to \cite{PicardoLanceVinc, Afonso2005nonlinear}. A discussion on the relation between our results and \cite{Afonso2005nonlinear}  will be considered in \autoref{sect-physics}.
Now, let us recall the FENE model.
\subsection{The FENE model.} Let us give a brief presentation of the polymer modeling and we refer  \textit{e.g.}  \cite[Section 4.2.]{BriTL09} for further details.
 Let   $R_t$  be a   vector, which  represents the orientation and elongation of  the polymer chain and   $X_t$ (the    center  of  mass) denotes its position. The polymer is embedded into a fluid having velocity $u(t,x)$, which stretches $R_t$ by $\nabla u(x,t)$. The equation for $R_t$ contains also a damping term with relaxation time $\beta$ and Brownian fluctuations which keeps the equilibrium length, at rest ($u=0$), equals to $R_0$. Mathematically, $R_t$ is a stochastic process and  the dynamics is given by 
  \begin{align}\label{2-1}
	\begin{cases}
	dR_t&=\nabla u(t,X_t) R_tdt-\dfrac{1}{\beta}\nabla \U(R_t)dt+\sqrt{\dfrac{2R_0^2}{\beta}}d\mathcal{W}_t, \\
	dX_t&=u(t,X_t)dt.
	\end{cases}
	\end{align}
   
   Up to considering $R_t/R_0$, we can always set $R_0=1$. 
		Commonly, two types of potentials $\U$ are employed: the Hookean model (\textit{i.e.} $\U(X)=\frac{1}{2}\vert X\vert^2$) and \textit{Finitely Extensible Nonlinear Elastic} (FENE) model \textit{i.e.} $\U(r)=-\dfrac{\k  b}{2} \log(1-\dfrac{\vert r\vert^2}{b})$ where $\k>0$ is the spring constant and  $\sqrt{b}$ is the upper limit of the length extension of the polymer chain. Notice that the force $F=\nabla\U$, in  FENE model, tends towards infinity when the norm of $R_t$ tends towards a limiting value $\sqrt{b}$. On the other hand, we have $ \nabla \U(r)=\dfrac{\k r}{1-\vert r\vert^2/b}$. 
         Recall that $R_t\in B(0,\sqrt{b})$ and by 
        considering  $\widetilde{R}_t=R_t/\sqrt{b} \in B(0,1) $ we get
\begin{align}\label{FENE}
	\begin{cases}
	d\widetilde{R}_t&=\nabla u(t,X_t) \widetilde{R}_tdt-\dfrac{\kappa}{\beta}\dfrac{\widetilde{R}_t}{1-\vert \widetilde{R}_t\vert^2}dt+\sqrt{\frac{2}{\beta}} d\widetilde{\mathcal{W}}_t, \\
	dX_t&=u(t,X_t).
	\end{cases}
	\end{align}
  Let   $ x\in \mathbb{T}^d \text{ and } r\in D:=B(0,1) \subset \mathbb{R}^d, d=2,3$ and 
consider	\eqref{FENE}
	with a given law $f_0 (x, r)$ of the initial datum,  the associated
	Fokker-Planck equation for the polymers density $f := f (x, r, t)$ is given by
	\begin{align*}
	\begin{cases}\partial_tf+\Div_x(uf)+\Div_r((\nabla u r)f)=\dfrac{1}{\beta}\Div_r \big(\dfrac{\k r}{1-\vert r\vert^2}f+	\nabla_rf\big)\\	
		f|_{t=0}=f_0.
			\end{cases}\end{align*}
            

    Our aim is to derive rigorously (\textit{via} stochastic diffusion and singular limits) an explicit formula, depending on various modelling parameters, describing the stationary polymers' elongation distribution in a (synthetic) turbulent regime. We also discuss  the relation between the  resulting formula and  the one obtained in the physical literature by statistical and numerical approaches \cite{Afonso2005nonlinear,Balk}.
    
\subsection{The main result}\label{scales-modeling-sect}
Suppose that the fluid velocity $u$ can be decomposed into a large-scale component $u_L$ and small-scale components $u_s$. In other words,  $u=u_L+u_s$.  In general, $u_L$ is a smooth function and $u_s$ contains the fluctuations i.e. $u_s$ is the `turbulent' part of the flow. We can assume for example that  $u_L\in C([0,T],C^2(\T^2;\R^2))$    such    that    $\Div_x(u_L)=0$ while we model $u_s$ with a synthetic and stochastic model of a turbulent flow 
\begin{equation}\label{small-scale-model-main}
    u_s(x,t)=\sum_{k\in K}\sigma_{k}^{N,\tau}\left(  x\right)\circ  \dot W_{t}^{k}.
    \end{equation}
    Where the definition of the noise coefficients $\sigma_k^{N, \tau}$ is given below in \autoref{assumption_noise-2D}. The quantities $N^{-1}, \tau$ appearing in the definition of the $\sigma_k$ represent the dominant spatial scale of the noise and its dominant time scale $\tau \rightarrow 0$.
We choose the Stratonovich multiplication $\circ$ both in virtue of Wong-Zakai principle (a white noise is the idealization of smooth noise) and because of conservation laws.
Since we are interested on the effect of small scales on the dynamic of polymers, without loss of generality (see \textit{e.g.} \cite{Flandoli-Tahraoui}) we can assume that $u_L\equiv 0$ in the rest of the paper. \\

Now, we are in position to formulate the problem we want to study. 
We consider a dilute  family of polymers subject to equations
\eqref{FENE}, thus described by the kinetic equation for the density $f^{N,\tau}:=f^{N,\tau}\left(
x,r,t\right)  $ of polymers with position $x$ and length $r$ at time $t$, where $u^{N, \tau}$ is given by \eqref{small-scale-model-main}.
 Thus  $f^{N,\tau}$ satisfies the following Stochastic PDE
\begin{align}\label{Stratomovich-form}
	\begin{cases}&d f^{N,\tau}+\sum_{k\in K}%
\sigma_{k}^{N,\tau}\cdot \nabla_xf^{N,\tau} \circ dW_{t}^{k}+ \sum_{k\in K}%
(\nabla_x\sigma_{k}^{N,\tau}r)\cdot\nabla_rf^{N,\tau}\circ dW_{t}^{k}\\&=	\dfrac{1}{\beta}\Div_r \big(\k rF(r)f^{N,\tau}+ \nabla_rf^{N,\tau}\big) dt; \qquad (x,r)\in \mathbb{T}^2\times D,\\	
 	&f^{N,\tau}|_{t=0}=f_0^\tau.
			\end{cases}\end{align}
      As usual, in order to perform the rigorous mathematical analysis, we understand \eqref{Stratomovich-form} in its formally equivalent It\^o form. Based on \cite[Lemma 2 and 3]{Flandoli-Tahraoui}, by assuming that the noise is space-homogeneous and satisfies mirror
symmetry property (see \autoref{assumption_noise-2D}), the It\^{o} form of the stochastic Fokker Planck equation \eqref{Stratomovich-form} is given by
	\begin{align}\label{Ito-FP}
		\begin{cases}&df^{N,\tau}=	\dfrac{1}{\beta}\Div_r \left(\k r F(r)f^{N,\tau}+	\left(I +\dfrac{\beta}{2}  A^{N,\tau}(r) \right)\nabla_rf^{N,\tau}\right)dt +\alpha_{N,\tau}\Delta_x f^{N,\tau}dt\\[0.15cm]
            &-\sum_{k\in K}\sigma_k^{N,\tau}.\nabla_xf^{N,\tau}dW^k_t -\sum_{k\in K}(\nabla_x \sigma_k^{N,\tau}r).\nabla_rf^{N,\tau} dW^k_t  ;\quad (x,r)\in \mathbb{T}^2\times D,\\[0.15cm]
			&f^{N,\tau}|_{t=0}=f_0^\tau.
	\end{cases}\end{align}
    Where the matrix $A^{N,\tau}$ is computed in \Cref{lem-matrix-r}. The presence of the term involving the matrix $A^{N, \tau}$, coming from the It\^{o}-Stratonovich corrector, alters the structure of the equation and of the equilibrium state of the system (\textit{c.f.} \autoref{sec:mod_operator}). 
The system is completed by adding appropriate no-flux boundary conditions, see 
\autoref{subsect-opert-deter} and  also \autoref{Def-sol-QR-*}.\\
 Let us now introduce 
 $$\M^\alpha_0 (r) = \mathbf{Z}_0^{-1} \left(
\dfrac{1 - \vert r\vert^{2}}{1 + \alpha \vert r\vert^{2}}
\right)^{\frac{\kappa}{2\left(1 + \alpha\right)}}, \qquad \mathbf{Z}^\alpha_0=\int_D \left(
\dfrac{1 - \vert r\vert^{2}}{1 + \alpha \vert r\vert^{2}}
\right)^{\frac{\kappa}{2\left(1 + \alpha\right)}} dr
.$$
If one formally takes the expected value of \eqref{Ito-FP}, under the scaling assumption in \autoref{assumption_noise-2D}, one can easily conjecture that 
\begin{align*}
    \EE{f^{N. \tau}} &\xrightarrow[]{N\rightarrow\infty} f^\tau, \end{align*}
where $f^\tau$ solves
    \begin{align}\label{formal-limit}
    \partial_t f^\tau &= \frac{1}{\beta}\Div_r \left(\k r F(r)f^{\tau}+	\left(I +\dfrac{\beta}{2}  A^{\tau}(r) \right)\nabla_rf^{\tau}\right)dt 
\end{align} 
If then $\alpha$ is chosen properly with respect to the parameters $\beta, \kappa$ of the problem, one can check that  $$f^\tau=\M_0^\alpha\otimes \rho(x)$$ is a stationary solution of \eqref{formal-limit}, for every $\rho(x)$. By then taking another limit where $\beta \sim \tau $\footnote{Recall that $\beta$  and $\tau$ denote polymer relaxation time and dominant time-scale of the fluid flow, respectively.} are of the same order and converge to zero while $\beta A^\tau$ converges to some $A$ (see \autoref{assumption_noise-2D}), one expects also that $f^\tau$ actually converges to a stationary solution of the form $\M_0^\alpha \otimes \rho_0(x)$, for some $\alpha$, still depending on all the parameters and their relations.
The main contribution of this manuscript is to show that, in a specific range of parameters (see \autoref{TH1-*_nocutoff}), the convergence in mean actually holds pathwise (in a weak sense), namely to prove that, as $N\to \infty$ and then $\tau\sim  \beta\to 0$, assuming $\displaystyle\lim_{\tau\to 0} f^\tau_0=f_0 \text{ in } L^2(\T^2\times D; \dfrac{1}{\mathbf{M}_0}dxdr)$, setting $\rho_0=\int_Df_0dr $ it holds
 \begin{align*}
   &f^{N,\tau}   \rightharpoonup \rho_0\otimes\mathbf{M}_0  \quad \text{ in }  L^2(\Omega;L^2(0,T;L^2(\T^2\times D; \dfrac{1}{\mathbf{M}_0}dxdr))).
 \end{align*} 
Outside of this regime of parameters, namely when the intensity of the turbulence is large with respect to the relaxation time of the polymers and to $\kappa$, we encounter serious difficulties in the analysis, so that we are even unable to construct meaningful solutions to \eqref{Ito-FP}. These difficulties arise when trying to control the effect of the noise on the solution $f$ for values of $r$ near the maximal elongation (where the restoring force $F$ becomes singular) (\textit{c.f.} \autoref{rem:coercivity_fails}). To circumvent this problem we regularize the system, introducing a cut-off on the noise, then we perform the $N\rightarrow +\infty$ scaling limit on the system with cut-off. Next we remove the cut-off and we perform the $\tau \rightarrow 0$ limit, obtaining at the end the same stationary density as above.  
   We refer to \autoref{Section-main-results} for the rigorous statements and the assumptions. The  novelty of our work lies
in the rigorous analysis based on analytical tools derived from the stochastic diffusion limit and singular limits of kinetic equations. We also refer to \autoref{sect-physics}  for a  discussion on the physical part.
\\

Next, let us present a brief overview of existing mathematical results that are related to our work.
              \subsection{Stochastic diffusion and singular limits}
               The origin of stochastic diffusion limits, based on the It\^o-Stratonovich corrector, is the paper \cite{Galeati} that threw a new light on passive scalars subject to turbulent models. The result of \cite{Galeati} has been generalized in several directions, showing in particular its strength for
nonlinear scalar problems (see \cite{FlaElibook} for a review). 
In particular, we refer to \cite{Flandoli2021scaling,Flandoli2024quantitative} for the results in this direction about  the  2D Euler  and  2D Navier-Stokes equations.    We refer  also to \cite[Section 1.2]{Flandoli-Tahraoui} for extended literature review on the stochastic diffusion limits and transport noise. It is worth underlying that the mentioned results concern scalar problems with transport noise. However,   when the turbulent fluid acts on vector fields,  in addition to   transport there is    
\textit{stretching}: namely the differential $D\phi_{t}$ of the Lagrangian flow $\phi_{t}$ of a fluid   produces a modification of the length of vectors $v$,  which leads to  an increase of length for many possible directions $v$. \\          

More challenging has been adapting the ideas to advected vector fields, because of lack of control on the stochastic stretching, when the diffusive scaling limit is performed. Nevertheless, positive results for particular models have been obtained, adapting the classical scaling of \cite{Galeati}, see  \textit{e.g.} \cite{ButFlaLuo, butori_meanfield_2025, butori3DHNSE, FL1,FlandoliLuo2024,Papini}. Recently, after considering a new scaling such as   the noise covariance goes to zero but a suitable covariance built on derivatives of the noise converges to a non zero limit, new results have been obtained in \cite{BFLT2024,Flandoli-Tahraoui} in the case of   stochastic transport-stretching. In \cite{BFLT2024}, we  introduced a background stochastic Vlasov equation behind stochastic transport and advection equations which gives additional information on the fluctuations and oscillations of solutions  based on Young measures.  In the case of a passive vector field the background Vlasov equation adds completely new statistical information to the stochastic advection equation and  the theory developed may help to recognize the existence of large values of the length of the vectors, where   the physical phenomenon of magnetic dynamo is an example. In \cite{Flandoli-Tahraoui}, the authors considered the new scaling by considering the Fokker-Planck equation \eqref{Stratomovich-form} of polymers density in the Hookean case and obtained that the polymers density equation  converges weakly, as the scaling parameter $N\uparrow \infty,$ towards the solution of a deterministic equation, with a new diffusion term in the variable $r$ reflecting the stretching effect on the dynamics. In \cite{tahraoui2025small}, the  author completed the study of \cite{Flandoli-Tahraoui} by studying  the singular limit of  Fokker-Planck equation of   polymers density as the dominant time-scale of small scale component of turbulent flow goes to zero and showed that the limit density has  generalized Cauchy distribution  for the end-to-end vector, with  certain parameters, related to turbulence modeling of the small scales.  \\

Finally, it is worth mentioning the asymptotic behavior with respect to some parameters appearing in kinetic equations have been studied extensively, where the the main interest when studying asymptotic is to derive equations of macroscopic quantities such as density and probability distribution with respect to mesoscopic variable such as particle velocity and the polymers end-to-end vector in our case. Without trying to be exhaustive,  the author studied  in \cite{Poupaud1992} the  asymptotic of electron distribution function in semiconductor kinetic theory where the equation is  given by Boltzmann transport equation for high electric fields. He proved that the limit  distribution function is given by the tensor product of the density (satisfying a  linear transport equation) and a probability distribution function of the velocity variable  satisfying an  appropriate PDE. 
We mention also   \cite{Masmoudi-Vlasov,Nieto2001high} for  the  asymptotic of the  Vlasov-Poisson-Fokker-Planck Systems. We refer \textit{e.g.} to \cite{goudon2005,jabin2000,Poupaud2000} for more results about   asymptotics of the kinetic equations.\\
\subsection{Physical considerations}\label{sect-physics}
We end this introduction by adding some physical considerations concerning the result of our work. In \cite{Afonso2005nonlinear} the stationary regime of the FENE model is investigated in a fully solvable model. Here, our model is similar, still based on a gaussian description of the turbulent flow, however in that work, the authors work on a fixed realization of the turbulent flow and then average over the realizations. In our setting this procedure would give immediately the correct result (taking the average of \eqref{Ito-FP}), but would completely disregard the fluctuations with respect to different realization of the flow. Our approach, based on a scaling limit with respect to an infinite separation of scales, can be thought as a law of large numbers with respect to theirs, and it is therefore more robust. In order to compare directly our result to theirs, recall that in their paper, the stationary distribution of polymer elongation (integrated along angular variables) in dimension $d=2$ is given by 
\begin{equation}\label{stationari_vincenzi}
f(R)=  \mathbf{Z}^{-1} R\left(1+ \frac{Wi}{2}\frac{R^2}{R_0^2}\right)^{-h}\left(1 - \frac{R^2}{b}\right)^h
\end{equation}
Where $R_0$ is the average length of the polymers in absence of flow due to thermal fluctuations, $b$ is the maximal elongation of the polymers, $\mathbf{Z}$ is a normalization constant and $Wi$ is the Weissenberg number which in random flows can be defined as $Wi=\Lambda \tau$ where lambda is the maximum Lyapunov exponent of the random flow and $\tau$ is the relaxation time of the polymers in absence of flow and 
$$h=\frac{b}{2{R_0^2} + bWi}.$$
Comparing their microscopic model with our rescaled one \eqref{FENE}, we see that in our model $b=1$ and $R_0^2= \kappa^{-1}$. Moreover, for the random flow defined in \autoref{assumption_noise-2D}, in the scaling limit, it holds that for $N\rightarrow \infty$, $\Lambda^N\rightarrow  2k_T$ (see \cite{monin_statistical_1987})
while the $\alpha$ appearing in our final formula is given by $\alpha= \beta \kappa_T$, and since in our rescaled model the relaxation time of the polymers is $\tau=\beta/\kappa$, we get $\alpha = \k Wi /2$ and
we see that our formula exactly matches the one found in \eqref{stationari_vincenzi}. \\

 The exponent $h= \k/2(1+\alpha)= \k/2(1+ \beta k_T)$ is the crucial parameter governing the so called coil-stretch transition. Indeed the behavior of the stationary distribution, changes somehow abruptly when $h$ crosses the critical value $1$, as observed in \cite{Afonso2005nonlinear}, revealing a sort of phase transition. 
 In application, the ratio between the maximal and the mean elongation $R_0^{-1} = \k$ usually lies between $10$ and $100$, thus the exponent $h$ is small only if the product $\beta k_T$ is at least as large as $\kappa$. In our main result, we are able to prove directly the convergence to the equilibrium system only under the assumptions that $h>1$ and $\beta k_T$ is small enough (`coil' regime), in particular our assumptions do not allow to reach the coil-stretch transition for physical values of $\kappa$ (\textit{c.f.} \autoref{TH1-*_nocutoff}). For this reason, in order to justify the final formula without restrictions on the parameters, we introduce also a cut-off on the system which tames the effect of the noise near the boundary. Performing first the scaling limit and then removing the cut-off, we see that we reach the same stationary distribution as above, thus partially completing the analysis. Avoiding the cut-off outside the `coiled' regime is an open problem that seems out of reach with the current techniques, as it even seems impossible to construct weak solutions to the Fokker-Plank equation \eqref{Ito-FP} (see \autoref{rem:coercivity_fails}). 

\subsubsection*{Structure of the paper} The manuscript is organized as follows: in \autoref{Section1-desciption}, we present some preliminaries. Then, we  formulate the problem,  collect the main  results and describe the strategy of the proof in \autoref{Section-main-results}.  \autoref{Section-exit-uniq} is devoted to the  proof of  existence and uniqueness within the class of \textit{quasi-regular weak solution} to \eqref{Ito-FP_eps-V2}, proving \autoref{TH1-*} which includes also a  uniform estimate with respect to $N$. In  \autoref{Section-diffusion-limit}, we     consider  the diffusion scaling limit    as $N\uparrow +\infty$, where   the  extra term emerges. Next,  \autoref{section-remove-cut-off} concerns  removing the cut-off in the corresponding regime. Finally, \autoref{Sec-singular-limit} contains the proof of the singular limit as $\tau\downarrow 0$ and the identification of the stationary formula in the polymer length.

\section{Notation, definitions and problem formulation}\label{Section1-desciption}
Given a separable Banach space $\mathbf{X}$, since $L^{\infty}(0,T;\mathbf{X})$	is	not	separable,	
we	introduce for convenience	the	following	space:
$$L^2_{w-*}(\Omega;L^\infty(0,T;\mathbf{X}))=\{ 	u:\Omega\to	L^\infty(0,T;\mathbf{X})	\text{	is	 weakly-* measurable		and	}	\E\Vert	u\Vert_{L^\infty(0,T;\mathbf{X})}^2<\infty\},$$
where	weakly-* measurable	stands	for	the	measurability	when	$L^\infty(0,T;\mathbf{X})$	is	endowed	with	the	$\sigma$-algebra	generated	by	the	Borel	sets	of	weak-*	topology, we recall that (see 	\cite[Thm. 8.20.3]{Edwards}) $$ L^2_{w-*}(\Omega ;L^\infty(0, T;\mathbf{X}))\simeq \left(L^2(\Omega ;L^1(0, T;\mathbf{X}^\prime)) \right)^\prime .$$	
We recall also that  $C_w([0, T];\mathbf{Y})$    denotes 	the	Bochner space of weakly continuous functions with values in a Banach space $\mathbf{Y}.$\\
We will be using the following notations:
\begin{itemize}
	\item   $(\nabla_x  g)_{i=1,2}=(\dfrac{\partial g}{\partial  x_i})_{i=1,2}; (\nabla_r  g)_{i=1,2}=(\dfrac{\partial g}{\partial  r_i})_{i=1,2} $ for scalar   function    $g$.
	\item   $(\nabla_x  g)_{i,j=1,2}=(\dfrac{\partial g^i}{\partial  x_j})_{i,j=1,2}; (\nabla_r  g)_{i,j=1,2}=(\dfrac{\partial g^i}{\partial  r_j})_{i,j=1,2} $ for vector  valued   function   $g$.
	\item     $\Delta_xg=\displaystyle\sum_{i=1}^2\dfrac{\partial^2 g}{\partial^2  x_i};  \Delta_rg=\displaystyle\sum_{i=1}^2\dfrac{\partial^2 g}{\partial^2  r_i}$  for scalar   function    $g$.   
\end{itemize}
and     $\Div_{x/r}   g=\nabla_{x/r}    \cdot   g$  for vector  valued   function   $g$.
\subsection{Weighted function spaces}\label{sect-funct-spaces}
Due to the degeneracy of the restoring force $\nabla \U$ on $\partial D$, we will mainly be working with appropriate weighted Sobolev spaces on the domain $D$ that we now introduce. 
Given a bounded function $\mathcal{J}> 0$ a.e. in $D$, set
\begin{align*}
    L^2_\mathcal{J}&=\{g:D \to \mathbb{R}: \quad \Vert g\Vert_{2,\mathcal{J}}^2=\int_D \left |\dfrac{g}{\mathcal{J}}\right |^2 \mathcal{J}dr<+\infty \},\\
        H^1_\mathcal{J}&=\{g:D \to \mathbb{R}: \quad \Vert g\Vert_{H_\mathcal{J}^1}^2=\int_D \left |\dfrac{g}{\mathcal{J}}\right|^2+ \left|\nabla\left(\dfrac{g}{\mathcal{J}}\right)\right|^2 \mathcal{J}dr <+\infty \}.
\end{align*}
Now,  we introduce the following spaces \begin{align*}    
V_\mathcal{J}:= L^2(\mathbb{T}^2;H^1_\mathcal{J}), \quad  H_\mathcal{J}=L^2(\mathbb{T}^2;L^2_\mathcal{J}).
\end{align*}
These spaces are 
equipped with their natural inner product. Namely, we have
\begin{align*}
(h,g)&:= \int_{\mathbb{T}^2\times D} h(x,r) g(x,r)dxdr,\quad	\forall	g,h \in L^2(\mathbb{T}^2\times D),\\
	(h,g)_{H_\mathcal{J}}&:= \int_{\mathbb{T}^2\times D} h(x,r) g(x,r)\dfrac{dr}{\mathcal{J}}dx,\quad	\forall	g,h \in H_{\mathcal{J}},\\
    (h,g)_{V_\mathcal{J}}&:= \int_{\mathbb{T}^2\times D} \hspace{-0.4cm}\nabla\left(\dfrac{h}{\mathcal{J}}\right)\cdot \nabla\left(\dfrac{g}{\mathcal{J}}\right) \mathcal{J}drdx, \qquad  [h]_{V_\mathcal{J}}^2 := \int _{\T^2 \times D}  \left|\nabla \left(\frac{h}{\mathcal{J}}\right)\right|^2\mathcal{J} drdx,  \\
    ((h,g))_{V_\mathcal{J}}&:= (h,g)_{H_\mathcal{J}}+ (h,g)_{V_\mathcal{J}},\quad	\forall	g,h \in V_\mathcal{J}.  
\end{align*} 
Let also $V^\prime_\mathcal{J}$ denote the dual space of $V_\mathcal{J}.$ We have the following isometries and embeddings: $$H_\mathcal{J}\simeq  L^2(\mathbb{T}^2\times D, \dfrac{dr}{\mathcal{J}}dx), \quad V_\mathcal{J} \hookrightarrow H_\mathcal{J} \hookrightarrow V_\mathcal{J}', \quad H_\mathcal{J}\hookrightarrow  L^2(\mathbb{T}^2\times D).$$
It should be noted that $L^2_\mathcal{J}$ and $H^1_\mathcal{J}$ are not classical weighted Sobolev spaces but one  checks easily that   $L^2_\mathcal{J}$ and $H^1_\mathcal{J}$  are Hilbert spaces.\\

We will make use use the following class of weights depending on a parameter $k>0$. 
  \begin{align}\label{weight}
   \mathbf{M}(r):=\dfrac{1}{\mathbf{Z}}(1-\vert r\vert^2)^{\frac{\kappa}{2}}, \qquad \mathbf{Z}= \int_D (1-\vert r\vert^2)^{\frac{\kappa}{2}} dr
\end{align}
Let us recall the following results, which we use in our analysis.
\subsubsection{The operator $\mathcal{L}_k$}\label{subsect-opert-deter} Given that the cases $\kappa\geq 2$ and $0<\kappa<2$ exhibit different properties, we recall some results of \cite[Subsection 3.3.]{Masmoudi2008}.\\

Let $\kappa>0$ and consider the following unbounded operator 
\begin{align}\label{operator-L}
   \mathcal{L}_\kappa(f):= -\Div_r \big(\kappa F(r) rf+	\nabla_r f\big), \quad F(r)=\dfrac{1 }{1-\vert r\vert^2}.
\end{align}
Since the weight $\M$ satisfies $\k \M F(r) + \nabla\M=0 $, the operator $\mathcal{L}_\k$ can be rewritten in the following way
$$\mathcal{L}_\k(f)=-\Div_r\left(\mathbf{M}\nabla_r\left(\dfrac{f}{\mathbf{M}}\right)\right).$$
Thus, we are naturally led to consider, on the space $L^2_\mathbf{M}$, the following domain 
\begin{align}\label{domain-operator-L}
    D(\mathcal{L}_k)=\{ g\in H^1_\mathbf{M} | \quad \mathcal{L}_\kappa g\in L^2_\mathbf{M}| \quad   \mathbf{M}\nabla_r(\dfrac{g}{\mathbf{M}})\cdot \eta\vert_{\partial D}=0\}.
\end{align}
We introduce also the following Hilbert space
\begin{align*}
    \mathcal{H}^2=\{ g\in  H^1_\mathbf{M} | \quad \int_D \left[\Div_r\left(\mathbf{M}\nabla_r\left(\dfrac{g}{\mathbf{M}}\right)\right) \right]^2\dfrac{dr}{\mathbf{M}}<+\infty\}.
\end{align*}
Concerning the boundary condition $\mathbf{M}\nabla_r(\dfrac{g}{\mathbf{M}})\cdot \eta\vert_{\partial D}=0,$ they should be understood in the weak sense \textit{i.e.} 
\begin{align}\label{BC-sense}
\text{ for any } h\in H^1_\mathbf{M}: \quad    \int_D h \mathcal{L}_\kappa g \dfrac{dr}{\mathbf{M}}=\int_D \mathbf{M}\nabla_r(\dfrac{g}{\mathbf{M}})\cdot \nabla_r(\dfrac{h}{\mathbf{M}}) dr.
\end{align}
Next, let us recall the following remark (see \cite[Remark 3.8]{Masmoudi2008}), which clarify  a key difference  between the cases $\kappa\geq 2$ and $0<\kappa<2.$
\begin{remark}\label{rmq-BC-cases}
     If $\kappa\geq 2,$ the boundary condition $\mathbf{M}\nabla_r(\dfrac{g}{\mathbf{M}})\cdot \eta\vert_{\partial D}=0$ is a consequence of the fact that $g\in \mathcal{H}^2$ and hence $D(\mathcal{L}_\kappa)=\mathcal{H}^2.$ The fact that $\mathbf{M}\nabla_r(\dfrac{g}{\mathbf{M}})\cdot \eta\vert_{\partial D}=0$  is not needed in this case should be related to a similar property of the SDE \eqref{FENE}. Namely, the process $\widetilde{R}_t$ does not reach the boundary $\partial D$ in finite time almost surly, see \cite{Jourdain2003mathematical,Jourdain2004existence}.
         In the case  $0<\kappa<2,$ the process $\widetilde{R}_t$  in \eqref{FENE} reaches the boundary $\partial D$ in finite time almost surely and this explains why the boundary condition for the operator $\mathcal{L}_\kappa$ is needed. Moreover, the inclusion $D(\mathcal{L}_\kappa) \subset \mathcal{H}^2$ is strict in this case.
\end{remark}

We now recall the following result from \cite[Proposition 3.6]{Masmoudi2008}, which we use in the construction of the solution to \eqref{Ito-FP} \textit{via} the Galerkin approximation.
\begin{proposition}\label{propo-construct-base}
$\mathcal{L}_\kappa$ is self-adjoint and positive. Moreover, it has a discrete spectrum formed by a sequence $(l_i)_{i\in \mathbb{N}}$ such that $l_i\uparrow +\infty$ when $i\uparrow +\infty$.
\end{proposition}
Consequently, the corresponding eigenvectors $(e_i)_{i\in \mathbb{N}}$ satisfying $\mathcal{L}_\kappa e_n=l_ne_n,\  n\in \mathbb{N}$ form  a countable orthonormal basis for the space $L^2_\mathbf{M}.$ Moreover, setting $\lambda_i=l_i+1$, we notice that
    \begin{align}\label{basis-eigen}
      (u,e_i)_{H^1_\mathbf{M}}=\lambda_i(u,e_i)_{L^2_\mathbf{M}}, \quad \forall u\in H^1_\mathbf{M}, \ \forall i\in \mathbb{N}.
    \end{align}
  Therefore,  $\{\tilde{e}_i=e_i/\sqrt{\lambda_i}\}_{i\in \mathbb{N}}$ is  an orthonormal basis of $H^1_\mathbf{M}.$ 
To avoid cluttering the notation, we use the following abbreviated notation:
\begin{align*}    
V:= L^2(\mathbb{T}^2;H^1_{\mathbf{M}})\hookrightarrow H=L^2(\mathbb{T}^2;L^2_\mathbf{M})\simeq  L^2(\mathbb{T}^2\times D, \dfrac{dr}{\mathbf{M}}dx) \hookrightarrow  V^\prime \text{ and } H\hookrightarrow  L^2(\mathbb{T}^2\times D).
\end{align*}

\subsubsection{The operator $\mathcal{L}_{\kappa,\gamma}$}\label{sec:mod_operator}
The presence of the noise will alter the structure of the deterministic equation due to the introduction of the Stratonovich corrector. It will be convenient then to introduce a corrected weight that take into account the effect of the noise. Let $\kappa,\gamma>0$ and consider the modified weight
\begin{align}\label{weight-limit}
     \mathbf{M}_0(r)= \dfrac{1}{\mathbf{Z}_0} \left(
\dfrac{1 - \vert r\vert^{2}}{1 + \gamma \vert r\vert^{2}}
\right)^{\frac{\kappa}{2\left(1 + \gamma\right)}}, \qquad  \mathbf{Z}_0=\int_D \left(
\dfrac{1 - \vert r\vert^{2}}{1 + \gamma \vert r\vert^{2}}
\right)^{\frac{\kappa}{2\left(1 + \gamma\right)}} dr.
\end{align}
It is easy to check that 
\begin{equation}\label{elliptic-eq-M0}
    \k rF(r)\M_0 + \nabla_r \M_0+ \gamma A(r)\nabla_r \M_0 =0.
\end{equation}
Thus we naturally introduce the operator
\begin{align}\label{operator-limit}
 \mathcal{L}_{\kappa,\gamma}(g)=-\Div_r\left(\mathbf{M}_0 (I+\gamma\mathcal{A}(r))\nabla_r(\dfrac{g}{\mathbf{M}_0})\right),    \quad  \mathcal{A}(r)= (3 \left\vert r\right\vert^2I-2r\otimes r),
\end{align}
 which is  defined on  $L^2_{\mathbf{M}_0}$ with domain $$D(\mathcal{L}_{\kappa,\gamma})=\{ g\in H^1_{\mathbf{M}_0};\quad \mathcal{L}(g)\in L^2_{\mathbf{M}_0} \text{ and }  \mathbf{M}_0 (I+\mathcal{A}(r))\nabla_r(\dfrac{g}{\mathbf{M}_0})\cdot \eta\vert_{\partial D}=0\}.$$ Similarly to \cite[Prop. 3.6.]{Masmoudi2008}, one can check that $-\mathcal{L}_{\kappa,\gamma}$ is self adjoint and positive  on $L^2_{\mathbf{M}_0}$.
Concerning the boundary condition $\mathbf{M}_0 (I+\mathcal{A}(r))\nabla_r(\dfrac{g}{\mathbf{M}_0})\cdot \eta\vert_{\partial D}=0,$ they should be understood in the natural weak sense  
\begin{align}\label{BC-sense-limit-case}
\text{ for any } h\in H^1_{\mathbf{M}_0}: \quad    \int_D h \mathcal{L}_{\kappa,\gamma} g \dfrac{dr}{\mathbf{M}_0}=\int_D \mathbf{M}_0(I+\mathcal{A}(r))\nabla_r(\dfrac{g}{\mathbf{M}_0})\cdot \nabla_r(\dfrac{h}{\mathbf{M}_0}) dr.
\end{align}
Now, notice that  there exists $C_{\kappa,\gamma}>0$  such that
\begin{align*}
  C_{\kappa,\gamma}(1 - \vert r\vert^{2})^{\frac{\kappa}{2\left(1 + \gamma\right)}} \leq   \left(
\dfrac{1 - \vert r\vert^{2}}{1 + \gamma \vert r\vert^{2}}
\right)^{\frac{\kappa}{2\left(1 + \gamma\right)}}
\leq (1 - \vert r\vert^{2})^{\frac{\kappa}{2\left(1 + \gamma\right)}}.
\end{align*}
Therefore the same conclusion of \autoref{rmq-BC-cases} holds. In particular,  introducing 
\begin{align*}
   \widetilde{\mathcal{H}}^2=\{ g\in  H^1_{\mathbf{M}_0} | \quad \int_D \big(\Div_r\bigl(\mathbf{M}_0 (I+\mathcal{A}(r))\nabla_r(\dfrac{g}{\mathbf{M}_0})\bigr))^2\dfrac{dr}{\mathbf{M}_0}<+\infty\},
\end{align*}

we have $D(\mathcal{L}_\kappa)=   \widetilde{\mathcal{H}}^2$ if $ \kappa\geq 2(1+\gamma)$ and  the strict inclusion $D(\mathcal{L}_{\kappa,\gamma})\subset    \widetilde{\mathcal{H}}^2$ if $ 0<\kappa< 2(1+\gamma)$.

The corresponding Lions Guelfand triple reads
\begin{align}\label{spaces-2}    
V_0:= L^2(\mathbb{T}^2;H^1_{\mathbf{M}_0})\hookrightarrow H_0=L^2(\mathbb{T}^2;L^2_{\mathbf{M}_0})\simeq  L^2(\mathbb{T}^2\times D, \dfrac{dr}{\mathbf{M}_0}dx) \hookrightarrow  V_0^\prime \text{ and } H\hookrightarrow  H_0.
\end{align}
\subsubsection{Weighted spaces, density and  trace}
Let us introduce the weighted Sobolev space $$\mathcal{H}^1_{\mathcal{J}}=\{g: \quad \int_D(\vert g\vert^2+\vert \nabla g\vert^2)\mathcal{J} dr<\infty\}.$$ We recall that  the following isometry $H^1_\mathcal{J} \to  \mathcal{H}^1_\mathcal{J}, \quad g\mapsto \dfrac{g}{\mathcal{J}}$ holds between the Hilbert spaces $H^1_\mathcal{J}$ and  $\mathcal{H}^1_\mathcal{J}.$
Let $\mathcal{J} \overset{|R| \to 1}{\sim} (1-|R|)^k, k>0$. We recall the following results.
\begin{itemize}
    \item  For any $k>0,$ we have $    \overline{C^1(\overline{D})}^{\mathcal{H}^1_\mathcal{J}}=\mathcal{H}^1_{\mathcal{J}}$. If $k\geq 1$, then     $\overline{C^\infty_c(D)}^{H^1_\mathcal{J}}=H^1_{\mathcal{J}}$, where $C^\infty_c(D)$ denotes the class of smooth functions with compact support in $D$,  see \cite[Prop. B.2]{Masmoudi2013}.
    \item If  $0<k<1,$ then  $\overline{C^\infty_c(D)}^{H^1_\mathcal{J}}\neq H^1_{\mathcal{J}}.$ In fact, functions of $\mathcal{H}^1_\mathcal{J}$ have a trace on the boundary $\partial D$. Namely the trace map $\gamma:\mathcal{H}^1_\mathcal{J} \to H^{\frac{1-k}{2}}(\partial D) $ is surjective.
\end{itemize}

Finally, we recall some useful technical lemmas regarding our weighted spaces (see \cite[Subsection 2.3.]{Chupin2010fokker} and \cite[Remark 3.7]{Masmoudi2008}).
\begin{lemma} \label{lemma-properties-weight}
Let $\kappa> 2$ and $d_{\partial D}(x) := 1-|x|^2$ denotes the distance from $\partial D$.  We have:
    \begin{enumerate}
       \item For any $\varphi\in H^1_\mathbf{M}, $ we have the following Hardy-type inequality, there exists a constant $C_H$ such that
     \begin{align}\label{Hardy-ineq}
          \int_D \dfrac{1}{d_{\partial D}^2}\left\vert\dfrac{\varphi}{\mathbf{M}}\right\vert^2 \mathbf{M}dr \le C_H \Vert \varphi\Vert^2_{H^1_\mathbf{M}}.
       \end{align}
        \item The following inclusions hold: $L^2_\mathbf{M} \subset L^2(D)$ and $H^1_\mathbf{M}\subset H^1_0(D).$
     \end{enumerate}
\end{lemma}
\begin{remark}\label{rem:hardy-constant}
    By inspecting the proof of \cite[Lemma 2.3]{Chupin2010fokker}, it turns out that for $\M \propto d_{\partial D}^\alpha$, the Hardy constant $C_H(\alpha)$, is uniformly bounded for $\alpha$ varying in a compact sets of $(0, +\infty)$, thus we will not specify this dependence from now on. 
\end{remark}
\begin{remark}
It is worth recalling that similar result of \autoref{lemma-properties-weight} holds with the weight $\mathbf{M}_0$ instead of $\mathbf{M}$ if $\kappa> 2(1+\gamma).$ 
\end{remark} 
\subsection{Definition of the noise coefficients}\label{assumption_noise-2D}
Let $(\Omega,\mathcal{F},(\mathcal{F}_t)_t,P)$  be  a   complete   filtered probability space and let us also consider a family $(W_t^k)_t^{k\in    \mathbb{Z}^2}$ of independent Brownian motions on the probability space $(\Omega,\mathcal{F},P)$ and a finite subset of indices $K \subset \mathbb{Z}^2$.  We set $$u_s(x,t)=\sum_{k\in K}%
\sigma_{k}^{N,\tau}\left(  x\right)  dW_{t}^{k}.$$

We now define the noise coefficients $\sigma_{k}^{N,\tau}$: they will depend upon two parameters, $\tau$, the dominant time-scale of the fluid flow and $N^{-1}$ its dominant spatial scale. 
Consider $\mathbb{Z}_{0}^{2}:=\mathbb{Z}^2-\{(0,0)\}$ divided into its four quadrants (write
$k=\left(  k_{1},k_{2}\right)  $)%
\begin{align*}
K_{++}  & =\left\{  k\in\mathbb{Z}_{0}^{2}:k_{1}\geq0,k_{2}>0\right\};\quad
K_{-+}   =\left\{  k\in\mathbb{Z}_{0}^{2}:k_{1}<0,k_{2}\geq0\right\}  \\
K_{--}  & =\left\{  k\in\mathbb{Z}_{0}^{2}:k_{1}\leq0,k_{2}<0\right\} ; \quad
K_{+-}   =\left\{  k\in\mathbb{Z}_{0}^{2}:k_{1}>0,k_{2}\leq0\right\}
\end{align*}
and set $  K_{+}   =K_{++}\cup K_{+-};  \quad 
K_{-}   =K_{-+}\cup K_{--}  \text{  and }   K=K_{+}\cup K_{-}.$
    Define \footnote{For $y=(y_1,y_2)\in   \R^2,$  $y^\perp$   stands  for $(-y_2,y_1).$}
\begin{align}\label{Def-sigma-k}
\sigma_{k}^{N,\tau}\left(  x\right)     =\theta_{k}^{N,\tau}\frac{k^{\perp}}{\left\vert
k\right\vert }\cos k\cdot x,\qquad k\in K_{+},  \quad 
\sigma_{k}^{N,\tau}\left(  x\right)     =\theta_{k}^{N,\tau}\frac{k^{\perp}}{\left\vert
k\right\vert }\sin k\cdot x,\qquad k\in K_{-}%
\end{align}
where
\begin{align}\label{scaling}
\theta^{N,\tau}_{k}&=\dfrac{a_\tau}{\left\vert k\right\vert^2 },  \qquad   N \leq \left\vert k\right\vert \leq  2N,\qquad N\in \mathbb{N}^* ; \quad
\theta^{N,\tau}_{k}=0 \qquad \text{elsewhere.}
\end{align}
where $a_\tau=\sqrt{\dfrac{\lambda}{\tau}\frac{8}{\pi \log(2)}}$ with $\lambda$ positive constant measuring the intensity of the turbulence, see  \textit{e.g.} \cite[Subsection 2.3.1]{tahraoui2025small}. In the following, we use
 the next result from \cite[Lemma 3]{Flandoli-Tahraoui}.
 \begin{lemma}\label[lemma]{lem-matrix-r} The following equalities    hold.
		\begin{align*}
			\dfrac{1}{2}	\operatorname{div}_x\left(\sum_{k\in K}(\sigma_k^N \otimes \sigma_k^N) \nabla_xf\right)&=
			\alpha_N	\Delta_xf
		\end{align*}
      where $\alpha_N:=\dfrac{1}{2}\sum_{k\in K_{++}}(\theta^{N,\tau}_k)^2$ and $\alpha_N \rightarrow 0$ as $N\rightarrow \infty$. Moreover, 
        \begin{align}\label{limit-A-N}
		A^N(r):=	\sum_{k\in K}\left((\nabla_x \sigma_k^Nr) \otimes (\nabla_x \sigma_k^Nr)\right) &= A(r)+\dfrac{1}{\tau}O(N^{-1})P^N(r),	\end{align}
		such that   $P^N$  is quadratic in $r$ and with coefficients uniformly bounded in $N$ and for $ r\in \mathbb{R}^2$
        \begin{align}\label{matrix-limit}
        A(r)
		=k_T (3 \left\vert r\right\vert^2I-2r\otimes r) , \qquad
		k_T=\dfrac{\lambda}{\tau}.      
        \end{align}
	\end{lemma}
      \begin{remark}\label{rmq-bound_A^N}
            For every $N>1$ due to the choice of the coefficients $\sigma_k^N$ (see \eqref{Def-sigma-k}), there exists a constant $C_A$ independent of $N, r\in B(0, 1)$ and $x\in \T^2$ such that for every $v\in \R^2$ it holds
            $$|A^N(r)v| \le C_A|v|.$$
              \end{remark}

\section{Main results}\label{Section-main-results} 
As mentioned in the introduction, due to technical reasons, in order to treat the system with any choice of parameters $\beta, \kappa, \alpha$, we introduce a cut-off on the noise. We will see that this cut-off plays no crucial role in a special range of parameters (\textit{c.f} \autoref{TH1-*_nocutoff}). For $\varepsilon>0$, let $\varphi_\varepsilon\in C^\infty(D)$ such that
\[\exists C>0: \vert \nabla \varphi_\varepsilon \vert \leq \dfrac{C}{\varepsilon}, \quad
0\le\varphi_\varepsilon\le 1,
\qquad
\varphi_\varepsilon(r)=
\begin{cases}
1, & 1-\vert r\vert\ge 2\varepsilon,\\
0, & 1-\vert r\vert\le \varepsilon.
\end{cases}
\]  Let $\varepsilon>0$ and  notice that  $\varphi_\varepsilon \to 1.$ 
Now,  we consider the following problem
\begin{align}\label{Ito-FP_eps-V2}
		\begin{cases}&df^{\varepsilon}_N=	\dfrac{1}{\beta}\Div_r \big(\k F(r) rf^{\varepsilon}_N+	\nabla_rf^{\varepsilon}_N\big) +\dfrac{1}{2}\Div_r \big( \varphi_\varepsilon^2(r) A^N(r) \nabla_rf^{\varepsilon}_N\big)dt +\alpha_N\Delta_x f^{\varepsilon}_Ndt\\[0.15cm]
            &-\sum_{k\in K}\sigma_k^N.\nabla_xf^{\varepsilon}_NdW^k -\sum_{k\in K}\varphi_\varepsilon(r)(\nabla_x \sigma_k^Nr).\nabla_rf^{\varepsilon}_N dW^k  ;\quad (x,r)\in \mathbb{T}^2\times D,\\[0.15cm]
            & \left( \kappa F(r)r f^{\varepsilon}_N + \nabla_r f^{\varepsilon}_N \right)\cdot \eta \vert_{\partial D\times \mathbb{T}^2} = 0,
            \\[0.15cm]
            			&f^{\varepsilon}_N|_{t=0}=f_0.
	\end{cases}\end{align}
   \begin{remark}
      We wish to draw the reader’s attention to the fact that  the boundary condition reads 
    \begin{align*}
        \left( \kappa F(r)r f^{\varepsilon}_N + \nabla_r f^{\varepsilon}_N \right)\cdot \eta \vert_{\partial D\times \mathbb{T}^2}=\left( \kappa F(r)r f^{\varepsilon}_N + \nabla_r f^{\varepsilon}_N+\varphi_\varepsilon^2(r) A^N(r) \nabla_rf^{\varepsilon}_N \right)\cdot \eta \vert_{\partial D\times \mathbb{T}^2}=0,
    \end{align*}
    since $\varphi_\varepsilon\equiv 0$ on $\partial D.$  \end{remark}
\begin{remark}
    In the following \autoref{Section-exit-uniq}, \autoref{Section-diffusion-limit} and \autoref{section-remove-cut-off}, it is clear that $f^\varepsilon_N:=f^\varepsilon_{N,\tau}$ and $f_\varepsilon$ depend on $\tau$ and $\beta$ also, but we don't dwell on the dependency until \autoref{Sec-singular-limit}.
\end{remark}

We  gather here the main results of this work. 
    \subsection{Existence and uniqueness of quasi-regular weak solution} 
    Before giving the main definition, we must recall the following result.
Let  $\mathcal{G}_t$ be  the filtration associated   with     $(W_t^k)_t^{k\in    \mathbb{Z}^2_0}$   namely $
\mathcal{G}_t=\sigma\{W_s^k;  s\in[0,t],  k\in    \mathbb{Z}^2_0\},$
and   denote  by  $\overline{\mathcal{G}}_t$    its completed   filtration\footnote{We assume that $\mathcal{F}_0$ contains all the $P$-null subset of    $\Omega$.}. For $T>0,$  let us  introduce
\begin{align*}
	\mathcal{H}&=L^2(\Omega,\overline{\mathcal{G}}_T,P),   \quad     M_n=\{k\in  \mathbb{Z}^2_0;\quad  \vert   k\vert  \leq    n\}\\
	G&=\displaystyle\bigcup_{n\in\mathbb{N}}G_n;\quad  G_n=\{g=(g_k)_{k\in M_n};    g_k\in  L^2(0,T);\quad    \forall k\in    M_n\}.  \end{align*}
For   $n\in   \mathbb{N}, g\in    G_n$,   we  set
\begin{align*}
	e_g(t)&=\exp{\big(\sum_{k\in M_n}\int_0^tg_k(s)dW^k(s)-\dfrac{1}{2}\sum_{k\in    M_n}\int_0^t\vert   g_k(s)\vert^2   ds\big)}, \text{  for }   t\in[0,T];\\
	\mathcal{D}&=\{e_g(T);    \quad   g\in    G\}.
\end{align*}
Based  on  the Wiener chaos decomposition, we  have (see \cite[Ch.   1]{Nualart2006}) 
\begin{proposition}
    $\mathcal{D}$ is  dense   in  $\mathcal{H}$.
\end{proposition}  

    \begin{definition}(Quasi-regular weak solutions)\label[definition]{Def-sol-QR-*}
Let $\varepsilon>0,$		    $f_0\in H$    and   $N\in   \mathbb{N}$.   We  say that  $f^\varepsilon_N$  is   quasi-regular weak solution   to  \eqref{Ito-FP_eps-V2}  if  
		$f^\varepsilon_N$    is $(\mathcal{F}_t)_t$-adapted and
		\begin{enumerate}
                \item $f^\varepsilon_N \in L^2(\Omega ;L^2(0, T;V))\cap L^2_{w-*}(\Omega ;L^\infty(0, T;H)) \cap L^2(\Omega;C_w([0,T];H))$.
        \item P-a.s. for any $t\in [0,T]$, it holds  
    	\begin{align*}
                    (f^\varepsilon_N(t), \phi )_H 
        &=(f_0, \phi )_H-\dfrac{1}{\beta}\int_0^t\left(f^\varepsilon_N , \phi \right)_Vds+\alpha_N\int_0^t( f^\varepsilon_N,\Delta_x \phi)_H ds\\&+\sum_{k\in K}\int_0^t\big(f^\varepsilon_N(s),( \sigma_k^N.\nabla_x\phi +(\nabla \sigma_k^Nr).\nabla_r (\varphi_\varepsilon\phi)+(\nabla \sigma_k^Nr).\dfrac{\kappa r}{1-\vert r\vert^2}\varphi_\varepsilon\phi)\big)_HdW^k(s)\notag\\
        & - \dfrac{1}{2}\int_0^t\left(\varphi_\varepsilon^2A^N(r)\nabla_rf^\varepsilon_N,\nabla_r\left(\dfrac{\phi}{\mathbf{M}}\right)\right) ds, \quad \forall \phi\in \mathbb{X}= L^2_{\mathbf{M}} (D;H^2(\mathbb{T}^2)) \cap V. \notag\end{align*}           
			\item (Regularity in Mean) For all    $n\in   \mathbb{N^*}$   and each function   $g\in   G_n$,   the deterministic   function    $V^\varepsilon_N(t,x,r)=\E[f^\varepsilon_N(t,x,r)e_g(t)]$ is  a   measurable  function,   which   belongs to  $ L^\infty(0, T;H)\cap    C_w([0, T];H)\cap L^2(0, T;V)$    and  for any $t\in[0,T],$    it  holds    (see    \Cref{prop-exisence-mean}) 
			\begin{align*}
				&( V^\varepsilon_N(t),\phi)_H=(f_0,\phi)_H+\int_0^t \left(V^\varepsilon_N(s), h_n\cdot \nabla_x \phi + y_n \cdot\nabla_r (\phi\varphi_\varepsilon)+y_n \cdot \dfrac{\kappa r}{1-\vert r\vert^2} (\phi\varphi_\varepsilon) \right)_Hds \\
				&        -\dfrac{1}{\beta}\int_0^t(V_N^\varepsilon, \phi)_V ds+\alpha_N\int_0^t( V^\varepsilon_N,\Delta_x \phi)_H ds- \dfrac{1}{2}\int_0^t\big(\varphi_\varepsilon^2 A^N(r) \nabla_rV^\varepsilon_N,\nabla_r\left(\dfrac{\phi}{\mathbf{M}}\right)\big) ds, \quad \forall \phi\in \mathbb{X},
			\end{align*}
       where   $K_n=\{k\in K: \min(n,N) \leq    \vert   k\vert  \leq    \max(2N,n)\}$ and 
			\begin{align*}
			h_n:=\sum_{k\in K_n}g_k\sigma_k^N   \text{  and } y_n:=\sum_{k\in K_n}g_k(\nabla \sigma_k^N r). 
			\end{align*}
		\end{enumerate}
	\end{definition}
    \begin{remark}
    We wish to draw the reader’s attention to the fact that non-flux boundary conditions in \eqref{Ito-FP_eps-V2} are implicitly encoded in the weak formulation of point (2).
\end{remark}
    In the following we will always use the weight $\M_0$ defined by \eqref{weight-limit} with $\gamma = k_T \beta /2$.
\begin{remark}\label{rmk:def_sol_nocutoff}
    For the system without cut-off, the notion of a solution is almost identical replacing all $\varphi_\varepsilon$ with the constant function $1$ and replacing the weights $\M$ with $\M_0$ and the weighted spaces $H$ and $V$ with $H_0$ and $V_0$ accordingly. In particular, the weak formulation reads:
    \begin{align*}
                    (f_N(t), \phi )_H 
        &=(f_0, \phi )_H-\dfrac{1}{\beta}\int_0^t\left(f_N , \phi \right)_{V_0}ds+\alpha_N\int_0^t( f_N,\Delta_x \phi)_{H_0} ds\\&+\sum_{k\in K}\int_0^t\big(f_N(s),( \sigma_k^N.\nabla_x\phi +(\nabla \sigma_k^Nr).\nabla_r \phi+(\nabla \sigma_k^Nr).\nabla \M_0\phi)\big)_{H_0}dW^k(s)\notag\\
        & - \dfrac{1}{2}\int_0^t\left((A^N-A)(r)\nabla_rf^\varepsilon_N,\nabla_r\left(\dfrac{\phi}{\mathbf{M}_0}\right)\right) ds, \quad \forall \phi\in \mathbb{X}= L^2_{\mathbf{M}_0} (D;H^2(\mathbb{T}^2)) \cap V_0. \notag\end{align*}
\end{remark}
 Let $\varepsilon>0, N\in \mathbb{M},$    similarly to the proof of  \cite[Theorem 7]{Flandoli-Tahraoui}  and  \cite[Theorem 8]{Flandoli-Tahraoui} (see also \autoref{sect-quai-regu-uni}),  we have
\begin{theorem}\label{TH1-*}\textbf{(Existence with cut-off.)} There exists  at    least   one     solution $f^\varepsilon_N$ to  \eqref{Ito-FP_eps-V2} in the sense of  \Cref{Def-sol-QR-*}. Moreover,  it holds 
\begin{align}\label{estimate-solution}
 \E\Big[\sup_{v\in[0,t]} \Vert f^\varepsilon_N(v)\Vert_H^2  \Big]&+\dfrac{4}{\beta}\E\int_0^t [f^\varepsilon_N(s)]_V^2ds\leq 2\Vert f_{0}\Vert_H^2\exp{(4\mathcal{K}_\varepsilon t)} \text{ for all } t\in [0,T],
\end{align}
where $\mathcal{K}_\varepsilon>0$ independent of $N.$
Moreover, for any two quasi-regular weak   solutions $f_1, f_2,$   of  \eqref{Ito-FP_eps-V2}  with    the same    initial data    $f_0$,  if   $(f_i(t),\varphi)$    is  $\overline{\mathcal{G}}_t$-adapted, for both    $i=1,2$ and for    any $\varphi\in V$, then   $f_1=f_2.$
\end{theorem}
    \begin{proof}
    For the proof of \autoref{TH1-*},   see \autoref{Section-exit-uniq}.
    \end{proof}

    \begin{theorem}\label{TH1-*_nocutoff}\textbf{(Existence without cut-off.)} Assume $\kappa /(2+ k_T \beta) >1$. There exists $\eta>0$ such that if $k_T\beta \le \eta$ then there exists  at    least   one     solution $f_N$ to  \eqref{Ito-FP_eps-V2} with $\varphi_\varepsilon\equiv 1$ in the sense of  \autoref{rmk:def_sol_nocutoff}. Moreover,  it holds 
\begin{align}\label{estimate-solution}
 \E\Big[\sup_{v\in[0,t]} \Vert f_N(v)\Vert_{H_0}^2  \Big]&+\E\int_0^t [f_N(s)]_{V_0}^2ds\leq 2\Vert f_{0}\Vert_{H_0}^2\exp{(4\mathcal{K} t)} \text{ for all } t\in [0,T],
\end{align}
where $\mathcal{K}>0$ independent of $N$ but depends on $\k$. 
Moreover, for any two quasi-regular weak   solutions $f_1, f_2,$   of  \eqref{Ito-FP_eps-V2}  with    the same    initial data    $f_0$,  if   $(f_i(t),\varphi)$    is  $\overline{\mathcal{G}}_t$-adapted, for both    $i=1,2$ and for    any $\varphi\in V_0$, then   $f_1=f_2.$
\end{theorem}
\begin{remark}
    Under the assumptions of \autoref{TH1-*_nocutoff}, we obtain also that the solution $f^{\varepsilon}_N $ provided by \autoref{TH1-*} additionally belongs to $L^2(\Omega; L^2(0, T; V_0)) \cap L^2_{w-*}(\Omega; L^\infty(0, T; H_0))$ uniformly in $\varepsilon>0$, satisfies the same weak formulation with $\M$ replaced by $\M_0$ as well as  the function spaces and test functions changed accordingly.  Moreover, $f^{\varepsilon}_N $ converges weakly to $f_N$ from \autoref{TH1-*_nocutoff} in the same space as $\varepsilon\rightarrow 0$. 
\end{remark}
     \subsection{Stochastic scaling limit, removing the cut-off and singular limit results}
The following result concerns  the scaling diffusion limit in the presence of cut-off of the stretching.
      \begin{theorem}\label{THm-scal-eps}(Stochastic scaling limit)
   Let  $f_0\in H.$       For any $\varepsilon>0,$ there exists a unique $f_\varepsilon$ such that
    \begin{itemize}
        \item a.s.  $f_\varepsilon \in L^2(0, T;V) \cap C([0, T];H),$
        \item   $f_\varepsilon$ satisfies:  $f_\varepsilon(0)=f_0$ and a.s. for any $\Phi \in V$ it holds
        \begin{align}\label{limit-PDE-eps-finalV1}
    \langle\dfrac{\partial f_\varepsilon}{\partial t},\Phi\rangle_{V^\prime,V}=-\int_{\mathbb{T}^2}\int_D\big ( \dfrac{\k}{\beta} f_\varepsilon F(r)r+\dfrac{1}{\beta}\nabla_rf_\varepsilon+ \dfrac{1}{2}\varphi_\varepsilon^2 A(r) \nabla_rf_\varepsilon\big)\cdot \nabla_r\left(\dfrac{\Phi}{\mathbf{M}}\right) drdx.
\end{align}
    \end{itemize}
      Moreover, the following convergence holds, as $N\to +\infty$
    \begin{align}\label{cv-N-eps-timeV1}
        f^\varepsilon_N(t) \rightharpoonup f_\varepsilon(t) \text{ in } L^2(\Omega,H) \text{ for any } t\in [0,T],
    \end{align}
  where $(f^\varepsilon_N)_{N}$ is unique quasi-regular weak solution (in the sense of \autoref{Def-sol-QR-*}) to       \eqref{Ito-FP_eps-V2}.
    \end{theorem}
    \begin{proof}
        See \autoref{Section-diffusion-limit}.
    \end{proof}
The next statement concerns the removing of the cut-off.
        \begin{theorem}\label{THm-eps-nocut}(Removing of the cut-off)
Let $f_0\in H_0.$        There exists  a unique $f$ such that
    \begin{itemize}
        \item a.s.  $f \in L^2(0, T;V_0) \cap C([0, T];H_0),$
        \item   $f$ satisfies:  $f(0)=f_0$ and a.s. for any $\Phi \in V_0$ it holds
        \begin{align}\label{limit-PDE-final-NOCUT-V1}
    \langle\partial_t f,\Phi\rangle_{(V_0)^\prime,V_0}=-\int_{\mathbb{T}^2}\int_D\bigl(\mathbf{M}_0 (\dfrac{1}{\beta}I+\dfrac{1}{2}A(r))\nabla_r(\dfrac{f}{\mathbf{M}_0})\bigr)\cdot \nabla_r\left(\dfrac{\Phi}{\mathbf{M}_0}\right) drdx.
\end{align}
    \end{itemize}
    Moreover, the following convergence hold, as $\varepsilon\to 0$
    \begin{align}\label{cv-eps-*V1}
        f_\varepsilon(t) \rightharpoonup f(t) \text{ in } H_0 \text{ for any } t\in [0,T],
    \end{align}
    where $(f_\varepsilon)_{\varepsilon>0}$ is the unique solution to  \eqref{limit-PDE-eps-finalV1} in the sense of \autoref{THm-scal-eps}.
    \end{theorem}
    \begin{proof}
        See \autoref{section-remove-cut-off}.
    \end{proof}
 {\begin{remark}
     Under the assumptions of  \autoref{TH1-*_nocutoff},  we can prove directly that, as  
  $N\to +\infty$    \begin{align}\label{cv-N--no-cut}
        f_N(t) \rightharpoonup f(t) \text{ in } L^2(\Omega,H_0) \text{ for any } t\in [0,T],
    \end{align}
    where $f$ is the unique solution to \eqref{limit-PDE-final-NOCUT-V1} and $f_N$ is the unique solution to  \eqref{Ito-FP_eps-V2} with $\varphi_\varepsilon\equiv 1$ in the sense of  \autoref{rmk:def_sol_nocutoff}. The proof is completely analogous, and in fact easier than that of \autoref{THm-scal-eps}, therefore we omit it. 
 \end{remark}}
  Let $\tau>0$    and consider the regime where the relaxation time of polymers and the dominant time-scale of the small scale turbulent flow satisfies $\beta=\zeta \tau, \zeta>0$. Now, set
  \begin{align}\label{limit-matrix-main}
   \mathbf{M}_0(r)=\dfrac{1}{\mathbf{Z}}\left(
\dfrac{1 - \vert r\vert^{2}}{1 + \alpha \vert r\vert^{2}}
\right)^{\frac{\kappa}{2\left(1 + \alpha \right)}}, \quad \mathcal{A}(r)= \alpha(3 \left\vert r\right\vert^2I-2r\otimes r),  \quad \alpha=\frac{\zeta\lambda}{2}
\end{align} where $\mathbf{Z}$ satisfies $\int_D \mathbf{M}_0 dr=1.$
 Based on \autoref{THm-eps-nocut}, there exists  a unique $f_\tau$ such that   $f_\tau \in L^2(0, T;V_0) \cap C([0, T];H_0)$    and  solves
        \begin{align}\label{limit-PDE-final-NOCUT-V1-main}
    \langle\partial_t f_\tau,\Phi\rangle_{(V_0)^\prime,V_0}=- \dfrac{1}{\zeta \tau}\int_{\mathbb{T}^2}\int_D\bigl(\mathbf{M}_0 (I+\mathcal{A}(r))\nabla_r(\dfrac{f_\tau}{\mathbf{M}_0})\bigr)\cdot \nabla_r\left(\dfrac{\Phi}{\mathbf{M}_0}\right) drdx,
\end{align}
supplemented with initial data  $f_\tau(0)=f_{\tau,0},$ we refer to \autoref{Sec-singular-limit} for more details. The following result concerns the singular limit as $\tau\downarrow 0.$
 \begin{theorem}\label{Thm-singular-limit-main}(Singular limit)
    Assume that $f_{\tau, 0} \to f_0$ in $H_0$ as $\tau \downarrow 0.$ Then the following  holds.
    \begin{itemize}
        \item $ f_\tau$ converges to 
        $ \rho_0\otimes \M_0$ in $L^2(0,T;H_0)$, 
     where  $\rho_0(x)=\int_D f_0(x,r) dr$ and $f_\tau$ is the unique solution  to \eqref{limit-PDE-final-NOCUT-V1-main}. Moreover, the following estimate (convergence rate) holds
              \begin{align}\label{ineq-converg-equi-Main}
 \int_0^T  \Vert f_{\tau}(t)-\rho_0\otimes \mathbf{M}_0\Vert_{H_0}^2 dt \leq C_P  \xi\tau (\Vert f_{\tau, 0}\Vert_{H_0}^2+\Vert \rho_0 \Vert^2_{L^2(\mathbb{T}^2)})+2T\Vert f_{\tau,0}-f_0\Vert_{H_0}^2.
          \end{align}
          \item Assume that $f_{\tau,0}$ converges to $ \rho_{0}\otimes\mathbf{M}_0$ in $H_0.$ Then,  $ f_\tau$ converges to 
        $ \rho_0\otimes \M_0$ in $C([0,T];H_0)$, 
     where  $\rho_0(x)=\int_D f_0(x,r) dr$. Moreover, the following estimate (convergence rate) holds
     \begin{align}\label{ineq-conv-equi-Cont-Main}
         \displaystyle\sup_{t\in [0,T]}\Vert f_{\tau}(t)-\rho_0\otimes \mathbf{M}_0\Vert_{H_0}^2 \leq   \Vert f_{\tau,0}- \rho_{0}\otimes\mathbf{M}_0\Vert_{H_0}^2.
     \end{align}
    \end{itemize}
\end{theorem}
\begin{proof}
    See \autoref{Sec-singular-limit}.
\end{proof}
Recall that $H\hookrightarrow H_0.$
The following theorem is a consequence of \autoref{THm-scal-eps},  \autoref{THm-eps-nocut} and \autoref{Thm-singular-limit-main}.
 \begin{theorem} 
Let $\tau>0, \varepsilon>0$ and	 $N\in   \mathbb{N}$ and  set  $\beta=\zeta \tau, \zeta>0$	   in  \eqref{Ito-FP_eps-V2}.  Let $f_{\tau,0}\in H$ and  $f^\varepsilon_{\tau,N}$  be the   quasi-regular weak solution   to  \eqref{Ito-FP_eps-V2}  in the sense of \autoref{Def-sol-QR-*} with the initial data $f_{\tau,0}.$
The following holds.
\begin{itemize}
    \item   Assume there exists $f_0\in H_0$ such that   $f_{\tau, 0} \to f_0$ in $H_0$ as $\tau \downarrow 0.$ Then, we have
    \begin{align*}
    \displaystyle \text{weak}-\lim_{\tau\downarrow 0}\lim_{\varepsilon\downarrow 0}\lim_{N\uparrow +\infty}   f^\varepsilon_{\tau,N}= \rho_{0}\otimes\mathbf{M}_0 \text{ in } L^2(0,T;L^2(\Omega;H_0)), \quad \rho_0(x)=\int_D f_0(x,r) dr.
    \end{align*}
 \item     Assume that $f_{\tau,0}$ converges to $ \rho_{0}\otimes\mathbf{M}_0$ in $H_0$ as $\tau \downarrow 0.$ Then,  we have
      \begin{align*}
    \displaystyle \text{weak}-\lim_{\tau\downarrow 0}\lim_{\varepsilon\downarrow 0}\lim_{N\uparrow +\infty}   f^\varepsilon_{\tau,N}(t)= \rho_{0}\otimes\mathbf{M}_0 \text{ in } L^2(\Omega;H_0) \text{ for any } t\in [0,T].
    \end{align*}
\end{itemize} 
    \end{theorem}

  {\begin{remark}
  Let $\tau>0$,	 $N\in   \mathbb{N}$ and  set  $\beta=\zeta \tau, \zeta>0$	   in  \eqref{Ito-FP_eps-V2}.           Under the assumptions of  \autoref{TH1-*_nocutoff}, let   $f_{\tau,N}$ be the unique quasi-regular weak solution to  \eqref{Ito-FP_eps-V2} with $\varphi_\varepsilon\equiv 1$ in the sense of  \autoref{rmk:def_sol_nocutoff}, with the initial data $f_{\tau,0}.$
The following holds.
\begin{itemize}
    \item   Assume there exists $f_0\in H_0$ such that   $f_{\tau, 0} \to f_0$ in $H_0$ as $\tau \downarrow 0.$ Then, we have
    \begin{align*}
    \displaystyle \text{weak}-\lim_{\tau\downarrow 0}\lim_{N\uparrow +\infty}   f_{\tau,N}= \rho_{0}\otimes\mathbf{M}_0 \text{ in } L^2(0,T;L^2(\Omega;H_0)), \quad \rho_0(x)=\int_D f_0(x,r) dr.
    \end{align*}
 \item     Assume that $f_{\tau,0}$ converges to $ \rho_{0}\otimes\mathbf{M}_0$ in $H_0$ as $\tau \downarrow 0.$ Then,  we have
      \begin{align*}
    \displaystyle \text{weak}-\lim_{\tau\downarrow 0}\lim_{N\uparrow +\infty}   f_{\tau,N}(t)= \rho_{0}\otimes\mathbf{M}_0 \text{ in } L^2(\Omega;H_0) \text{ for any } t\in [0,T].
    \end{align*}
\end{itemize} 
  \end{remark}}
\begin{remark}(The 3D case)
   It is worth drawing the reader's attention to the fact that with cosmetic changes to what was made here, it is possible to prove a similar result in a three-dimensional setting.  However, we preferred to present the result in 2D because we could present the stochastic turbulent velocity explicitly, see \autoref{assumption_noise-2D}. The main difference between 2D and 3D will be in the form of the stochastic turbulent velocity of the small scales,   see \cite[Subsection 7.4]{Flandoli-Tahraoui} for the structure of the noise in 3D (equation (73)). We emphasize that the matrix obtained, after the scaling limit, as $N\to +\infty$, has a shape similar to that of the matrix  $A$ given by \eqref{matrix-limit}. More precisely, a similar line of reasoning can be used in the 3D case to obtain the following matrix $\mathbf{A}$ given by 
   $$\mathbf{A}(r)= \dfrac{\lambda}{\tau}\left(2\lvert r\rvert^2 I-r\otimes r\right), \quad r\in \mathbb{R}^3.$$ Consequently, the final density $\mathbf{M}_0$  given by \eqref{limit-matrix-main} is the same with 
   $\widetilde{\alpha}=\dfrac{\zeta\lambda}{2}$ instead of $\alpha$.
 
\end{remark}

 \section{Existence and uniqueness of  quasi-regular weak solution (\autoref{TH1-*})}\label{Section-exit-uniq}
In this section, we demonstrate the existence and uniqueness within the class of \textit{quasi-regular weak solution} to \eqref{Ito-FP_eps-V2}, proving \autoref{TH1-*}. We first demonstrate the existence of weak solutions by constructing a solution via a finite-dimensional approximation (Galerkin approximation with  suitable basis). Then, we  derive some uniform estimates in some  Sobolev spaces with appropriate weight. Next, we prove the existence and uniqueness of the solution to the associated mean equation, which ultimately allows us to demonstrate uniqueness within the class of quasi-regular weak solutions.
\subsection{Existence of weak solutions}

In the following, let $\varepsilon>0$ and $N\in \mathbb{N}^*.$

\subsubsection{Galerkin approximation}
 Let $\{k_l\}_{l\in \mathbb{N}}$ be an orthonormal basis of $L^2(\mathbb{T}^2)$ (\textit{e.g.}  the  basis constructed \textit{via} the eigenfunctions of the Laplacian  on $\mathbb{T}^2$ where $\{k_l\}_{l\in \mathbb{N}} \subset C^\infty(\mathbb{T}^2)$  ). Thus, we define the Hilbert basis  of $H$ by $\{w_{i,l}=e_i\otimes k_l\}_{(i,l)\in \mathbb{N}\times \mathbb{N}},$
where $(e_i)$ are given by \eqref{basis-eigen}. 	To simplify the notation, the duality between	$V$	and	$V^\prime$	will	be	denoted	$\langle	\cdot,\cdot\rangle$	instead	of		$\langle	\cdot,\cdot\rangle_{V^\prime,V}$.	Thus,	we	recall	the	following	equality (Lions-Guelfand triple)
	\begin{align}\label{Lions-Gelfand-triple}
		\langle	f,u\rangle=(f,u)_H,	\quad	\forall	f\in	L^2(\T^2; L^2_\mathbf{M}),	\quad \forall	u\in	V.
	\end{align} 

	Now,	let $m\in   \mathbb{N}^*$ and
	denote	by	$H_m=\text{span}\{w_{i,l}:\quad i,l\leq m\}$	and	the		projection operator	$P_m$	defined from	$V^\prime$	to	$H_m$		by
	$P_m:V^\prime\to	H_m;\quad	u\mapsto	P_mu=\displaystyle\sum_{i,l=1}^m\langle	u,w_{i,l}\rangle_{V^\prime,V}	w_{i,l}.	$
	In	particular,	the	restriction	of	$P_m$	to	$H$,	denoted	by	the	same	way,	is	the	$(\cdot,\cdot)$-orthogonal	projection	from	$H$	to	$H_m$	and	given	by
	\begin{align}\label{def-Projection}
		P_m:H\to	H_m;\quad	u\mapsto	P_mu=\displaystyle\sum_{i+l=1}^m 
		(	u,w_{i,l})_H	w_{i,l}.		\end{align}
	
	We 	notice	that	$\Vert	P_mu\Vert_H\leq\Vert	u\Vert_H$, $\forall u\in H$,
	then 	$\Vert	P_m\Vert_{L(H,H)}\leq	1$.
			Let us introduce
    \begin{align}\label{approximation-seq}
         f_m(t)=\displaystyle\sum_{i,l=1}^mg^{i,l}_m(t)w_{i,l} \text{ and set }
	f_m(0)=P_mf_0\in H_m. 
    \end{align}
 		 and 	consider	the	following	finite	dimensional	SDE
\begin{align}\label{approx-linear*SDE}
    &(f_m(t),w_{i,l})_H-(f_{0},w_{i,l})_H=\int_0^t(\mathcal{Y}^mf_m,w_{i,l})_Hds, \quad 1\leq i,l \leq m \\
    &	-\int_0^t\dfrac{1}{\beta}\big(\mathbf{M}\nabla_r(\dfrac{f_m}{\mathbf{M}}),\nabla_r(\dfrac{w_{i,l}}{\mathbf{M}})\big)ds -\alpha_N\int_0^t(\nabla_x f_m,\nabla_x w_{i,l})_Hds- \dfrac{1}{2}\int_0^t\big(\varphi_\varepsilon^2A^N(r) \nabla_rf_m,\nabla_r(\dfrac{w_{i,l}}{\mathbf{M}})\big)ds\notag\\
            &-\sum_{k\in K}\int_0^t(\sigma_k^N.\nabla_xf_mdW^k,w_{i,l})_H -\sum_{k\in K}\int_0^t(\varphi_\varepsilon(\nabla \sigma_k^Nr).\nabla_rf_m dW^k,w_{i,l})_H, \notag
\end{align}
where $
    (\mathcal{Y}^mf_m,w_{i,l})_H=\sum_{k\in K}\big(\sigma_k^N.\nabla_x P_m[\varphi_\varepsilon (\nabla \sigma_k^Nr).\nabla_rf_m],w_{i,l}\big)_H.$
	For every $m\in \mathbb{N}^*$, 	the above system \eqref{approx-linear*SDE}   is   a   linear system  of  SDE, by  classical result    (see    e.g.    \cite[Chapter   V]{Oksendal2013}),   we  get the existence and uniqueness of $\mathcal{F}_t$-adapted   solution    $f_m\in C([0,T],L^2(\Omega;H_m))$.
\begin{lemma}\label{sec:estimates}
The solution $f_m$ of \eqref{approx-linear*SDE} satisfies
    \begin{align}\label{key-est-Galerkin}
 \E\Big[\sup_{v\in[0,t]} \Vert f_m(v)\Vert_H^2  \Big]&+\dfrac{4}{\beta}\E\int_0^t [f_m(s)]_V^2ds\leq 2\Vert f_{0}\Vert_H^2\exp{(4\mathcal{K}_\varepsilon t)} \text{ for all } t\in [0,T].
\end{align}
where $\mathcal{K}_\varepsilon=\dfrac{2C_B^2\mathbf{C}(\kappa+1)^2 a_\tau^2+\kappa (5+\kappa)C_A}{4\varepsilon^2}.$
\end{lemma}
\begin{proof} 
By applying It\^o formula and \eqref{approx-linear*SDE}, one gets
  \begin{align*}
   \dfrac{1}{2} \Vert f_m(t)\Vert_H^2&-\dfrac{1}{2}\Vert P_mf_{0}\Vert_H^2=\int_0^t(\mathcal{Y}^mf_m,f_m)_Hds  \notag\\
    &	-\int_0^t\dfrac{1}{\beta}\big(\mathbf{M}\nabla_r(\dfrac{f_m}{\mathbf{M}}),\nabla_r(\dfrac{f_m}{\mathbf{M}})\big)ds -\alpha_N\int_0^t(\nabla_x f_m,\nabla_x f_m)_Hds \notag\\&- \dfrac{1}{2}\int_0^t\big(\sum_{k\in K}\left(\varphi_\varepsilon^2(\nabla \sigma_k^Nr) \otimes (\nabla \sigma_k^Nr)\right) \nabla_rf_m,\nabla_r(\dfrac{f_m}{\mathbf{M}})\big)ds\\
            &-\sum_{k\in K}\int_0^t(\sigma_k^N.\nabla_xf_mdW^k,f_m)_H -\sum_{k\in K}\int_0^t(\varphi_\varepsilon(\nabla \sigma_k^Nr).\nabla_rf_m dW^k,f_m)_H \\
            &+\dfrac{1}{2}[\sum_{i,l=1}^m\sum_{k\in K}\int_0^t(\sigma_k^N.\nabla_xf_m,w_{i,l})_H^2 ds +\int_0^t(\varphi_\varepsilon(\nabla \sigma_k^Nr).\nabla_rf_m,w_{i,l})_H^2 ds\\
            &+2\sum_{k\in K}\int_0^t\big(\sigma_k^N.\nabla_xf_m, \sum_{i,l=1}^m(\varphi_\varepsilon(\nabla \sigma_k^Nr).\nabla_rf_m,w_{i,l})_Hw_{i,l}\big)_H ds].
\end{align*}
Note that 
\begin{align*}
    &\sum_{k\in K}\int_0^t\big(\sigma_k^N.\nabla_xf_m, \sum_{i,l=1}^m(\varphi_\varepsilon(\nabla \sigma_k^Nr).\nabla_rf_m,w_{i,l})_Hw_{i,l}\big)_H ds\\&=\sum_{k\in K}\int_0^t\big(\sigma_k^N.\nabla_xf_m, P_m(\varphi_\varepsilon(\nabla \sigma_k^Nr).\nabla_rf_m)\big)_H ds=-\int_0^t(\mathcal{Y}^mf_m,f_m)_Hds.
\end{align*}
Moreover, we have 
\begin{align*}
    &\dfrac{1}{2}\sum_{i,l=1}^m\sum_{k\in K}\int_0^t(\sigma_k^N.\nabla_xf_m,w_{i,l})_H^2 ds-\alpha_N\int_0^t(\nabla_x f_m,\nabla_x f_m)_Hds\\&=\dfrac{1}{2}\sum_{k\in K}\int_0^t\Vert P_m\sigma_k^N.\nabla_xf_m\Vert_H^2 ds-\alpha_N\int_0^t(\nabla_x f_m,\nabla_x f_m)_Hds\\
    &\leq  \dfrac{1}{2}\sum_{k\in K}\int_0^t\Vert \sigma_k^N.\nabla_xf_m\Vert_H^2 ds-\alpha_N\int_0^t(\nabla_x f_m,\nabla_x f_m)_Hds\leq 0.
\end{align*}
Since $\Div \sigma_k^N=0,$ we get $(\sigma_k^N.\nabla_xf_m,f_m)_H=0$ P-a.s. Therefore
\begin{align}
   \dfrac{1}{2} \Vert f_m(t)\Vert_H^2&+\dfrac{1}{\beta}\int_0^t\big(\mathbf{M}\nabla_r(\dfrac{f_m}{\mathbf{M}}),\nabla_r(\dfrac{f_m}{\mathbf{M}})\big)ds-\dfrac{1}{2}\Vert P_mf_{0}\Vert_H^2\le\notag\\
  &- \dfrac{1}{2}\int_0^t\big(\sum_{k\in K}\left(\varphi_\varepsilon^2(\nabla \sigma_k^Nr) \otimes (\nabla \sigma_k^Nr)\right) \nabla_rf_m,\nabla_r(\dfrac{f_m}{\mathbf{M}})\big)ds\label{ineq-estm-Galer}\\
            & -\sum_{k\in K}\int_0^t(\varphi_\varepsilon(\nabla \sigma_k^Nr).\nabla_rf_m dW^k(s),f_m)_H+ \sum_{k\in K}\sum_{j,k \le m}\dfrac{1}{2}\int_0^t(\varphi_\varepsilon(\nabla \sigma_k^Nr).\nabla_rf_m,w_{i,l})_H^2 ds.\notag
\end{align}
Notice that 
\begin{align}
    &- \dfrac{1}{2}\int_0^t\big(\sum_{k\in K}\left(\varphi_\varepsilon^2(\nabla \sigma_k^Nr) \otimes (\nabla \sigma_k^Nr)\right) \nabla_rf_m,\nabla_r(\dfrac{f_m}{\mathbf{M}})\big)ds+ \sum_{k\in K}\sum_{j,k \le m}\dfrac{1}{2}\int_0^t(\varphi_\varepsilon(\nabla \sigma_k^Nr).\nabla_rf_m,w_{i,l})_H^2 ds \notag\\
         &=- \dfrac{1}{2}\int_0^t\big(\sum_{k\in K}\Vert \varphi_\varepsilon (\nabla \sigma_k^Nr)\cdot\nabla_rf_m\Vert_H^2ds+ \sum_{k\in K}\dfrac{1}{2}\int_0^t\Vert P_m\varphi_\varepsilon(\nabla \sigma_k^Nr).\nabla_rf_m\Vert_H^2 ds\notag\\&-\dfrac{1}{4}\int_0^t\big(\sum_{k\in K}\left(\varphi_\varepsilon^2(\nabla \sigma_k^Nr) \otimes (\nabla \sigma_k^Nr)\right) \nabla_rf_m^2,\nabla(\dfrac{1}{\mathbf{M}})\big)ds\notag\\
        &\leq  -\dfrac{1}{4}\int_0^t\big(\sum_{k\in K}\left(\varphi_\varepsilon^2(\nabla \sigma_k^Nr) \otimes (\nabla \sigma_k^Nr)\right) \nabla_rf_m^2,\nabla(\dfrac{1}{\mathbf{M}})\big)ds\notag\\
        &\leq \dfrac{1}{4}\int_0^t\big(f_m^2, \sum_{k\in K} \Div_r\Big[\left(\varphi_\varepsilon^2(\nabla \sigma_k^Nr) \otimes (\nabla \sigma_k^Nr)\right)\nabla(\dfrac{1}{\mathbf{M}})\Big]\big)ds \leq  \dfrac{\kappa (5+\kappa)C_A}{4\varepsilon^2}\int_0^t\Vert f_m(s)\Vert_H^2 ds. \label{esti-cov-uniq}
\end{align}
Indeed,  by using \Cref{rmq-bound_A^N},  $\varphi_\varepsilon\dfrac{1}{1-\vert r\vert^2} \leq \dfrac{1}{\varepsilon}$  and  
\begin{align*} 
    &\int_0^t\big(f_m^2, \sum_{k\in K} \Div_r\Big[\left(\varphi_\varepsilon^2(\nabla \sigma_k^Nr) \otimes (\nabla \sigma_k^Nr)\right)\nabla(\dfrac{1}{\mathbf{M}})\Big]\big)ds\\&=2\int_0^t\big(f_m^2, \sum_{k\in K} \left(\varphi_\varepsilon \nabla \varphi_\varepsilon(\nabla \sigma_k^Nr) \otimes (\nabla \sigma_k^Nr)\right)\nabla(\dfrac{1}{\mathbf{M}})\big)ds\\
    &+ \int_0^t\big(\varphi_\varepsilon^2 f_m^2, \sum_{k\in K} \left((\nabla \sigma_k^Nr) \otimes (\nabla \sigma_k^Nr)\right):\mathcal{D}^2(\dfrac{1}{\mathbf{M}})\big)ds,
\end{align*}
where $\mathcal{D}^2$ denotes the Hessian matrix given by 
\begin{align*}
    \nabla(\dfrac{1}{\mathbf{M}})=\dfrac{\kappa r}{1-\vert r\vert^2}\dfrac{1}{\mathbf{M}} \text{ and }  \mathcal{D}^2(\dfrac{1}{\mathbf{M}})=\dfrac{\kappa(1-\vert r\vert^2)I+\kappa(2+\kappa)r\otimes r}{(1-\vert r\vert^2)^2}\dfrac{1}{\mathbf{M}}.
\end{align*}
Finally, note that
\begin{align*}
   &- \sum_{k\in K}\int_0^t(\varphi_\varepsilon(\nabla \sigma_k^Nr).\nabla_rf_m ,f_m)_H dW^k(s)=\dfrac{1}{2}\sum_{k\in K}\int_0^t(f_m^2 ,(\nabla \sigma_k^Nr)\cdot\nabla(\varphi_\varepsilon \dfrac{1}{\mathbf{M}})) dW^k(s)\\
   &=\dfrac{1}{2}\sum_{k\in K}\int_0^t(\dfrac{f_m^2}{\mathbf{M}} ,(\nabla \sigma_k^Nr)\cdot\nabla\varphi_\varepsilon) dW^k(s)+\dfrac{1}{2}\sum_{k\in K}\int_0^t(\dfrac{f_m^2}{\mathbf{M}}, \varphi_\varepsilon(\nabla \sigma_k^Nr)\cdot \dfrac{\kappa r}{1-\vert r\vert^2} ) dW^k(s).
\end{align*}
We have
\begin{align*}
    & \sum_{k\in K} \vert(\nabla \sigma_k^Nr)\cdot\nabla\varphi_\varepsilon\vert^2 \leq \dfrac{1}{\varepsilon^2} \sum_{k\in K}\vert (\nabla \sigma_k^Nr)\vert^2
     \leq \dfrac{a_\tau^2}{\varepsilon^2} \sum_{k\in K} \dfrac{1}{\vert k\vert^2}\leq  \dfrac{ \mathbf{C}a_\tau^2}{\varepsilon^2}, \\
     & \sum_{k\in K} \vert\varphi_\varepsilon(\nabla \sigma_k^Nr)\cdot \dfrac{\kappa r}{1-\vert r\vert^2}\vert^2 \leq \dfrac{\kappa^2}{\varepsilon^2} \sum_{k\in K}\vert (\nabla \sigma_k^Nr)\vert^2
     \leq \dfrac{\kappa^2 a_\tau^2}{\varepsilon^2} \sum_{k\in K} \dfrac{1}{\vert k\vert^2}\leq  \dfrac{\mathbf{C}\kappa^2 a_\tau^2}{\varepsilon^2}
\end{align*}
since $\sum_{k\in K} \dfrac{1}{\vert k\vert^2}\leq \mathbf{C}$ and $\mathbf{C}$ is independent of $N$. Next, Burkholder-Davis-Gundy inequality ensures the existence of $C_B>0$ such that 
\begin{align*}
  &\E \sup_{v\in[0,t]}   \vert \sum_{k\in K}\int_0^v(\varphi_\varepsilon(\nabla \sigma_k^Nr).\nabla_rf_m ,f_m)_H dW^k(s)\vert \\
  &\leq  \dfrac{C_B}{2}\E\Big[\big(\sum_{k\in K}\int_0^t(\dfrac{f_m^2}{\mathbf{M}} ,(\nabla \sigma_k^Nr)\cdot\nabla\varphi_\varepsilon)^2 ds\big)^{\frac{1}{2}} \Big]+\dfrac{C_B}{2}\E\Big[(\sum_{k\in K}\int_0^t(\dfrac{f_m^2}{\mathbf{M}}, \varphi_\varepsilon(\nabla \sigma_k^Nr)\cdot \dfrac{\kappa r}{1-\vert r\vert^2} )^2 ds\big)^{\frac{1}{2}} \Big]\\
  &\leq  \dfrac{C_B\sqrt{\mathbf{C}}(\kappa+1) a_\tau}{2\varepsilon}\E\Big[(\int_0^t(\dfrac{f_m^2}{\mathbf{M}}, 1)^2 ds\big)^{\frac{1}{2}} \Big] \\ 
  &\leq \dfrac{C_B\sqrt{\mathbf{C}}(\kappa+1) a_\tau}{2\varepsilon}\E\Big[(\int_0^t\Vert f_m(s)\Vert_H^4 ds\big)^{\frac{1}{2}} \Big]\\
  &\leq  \dfrac{C_B\sqrt{\mathbf{C}}(\kappa+1) a_\tau}{2\varepsilon}\E\Big[\sup_{v\in[0,t]} \Vert f_m(v)\Vert_H (\int_0^t\Vert f_m(s)\Vert_H^2 ds\big)^{\frac{1}{2}} \Big].
\end{align*} 
Young's inequality ensures
\begin{align*}
     &\E \sup_{v\in[0,t]}   \vert \sum_{k\in K}\int_0^v(\varphi_\varepsilon(\nabla \sigma_k^Nr).\nabla_rf_m ,f_m)_H dW^k(s)\vert \\&\leq  \dfrac{C_B^2\mathbf{C}(\kappa+1)^2 a_\tau^2}{2\varepsilon^2}\E\Big[ \int_0^t\Vert f_m(s)\Vert_H^2 ds \Big]+ \dfrac{1}{4}\E\Big[\sup_{v\in[0,t]} \Vert f_m(v)\Vert_H^2  \Big].
\end{align*}
Therefore, from \eqref{ineq-estm-Galer}, we infer for any $t\in [0,T]$ \begin{align*}
 \dfrac{1}{4}\E\Big[\sup_{v\in[0,t]} \Vert f_m(v)\Vert_H^2  \Big]&+\dfrac{1}{\beta}\E\int_0^t [f_m(s)]_V^2ds\leq \dfrac{1}{2}\Vert f_{0}\Vert_H^2
 +\mathcal{K}_\varepsilon\E\Big[ \int_0^t\Vert f_m(s)\Vert_H^2 ds \Big]
\end{align*}
where $\mathcal{K}_\varepsilon=\dfrac{2C_B^2\mathbf{C}(\kappa+1)^2 a_\tau^2+\kappa (5+\kappa)C_A}{4\varepsilon^2}.$   Grönwall's lemma then ensures \eqref{key-est-Galerkin}. 
\end{proof}

We now pass to the limit $m\rightarrow 0$ constructing a weak solution satisfying (1) and (2) in \autoref{Def-sol-QR-*}.
	\subsubsection*{$1^{st} step$}
Consider the following space
    $$\mathbb{Y}:=L^2(\Omega ;L^2(0, T;L^2(\mathbb{T}^2,H^1(D;\mathbf{M}dr))) \cap  L^2(\Omega ;L^\infty(0, T;L^2(\mathbb{T}^2,L^2(D;\mathbf{M}dr)).$$
	 There exists $g_m \in \mathbb{Y}$ such that   $\mathbf{M}g_m=f_m.$ By  using   \eqref{key-est-Galerkin}, $(g_m)_m$ is bounded in $\mathbb{Y}.$ By using  Banach–Alaoglu theorem    and  diagonal extraction argument,
     there exists $$\widetilde{g} \in L^2(\Omega ;L^2(0, T;L^2(\mathbb{T}^2,H^1(D;\mathbf{M}dr)))) \cap  L^2_{w-*}(\Omega ;L^\infty(0, T;L^2(D;\mathbf{M}dr)))$$  
 such that the following   convergences  (up to a subsequence denoted by the same way) hold
	\begin{align}
		g_m \rightharpoonup \widetilde{g} &\text{		in	}  L^2(\Omega ;L^2(0, T;L^2(\mathbb{T}^2,H^1(D;\mathbf{M}dr)))),\label{cv-appro-3-g}\\
		g_m \overset{*}{\rightharpoonup} \widetilde{g} &\text{		in	}  L^2_{w-*}(\Omega ;L^\infty(0, T;L^2(\mathbb{T}^2,L^2(D;\mathbf{M}dr)))).\label{cv-appro-4-g}
	\end{align}
     By the definition of the spaces $H$ and $V$,  it holds   (up to a subsequence) 
	\begin{align}
		f_m \rightharpoonup  \mathbf{M}\widetilde{g}=\widetilde{f} &\text{		in	}  L^2(\Omega ;L^2(0, T;V)), \text{ and }
		f_m \overset{*}{\rightharpoonup} \mathbf{M}\widetilde{g}=\widetilde{f} &\text{		in	}  L^2_{w-*}(\Omega ;L^\infty(0, T;H)).\label{cv-appro-4}
	\end{align} Let  $\phi \in L^2_\mathbf{M}(D;H^1(\mathbb{T}^2)) \cap V$, note that    $(f_m(t),\phi)_H$	is adapted with    respect to  $(\mathcal{F}_t)_t$    and recall  that     the space of 
	adapted processes is a closed convex    subspace of $L^2 (\Omega \times [0,T])$, hence weakly
	closed. Therefore  $(\widetilde{f}(t),\phi)_H$ is also adapted  and  its    It\^o   integral    is  well    defined and  bounded. 
    Next, due to  linearity and 
	\eqref{cv-appro-4}, we infer that 
	\begin{align*}
		\sum_{k\in K}&\int_0^t\big(f_m(s),( \sigma_k^N.\nabla_x\phi +(\nabla \sigma_k^Nr).\nabla_r (\varphi_\varepsilon\phi)+(\nabla \sigma_k^Nr).\dfrac{\kappa r}{1-\vert r\vert^2}\varphi_\varepsilon\phi)\big)_HdW^k(s) \\&\rightharpoonup \sum_{k\in K}\int_0^t\big(\widetilde{f}(s),( \sigma_k^N.\nabla_x\phi +(\nabla \sigma_k^Nr).\nabla_r (\varphi_\varepsilon\phi)+(\nabla \sigma_k^Nr).\dfrac{\kappa r}{1-\vert r\vert^2}\varphi_\varepsilon\phi)\big)_HdW^k(s) \text{  in } L^2(\Omega\times [0,T]).\notag
\end{align*}
	\subsubsection*{$2^{nd} step$}
    	Let $\phi\in H_m$ and  $t\in [0,T]$, let us  set 
	\begin{align*}
		B_m(t) :=(f_m(t), \phi )_H  -\sum_{k\in K}\int_0^t\big(f_m(s),( \sigma_k^N.\nabla_x\phi +(\nabla \sigma_k^Nr).\nabla_r (\varphi_\varepsilon\phi)+(\nabla \sigma_k^Nr).\dfrac{\kappa r}{1-\vert r\vert^2}\varphi_\varepsilon\phi)\big)_HdW^k(s).
	\end{align*}
	From    \eqref{approx-linear*SDE},   we  write   (in distributional sense    with    respect to  $t$)
	\begin{align*}&\dfrac{d}{dt}B_m=-\sum_{k\in K}\big(\nabla_rf_m, \varphi_\varepsilon(\nabla \sigma_k^Nr)P_m\sigma_k^N.\nabla_x w_{i,l}\big)_H\\&\quad-\dfrac{1}{\beta}\big(\mathbf{M}\nabla_r(\dfrac{f_m}{\mathbf{M}}),\nabla_r(\dfrac{w_{i,l}}{\mathbf{M}})\big)+\alpha_N( f_m,\Delta_x w_{i,l})_H - \dfrac{1}{2}\big(\sum_{k\in K}\left(\varphi_\varepsilon^2(\nabla \sigma_k^Nr) \otimes (\nabla \sigma_k^Nr)\right) \nabla_rf_m,\nabla_r(\dfrac{w_{i,l}}{\mathbf{M}})\big).\notag\end{align*}
	Let $A\in \mathcal{F}$ and $\xi\in  C^\infty_c(]0,T[)$, by multiplying the    last    equation    by    $\mathbb{I}_A\xi$ and integrating  over $\Omega \times [0,T]$ we derive 
    	\begin{align}&-\int_A\int_0^T\bigl[B_m\dot\xi \bigr]dsdP=-\int_A\int_0^T\sum_{k\in K}\big(\nabla_rf_m, \varphi_\varepsilon(\nabla \sigma_k^Nr)P_m\sigma_k^N.\nabla_x w_{i,l}\big)_H \xi dsdP\notag\\&\quad-\dfrac{1}{\beta}\int_A\int_0^T\big(\mathbf{M}\nabla_r(\dfrac{f_m}{\mathbf{M}}),\nabla_r(\dfrac{w_{i,l}}{\mathbf{M}})\big) \xi dsdP+\alpha_N\int_A\int_0^T( f_m,\Delta_x w_{i,l})_H\xi dsdP\notag\\& - \dfrac{1}{2}\int_A\int_0^T\big(\sum_{k\in K}\left(\varphi_\varepsilon^2(\nabla \sigma_k^Nr) \otimes (\nabla \sigma_k^Nr)\right) \nabla_rf_m,\nabla_r(\dfrac{w_{i,l}}{\mathbf{M}})\big)\xi dsdP.\label{eq-to-pass-base-elem}\end{align}
	Now,    let us  prove the following.
	\begin{align}\label{limit-covarnaince-0}\int_A\int_0^T\sum_{k\in K}\big(\nabla_rf_m, \varphi_\varepsilon(\nabla \sigma_k^Nr)P_m\sigma_k^N.\nabla_x w_{i,l}\big)_H \xi dsdP\to0 \text{    as  }   m\to    +\infty.
	\end{align}
	For $1\leq  i,l\leq   m$,   the following   convergence holds
	\begin{align*}
		\sum_{k\in K}  (\nabla\sigma_k^Nr)P_m[\sigma_k^N\cdot\nabla_xw_{i,l} ]    \to \sum_{k\in K}  (\nabla\sigma_k^Nr)\sigma_k^N\cdot\nabla_xw_{i,l}   \text{  in  }  H. 
	\end{align*}
	Indeed, since   $P_m$   is  an  orthogonal projection  on  $H$  we  have
	\begin{align*}
		\Vert \sum_{k\in K}  (\nabla\sigma_k^Nr)(P_m-I)[\sigma_k^N\cdot\nabla_xw_{i,l} ]  \Vert_H^2&\leq    
		\sum_{k\in K} (\theta^{N,\tau}_k)^2\vert   k\vert^2 \Vert\vert   r\vert  (P_m-I)[\sigma_k^N\cdot\nabla_xw_{i,l} ]  \Vert_H^2\\   
		&\leq  \sum_{k\in K} (\theta^{N,\tau}_k)^4\vert   k\vert^2 \Vert P_m-I\Vert_{L(H,H)}^2\Vert\nabla_x w_{i,l}   \Vert_H^2\\
		&\leq  C \Vert P_m-I\Vert_{L(H,H)}^2\Vert \nabla_xw_{i,l}   \Vert_H^2 \to 0,
	\end{align*}
	where $C>0$ \footnote{
		Recall that     $\displaystyle\sum_{k\in K} (\theta^{N,\tau}_k)^4\vert   k\vert^2   = \sum_{k\in K} \dfrac{1}{\vert   k\vert^6}  \leq    C$.}.  On  the other   hand,   $f_m$   converges   weakly  to  $\widetilde f$   in $ L^2(\Omega ;L^2(0, T;V))$. Therefore
	\begin{align*}
		\lim_m        &\int_A\int_0^T\sum_{k\in K}(\nabla_rf_m(s)(\nabla\sigma_k^Nr)P_m[\sigma_k^N\cdot\nabla_x\phi ])_H \xi(s) dsdP\\
		&=\sum_{k\in K}\int_A\int_0^T(   \nabla_r\widetilde f(s),(\nabla\sigma_k^Nr)\sigma_k^N\cdot\nabla_xw_{i,l} )_H\xi(s)   dsdP\\
		&=-\int_A\int_0^T(  \widetilde f(s),\sum_{k\in K}(\nabla\sigma_k^Nr)\cdot\nabla_r[\sigma_k^N\cdot\nabla_x \dfrac{w_{i,l}}{\mathbf{M}}])\xi(s)   dsdP=0\quad   \forall i\in    \mathbb{N}.
	\end{align*}
	Indeed, for any given function $\psi$ we have  
	\begin{align*}
		\displaystyle\sum_{k\in K}	(\nabla \sigma_k^Nr).\nabla_r(\sigma_k^N.\nabla_x\psi)=\sum_{k\in K} \sum_{l,\gamma,i=1}^2 \partial_{x_\gamma}\sigma_k^ir_\gamma\partial_{r_i}(\sigma_k^l\partial_{x_l} \psi)=	\sum_{l,\gamma,i=1}^2\partial_{x_\gamma} Q_{i,l}(0)\partial_{x_l}(r_\gamma\partial_{r_i}\psi).
	\end{align*}
	Since   the covariance  matrix  $Q$ satisfies   $Q(x)=Q(-x)$ then   $\partial_{x_\gamma} Q_{i,l}(0)=0$. As a result we get
\begin{align*}&-\int_A\int_0^T\bigl[B\dot\xi \bigr]dsdP=-\dfrac{1}{\beta}\int_A\int_0^T\big(\mathbf{M}\nabla_r(\dfrac{\widetilde{f}}{\mathbf{M}}),\nabla_r\left(\dfrac{\phi}{\mathbf{M}}\right)\big) \xi dsdP+\alpha_N\int_A\int_0^T( \widetilde{f},\Delta_x \phi)_H\xi dsdP\\& - \dfrac{1}{2}\int_A\int_0^T\big(\sum_{k\in K}\left(\varphi_\varepsilon^2(\nabla \sigma_k^Nr) \otimes (\nabla \sigma_k^Nr)\right) \nabla_r\widetilde{f},\nabla_r\left(\dfrac{\phi}{\mathbf{M}}\right)\big)\xi dsdP,\quad \forall \phi \in \mathbb{X},\notag\end{align*}
 where 
        \begin{align*}
            B(t) :=(\widetilde{f}(t), \phi )_H  -\sum_{k\in K}\int_0^t\big(\widetilde{f}(s),( \sigma_k^N.\nabla_x\phi +(\nabla \sigma_k^Nr).\nabla_r (\varphi_\varepsilon\phi)+(\nabla \sigma_k^Nr).\dfrac{\kappa r}{1-\vert r\vert^2}\varphi_\varepsilon\phi)\big)_HdW^k(s).
        \end{align*}

	Then, taking into account the regularity    of $\widetilde    f$ , we infer that the distributional derivative $\dfrac{dB}{dt}$ belongs to the space $L^2(\Omega\times[0,T]).$ Recalling that $B\in L^2(\Omega\times[0,T])$, we conclude that $ B\in   L^2(\Omega;C([0,T])$.
	Considering the properties of  It\^o's integral, we deduce
	$ (\widetilde    f(\cdot), \phi)_H \in L^2(\Omega;C([0,T]),$
	which  means   that   $\widetilde    f\in  L^2(\Omega;C([0,T];\mathbb{X}^\prime)$ and   therefore 
	$\widetilde    f\in  L^2(\Omega;C_w([0,T];H),$ thanks to \eqref{cv-appro-4}    and \cite[Lemma. 1.4 p. 263]{Temam77}. We   finish  the proof   by  showing some    continuous  convergence in  time.   Indeed,
	let  $\xi \in \mathcal{C}^\infty([0,t])$ for $t\in ]0,T]$ and note that the	following	integration	by	parts	formula	holds
	\begin{align}\label{IPPtimez-1}
		\int_0^t	\dfrac{dB}{ds}(s)\xi(s) ds&=-\int_0^tB(s)\dot\xi ds+B(t)\xi(t)-\int_{\mathbb{T}^2}\int_{\mathbb{R}^2}    f_0 \phi drdx\xi(0).
	\end{align}
	Now,    by  standard   arguments  (see \textit{e.g.}  \cite[proof    of  Prop. 3.]{vallet2019well} ) we  get
	for any $t\in ]0,T]$, $$    f_m(t)  \rightharpoonup \widetilde    f(t) \text{ in } L^2(\Omega,H), \text{ as } m\to \infty.$$
	and    $\widetilde f(0)=f_0$   in  $H$-sense. 	In  conclusion,     there   exists  a   solution $f^\varepsilon_N$  ($\widetilde  f=f^\varepsilon_N$  to  stress  the dependence  $N$,    since   we  will    pass    to  the limit   as  $N\to   +\infty $   in  \autoref{Section-diffusion-limit}) satisfying
    \begin{itemize}
        \item $\widetilde{f}$ is adapted with    respect to  $(\mathcal{F}_t)_t$ and  $$\widetilde{f} \in L^2(\Omega ;L^2(0, T;V))\cap L^2_{w-*}(\Omega ;L^\infty(0, T;H)) \cap L^2(\Omega;C_w([0,T];H).$$
        \item P-a.s. for any $t\in [0,T]$, it holds
    	\begin{align*}
                    (\widetilde{f}(t), \phi )_H 
        &=(f_0, \phi )_H+\sum_{k\in K}\int_0^t\big(\widetilde{f}(s),( \sigma_k^N.\nabla_x\phi +(\nabla \sigma_k^Nr).\nabla_r (\varphi_\varepsilon\phi)+(\nabla \sigma_k^Nr).\dfrac{\kappa r}{1-\vert r\vert^2}\varphi_\varepsilon\phi)\big)_HdW^k(s)\\
        &-\dfrac{1}{\beta}\int_0^t\big(\mathbf{M}\nabla_r(\dfrac{\widetilde{f}}{\mathbf{M}}),\nabla_r\left(\dfrac{\phi}{\mathbf{M}}\right)\big) ds+\alpha_N\int_0^t( \widetilde{f},\Delta_x \phi)_H ds\\& - \dfrac{1}{2}\int_0^t\big(\sum_{k\in K}\left(\varphi_\varepsilon^2(\nabla \sigma_k^Nr) \otimes (\nabla \sigma_k^Nr)\right) \nabla_r\widetilde{f},\nabla_r\left(\dfrac{\phi}{\mathbf{M}}\right)\big) ds, \quad \forall \phi\in \mathbb{X}.\end{align*}
    \end{itemize}
    \begin{remark}
It is worth mentioning that we first  pass to the limit in \eqref{eq-to-pass-base-elem} with fixed elements of the basis, then we obtain the validity of the limit for any elements in $H^1_\mathbf{M}\otimes H^2(\T^2)$  and finally the last equality  in $\mathbb{X}^\prime$ by another approximation argument.
         \end{remark}
         \subsection{Construction of weak solutions in the restricted parameters regime}\label{sec:ex_nocutoff}
         The aim of this subsection is to prove \autoref{TH1-*_nocutoff}, namely the existence of solution for the system without cut-off, assuming some constraints on the physical parameters. Since most of the steps are the same as in the previous section we only give a sketch proving only the main a-priori estimate. We begin proving a key coercivity estimate. 
         \begin{lemma}\label{lem:coercivity_est}
             Let $\kappa/(2+ \k_T \beta)> 1$, then there exists $\eta >0$ such that if $k_T\beta \le \eta$, there exists positive constants $C$, $\theta$ such that for every $f\in V_0$ it holds 
             \begin{align}\label{coercivity-est}
    \frac{1}{\beta}\int_{\T^2 \times D} \M_0 \left| \nabla_r \left(\frac{f}{\M_0}\right)\right|_{I+ \frac{\beta}{2}A}^2 drdx - \frac{1}{2}\int_{\T^2 \times D} A \nabla_r f \cdot \frac{\nabla_r f}{\M_0}drdx \ge -C \|f\|^2_{H_0} + \theta [f]^2_{V_0}
\end{align}
         \end{lemma}
\begin{remark}\label{rem:coercivity_fails}
Note that the above estimate trivially fails if the condition $\kappa/(2+ \k_T \beta)> 1$ is not satisfied, indeed by choosing $f=\M_0 \in V_0$, \eqref{coercivity-est}  reduces to 
$$ \frac{1}{2}\int_{\T^2 \times D} A \nabla_r \M_0 \cdot \frac{\nabla_r \M_0}{\M_0}drdx \le C'$$
and since $A$ is bounded and $\M_0 \sim (1-|r|^2)^{\frac{\kappa}{2+k_T \beta}}$ near $|r|=1$, we have that the LHS is integrable if and only if $\frac{\kappa}{2 + k_T \beta} > 1$. 
         \end{remark}
         \begin{proof}
             By expanding the first term in the LHS of \eqref{coercivity-est} we get 
\begin{align*}
     \frac{1}{\beta}\int_{\T^2 \times D} \M_0 \left| \nabla_r \left(\frac{f}{\M_0}\right)\right|_{I+ \frac{\beta}{2}A}^2 drdx  &= \frac{1}{\beta}\int_{\T^2 \times D} \M_0 \left| \nabla_r \left(\frac{f}{\M_0}\right)\right|^2 drdx \\ &+ \frac{1}{2}\int_{\T^2 \times D}  \left(A\nabla_r f - fA\frac{\nabla_r \M_0}{\M_0}\right)\cdot \left( \frac{\nabla_r f}{\M_0} - f\frac{\nabla_r \M_0}{\M_0^2}\right)drdx.
\end{align*}
Thus 
\begin{align}
     &\frac{1}{\beta}\int_{\T^2 \times D} \M_0 \left| \nabla_r \left(\frac{f}{\M_0}\right)\right|_{I+ \frac{\beta}{2}A}^2 drdx - \frac{1}{2}\int_{\T^2 \times D} A \nabla_r f \cdot \frac{\nabla_r f}{\M_0}drdx \notag\\ &= \frac{1}{\beta}\int_{\T^2 \times D} \M_0 \left| \nabla_r \left(\frac{f}{\M_0}\right)\right|^2 drdx \notag\\ &\quad- \int_{\T^2 \times D} \frac{f}{\M_0}\nabla_r f \cdot \frac{A\nabla_r\M_0}{\M_0} drdx + \frac{1}{2}\int_{\T^2 \times D} \frac{f^2}{\M_0}A\nabla_r\M_0\cdot \frac{\nabla_r\M_0 }{\M_0^2}drdx. \notag
\end{align}
By using $\nabla f/\M_0 = \nabla(f/\M_0) + f\nabla\M_0/\M_0^2 $  we get
\begin{align}
&\frac{1}{\beta}\int_{\T^2 \times D} \M_0 \left| \nabla_r \left(\frac{f}{\M_0}\right)\right|_{I+ \frac{\beta}{2}A}^2 drdx - \frac{1}{2}\int_{\T^2 \times D} A \nabla_r f \cdot \frac{\nabla_r f}{\M_0}drdx \notag\\ &= \frac{1}{\beta}\int_{\T^2 \times D} \M_0 \left| \nabla_r \left(\frac{f}{\M_0}\right)\right|^2 drdx \notag\\ &\quad - \int_{\T^2 \times D} f\nabla \left( \frac{f}{\M_0}\right)A\frac{\nabla \M_0}{\M_0}drdx - \frac{1}{2}\int_{\T^2 \times D} \frac{f^2}{\M_0}A\nabla\M_0\cdot \frac{\nabla\M_0 }{\M_0^2}drdx \label{coerc-to-est}.
\end{align}
Since $\M_0$ is radial and thanks to the form of $A$ (see \eqref{matrix-limit}), it holds $A\nabla\M_0 = k_T|r|^2 \nabla \M_0$. From now on, to lighten the notation we introduce $\gamma := \k_T \beta/2$. Next we observe that 
$$\frac{\nabla \M_0}{\M_0}= \frac{\k}{2(1+\gamma)}\left(\frac{1-|r|^2}{1+\gamma |r|^2}\right)^{-1} \left(\frac{-2r(1+\gamma)}{(1+\gamma|r|^2)^2}\right)=\frac{-\k r}{1+\gamma |r|^2}(1-|r|^2)^{-1}.$$
This means that near the boundary $|r|=1$, $|\frac{\nabla \M_0}{\M_0} |^2\sim \frac{\kappa^2}{(1+\gamma)^2} d_{\partial_D}^{-2}$. 
Let us focus on the last two terms. 
\begin{align*}
    \int_{\T^2 \times D} \frac{f^2}{\M_0}A\nabla\M_0\cdot \frac{\nabla\M_0 }{\M_0^2} drdx= \int_{\T^2 \times D} \frac{f^2}{\M_0}k_T|r|^2 \frac{|\nabla\M_0|^2 }{\M_0^2}drdx
\end{align*}
Thus, since in the interior of $D$ all the singular terms are bounded, choosing $\delta= \delta(\gamma)\in (0,1)$ such that  $$\left|\frac{\nabla \M_0}{\M_0} \right|^2\le 2\frac{\kappa^2}{(1+\gamma)^2} d_{\partial_D}^{-2} \qquad \text{ for } 1-|r|^2 \le \delta,$$
the last integral can be bounded by 
$$\int_{\T^2 \times D} \frac{f^2}{\M_0}k_T|r|^2 \frac{|\nabla\M_0|^2 }{\M_0^2}drdx \le  \frac{2k_T\k^2}{(1+\gamma)^2}\int_{\T^2\times \{1-|r|^2 \le \delta\}} \frac{1}{d^2_{{\partial_D}}}\frac{f^2}{\M_0}drdx + C_\delta\int_{\T^2 \times D} \frac{f^2}{\M_0} dxdr.$$

By (1) in \autoref{lemma-properties-weight}, which holds thanks to the assumption $\frac{\k}{2(1+\gamma)}>1$, there exists a constant $C_H$ such that 
$$\frac{2k_T\k^2}{(1+\gamma)^2}\int_{\T^2\times \{1-|r|^2 \le \delta\}} \frac{1}{d^2_{{\partial_D}}} \frac{f^2}{\M_0}drdx \le C_H \frac{2k_T\k^2}{(1+\gamma)^2} \int_{\T^2 \times D} \M_0 \left| \nabla\left(\frac{f}\M_0{}\right)\right|^2 drdx + C'\int_{\T^2 \times D} \frac{f^2}{\M_0}drdx.$$
Now we turn to the first term in \eqref{coerc-to-est}.  For any $\varepsilon>0,$ we have
\begin{align*}
    &\left|\int_{\T^2 \times D} f\nabla \left( \frac{f}{\M_0}\right)A\frac{\nabla \M_0}{\M_0}drdx \right| = \left| \int_{\T^2 \times D} f\nabla \left( \frac{f}{\M_0}\right)k_T|r|^2\frac{\nabla \M_0}{\M_0}drdx\right|\\
    &\quad\le k_T\left(\int_{\T^2 \times D} \M_0 \left|\nabla\left(\frac{f}{\M_0}\right)\right|^2drdx\right)^{1/2}\left(\int_{\T^2 \times D} \frac{f^2}{\M_0}\left|\frac{\nabla \M_0}{\M_0} \right|^2drdx\right)^{1/2}\\
    &\quad\le (\frac{\sqrt{2 C_H}k_T\k}{1+\gamma} + \varepsilon)\left(\int_{\T^2 \times D} \M_0 \left|\nabla\left(\frac{f}{\M_0}\right)\right|^2drdx\right) + C_\varepsilon\int_{\T^2 \times D} \frac{f^2}{\M_0}drdx.
\end{align*}

Where in the last step we have used the previous estimate and the Young inequality. 
Finally, putting all the terms together we arrive at 
\begin{align}
\frac{1}{\beta}\int_{\T^2 \times D} &\M_0 \left| \nabla_r \left(\frac{f}{\M_0}\right)\right|_{I+ \frac{\beta}{2}A}^2 drdx - \frac{1}{2}\int_{\T^2 \times D} A \nabla_r f \cdot \frac{\nabla_r f}{\M_0}drdx \notag\\
&{\ge \left(\frac{1}{\beta} - \frac{\sqrt{2} C_Hk_T \k}{1+\gamma} -\varepsilon -  \frac{C_Hk_T \k^2}{(1+\gamma)^2}\right) \left(\int_{\T^2 \times D} \M_0 \left|\nabla\left(\frac{f}{\M_0}\right)\right|^2drdx\right) - C\int_{\T^2 \times D} \frac{f^2}{\M_0}drdx}
\end{align}
Finally, recalling from \autoref{rem:hardy-constant} that $C_H$ can be chosen uniformly with respect to all the parameters, provided $ \k/{(2+k_T \beta)} > 1$, we get that if the parameters $\kappa, \beta, k_T$ satisfy 
\begin{equation}\label{restriction-param}
    \frac{\k}{2+k_T \beta} \ge 1, \qquad k_T\beta << 1 
\end{equation}
then {$ \frac{1}{\beta} - \frac{\sqrt{2} C_Hk_T \k}{1+\gamma} -\varepsilon -  \frac{ C_Hk_T \k^2}{(1+\gamma)^2}>0$.}
         \end{proof}
         \begin{proof}[Proof of \autoref{TH1-*_nocutoff} (Sketch)]
             By the same steps as in the proof of \autoref{sec:estimates}, constructing the Galerkin approximation with base functions defined  in terms of the spaces $H_0$ and $V_0$, we arrive at the analogous of \eqref{ineq-estm-Galer}
             \begin{align}
   \dfrac{1}{2} \Vert f_m(t)\Vert_{H_0}^2&+\dfrac{1}{\beta}\int_0^t\big(\mathbf{M_0}(I+ \frac{\beta}{2}A(r))\nabla_r(\dfrac{f_m}{\mathbf{M_0}}),\nabla_r(\dfrac{f_m}{\mathbf{M_0}})\big)ds-\dfrac{1}{2}\Vert P_mf_{0}\Vert_{H_0}^2\le\notag\\
  &- \dfrac{1}{2}\int_0^t\big(\left(A^N(r)- A(r)\right) \nabla_rf_m,\nabla_r(\dfrac{f_m}{\mathbf{M_0}})\big)ds\label{ineq-estm-Galer_nocutoff}\\
            & -\sum_{k\in K}\int_0^t((\nabla \sigma_k^Nr).\nabla_rf_m dW^k(s),f_m)_{H_0}+ \sum_{k\in K}\sum_{i,l \le m}\dfrac{1}{2}\int_0^t((\nabla \sigma_k^Nr).\nabla_rf_m,w_{i,l})_{H_0}^2 ds.\notag
\end{align}
Where we have used that for the weight $M_0$ it holds 
$$\frac{1}{\beta}\Div\left(\k rF(r)f + \nabla_r f + \frac{\beta}{2}A(r)\nabla_r f\right) = \frac{1}{\beta}\Div\left(\mathbf{M_0}(I+ \frac{\beta}{2}A(r))\nabla_r(\dfrac{f}{\mathbf{M_0}})\right).$$
Note that
\begin{align*} \sum_{k\in K}\sum_{i,l \le m}\dfrac{1}{2}\int_0^t((\nabla \sigma_k^Nr).\nabla_rf_m,w_{i,l})_{H_0}^2 ds 
&\leq  \frac{1}{2}\int_{\T^2 \times D} A^N(r) \nabla_r f_m \cdot \frac{\nabla_r f_m}{\M_0}drdx,
\end{align*}
and 
\begin{align*}
    \frac{1}{2}\int_{\T^2 \times D} A^N(r) \nabla_r f_m \cdot \frac{\nabla_r f_m}{\M_0}drdx&=\frac{1}{2}\int_{\T^2 \times D} A(r) \nabla_r f_m \cdot \frac{\nabla_r f_m}{\M_0}drdx\\&+\frac{1}{2}\int_{\T^2 \times D} (A^N(r)-A(r)) \nabla_r f_m \cdot \nabla_r(\dfrac{f_m}{\mathbf{M_0}})drdx\\
    &+\frac{1}{2}\int_{\T^2 \times D} (A^N(r)-A(r)) \nabla_r(\dfrac{f_m}{\mathbf{M_0}}) \cdot  \dfrac{\nabla_r\mathbf{M_0}}{\mathbf{M_0}}f_mdrdx\\
    &+\frac{1}{2}\int_{\T^2 \times D} (A^N(r)-A(r)) \nabla_r\mathbf{M_0} \cdot  \dfrac{\nabla_r\mathbf{M_0}}{\mathbf{M_0}^2}\dfrac{f_m^2}{\mathbf{M_0}}drdx.
\end{align*}
Now, thanks to the fact that $|A^N- A| =O(N^{-1})$, we bound $$\int_0^t\big(\left(A^N(r)- A(r)\right) \nabla_rf_m,\nabla_r(\dfrac{f_m}{\mathbf{M_0}})\big)ds \le CN^{-1} ([f_m]_{V_0}^2 + \|f_m\|_{H_0}^2)$$ by similar computations as in the proof of \autoref{lem:coercivity_est}. Thanks to  \autoref{lem:coercivity_est}, for $N$ sufficiently large, we have a constant $\beta^{-1}>\theta>0$ for which 
\begin{align}
   \dfrac{1}{2} \Vert f_m(t)\Vert_{H_0}^2&+\theta\int_0^t[f_m]_{V_0}ds-\dfrac{1}{2}\Vert P_mf_{0}\Vert_{H_0}^2\le\notag\\
            & -\sum_{k\in K}\int_0^t((\nabla \sigma_k^Nr).\nabla_rf_m dW^k(s),f_m)_{H_0} + C\int_0^t\|f_m\|_{H_0}^2 ds \notag
\end{align}
Thanks to BDG, we estimate 
\begin{align*}\mathbb{E}\left[\sup_{[0, T]} \left| -\sum_{k\in K}\int_0^t((\nabla \sigma_k^Nr).\nabla_rf_m dW^k(s),f_m)_{H_0} \right|\right] &\lesssim \mathbb{E}\left[ \left| -\sum_{k\in K}\int_0^T((\nabla \sigma_k^Nr).\nabla_rf_m ,f_m)_{H_0}^2ds \right|^{1/2}\right]\\
& \lesssim \mathbb{E}\left[\left | \int_0^T(A^N \nabla f_m, { \frac{ \nabla f_m}{\M_0}) \|f_m\|^2_{H_0}}\right|^{1/2}\right].
\end{align*}
Again by the same steps as in the proof of \autoref{lem:coercivity_est}, we can bound $$\dfrac{1}{2}\int_0^t\big(A^N(r) \nabla_rf_m,\dfrac{\nabla_rf_m}{\mathbf{M_0}}\big)ds \le C(k_T, \beta)([f_m]_{V_0}^2 + \|f_m\|_{H_0}^2),$$ where $C(k_T, \beta)$ is some constant (which is allowed to change line by line in the following computations) with the property that $\beta C(k_T, \beta) \rightarrow 0$ as $k_T\beta \rightarrow 0$ (in particular $C(k_T, \beta) \le \theta$ for $k_T \beta$ sufficiently small), thus 
\begin{align*}&\mathbb{E}\left[\sup_{[0, T]} \left| \sum_{k\in K}\int_0^t((\nabla \sigma_k^Nr).\nabla_rf_m dW^k(s),f_m)_{H_0} \right|\right]\\ &\le C(k_T, \beta)\mathbb{E}\left[\left | \int_0^T([f_m]_{V_0}^2 + \|f_m\|_{H_0}^2) \|f_m\|^2_{H_0}\right|^{1/2}\right] \\
&\le \mathbb{E}\left[ \frac{1}{4}\sup_{[0, T]}\|f_m\|^2_{H_0} + C(k_T, \beta)\left | \int_0^T([f_m]_{V_0}^2 + \|f_m\|_{H_0}^2) \right|\right]
\end{align*}
Putting all together, for $k_T\beta$ sufficiently small, we arrive at 
\begin{align*}
    \mathbb{E}\left[\sup_{[0, T]}\|f_m\|^2_{H_0} +  \int_0^T[f_m]_{V_0}^2  dt\right] \le  \mathbb{E}\left[ \|P_mf_0\|^2_{H_0} + C \int_0^T\|f_m\|_{H_0}^2 dt\right].
\end{align*}
From which we obtain by Gr\"onwall's  lemma $f_m \in L^2(\Omega; L^2(0, T; V_0))\cap L^2_{w-*}(\Omega; L^\infty (0, T; H_0))$ uniformly in $m>0,$ which allows to conclude the proof with analogous steps as in the proof of \autoref{TH1-*}. 
         \end{proof}
    \subsection{Quasi-regular weak solution and uniqueness}\label{sect-quai-regu-uni}
        In order to complete the proof of \autoref{TH1-*} are left to  prove  (3) in \Cref{Def-sol-QR-*}   and the  uniqueness  of ``\textit{quasi-regular   weak    solution}''    to  \eqref{Ito-FP_eps-V2},   we  need    first   to  prove   the existence   and uniqueness  of  solution   $V^\varepsilon_N$ to  an  appropriate mean    equation.   Namely, we  prove   the following. 
	\begin{proposition}\label[proposition]{prop-exisence-mean}
		For   any $t\in[0,T],$   there   exists  $V^\varepsilon_N(t)=\E[  f^\varepsilon_N(t)e_g(t)]$  such    that
		\begin{enumerate}
			\item $V^\varepsilon_N \in  L^\infty(0, T;H)\cap     L^2(0, T;V)$  and $V^\varepsilon_N\in  C_w([0,T];H)$.
			\item   For any $t\in[0,T],$    it  holds
			\begin{align*}
				&( V^\varepsilon_N(t),\phi)_H=(f_0,\phi)_H+\int_0^t \left(V^\varepsilon_N(s), h_n\cdot \nabla_x \phi + y_n \cdot\nabla_r (\phi\varphi_\varepsilon)+y_n \cdot \dfrac{\kappa r}{1-\vert r\vert^2} (\phi\varphi_\varepsilon) \right)_Hds \\
				&        -\dfrac{1}{\beta}\int_0^t\big(\mathbf{M}\nabla_r(\dfrac{V^\varepsilon_N}{\mathbf{M}}),\nabla_r\left(\dfrac{\phi}{\mathbf{M}}\right)\big) ds+\alpha_N\int_0^t( V^\varepsilon_N,\Delta_x \phi)_H ds\\& - \dfrac{1}{2}\int_0^t\big(\sum_{k\in K}\left(\varphi_\varepsilon^2(\nabla \sigma_k^Nr) \otimes (\nabla \sigma_k^Nr)\right) \nabla_rV^\varepsilon_N,\nabla_r\left(\dfrac{\phi}{\mathbf{M}}\right)\big) ds, \quad \forall \phi\in \mathbb{X},
			\end{align*}
            where 	 $\displaystyle\sum_{k\in K_n}g_k\sigma_k^N=h_n$    and $\displaystyle\sum_{k\in K_n}g_k (\nabla \sigma_k^Nr)=y_n$ and  $K_n=\{k\in K: \min(n,N) \leq    \vert   k\vert  \leq    \max(2N,n)\}$.
		\end{enumerate}
	\end{proposition}
    \begin{proof}[Proof of  \Cref{prop-exisence-mean}]
	Let $g\in   G_n$,   denote  $K_n=\{k\in K: \min(n,N) \leq    \vert   k\vert  \leq    \max(2N,n)\}$  and set $V_m(t)=\E(f_m(t)e_g(t)).$      
    From \eqref{approx-linear*SDE},   by  using   It\^o formula to the
	product and taking the expectation,    we  get 
    \begin{align}\label{approx-linear*SDE-mean}
    (V_m(t),w_{i,l})_H&-(f_{0},w_{i,l})_H=\int_0^t(\mathcal{Y}^mV_m,w_{i,l})_Hds, \quad 1\leq j,k \leq m \notag\\
    &	-\int_0^t\dfrac{1}{\beta}\big(\mathbf{M}\nabla_r(\dfrac{V_m}{\mathbf{M}}),\nabla_r(\dfrac{w_{i,l}}{\mathbf{M}})\big)ds -\alpha_N\int_0^t(\nabla_x V_m,\nabla_x w_{i,l})_Hds \notag\\&- \dfrac{1}{2}\int_0^t\big(\sum_{k\in K}\left(\varphi_\varepsilon^2(\nabla \sigma_k^Nr) \otimes (\nabla \sigma_k^Nr)\right) \nabla_rV_m,\nabla(\dfrac{w_{i,l}}{\mathbf{M}})\big)ds\\
            &-\sum_{k\in K_n }\int_0^t(g_k\sigma_k^N.\nabla_xV_m,w_{i,l})_H ds -\sum_{k\in K_n}\int_0^t(g_k\varphi_\varepsilon(\nabla \sigma_k^Nr).\nabla_rV_m,w_{i,l})_H ds,
 \notag
\end{align}
	Denote    $\displaystyle\sum_{k\in K_n}g_k\sigma_k^N=h_n$    and $\displaystyle\sum_{k\in K_n}g_k (\nabla \sigma_k^Nr)=y_n$.   Note that
	\eqref{approx-linear*SDE-mean}  is  linear  system  of  ODE,    by  using   a   classical   results  (see    e.g.    \cite[Chapter   V]{Oksendal2013})  we  get the    existence of   a   unique  $V_m\in C([0,T];H_m)$  to   \eqref{approx-linear*SDE-mean}.
    By  noticing  that  
		$\Div_x(h_n)=0$ and     $$
		\vert  y_n\vert      \lesssim_{N,a_\tau}   \vert   r\vert    \displaystyle(\sum_{    \vert   k\vert  \leq    n}\vert   g_k\vert^2)^{\frac{1}{2}}=\vert   r\vert  \Vert  g\Vert
		,$$  one obtains from \eqref{approx-linear*SDE-mean}    (by using  similar arguments to get \eqref{key-est-Galerkin})
        \begin{align}\label{key-est-Galerkin-Vm}
 \sup_{v\in[0,t]} \Vert V_m(v)\Vert_H^2 &+\dfrac{1}{\beta}\E\int_0^t [V_m(s)]_V^2ds\leq C_{\varepsilon,T}\Vert f_{0}\Vert_H^2 \text{ for all } t\in [0,T],
\end{align}
where $C_{\varepsilon,T}>0,$  independent of $m.$ Finally, we have proven \autoref{prop-exisence-mean}.
\end{proof}

\begin{lemma}\label{lem-uniq-Vn}
		Let   $\varepsilon>0, N\in   \mathbb{N}^*$.   Then, the   solution  $V^\varepsilon_N$ given by \autoref{prop-exisence-mean}    is  unique.
\end{lemma}
\begin{proof}
       Since the equation is linear with respect to $V^\varepsilon_N,$ it is sufficient to prove that $V^\varepsilon_N(t)\equiv 0$ for any $t\in ]0,T]$ if $f_0\equiv 0$. Let $\gamma>0$    and      $\rho$ be   a smooth density of a probability measure on
$\R^2$, compactly supported in $B(0, 1)$    and define  the approximation of identity for the convolution on $\R^2$ as  $\rho_\gamma(y)=\gamma^{-2}\rho(y/\gamma).$ 
	 Now, let  $\phi\in \mathbb{X}$, if  we  denote  $X=(x,r)\in \T^2\times  D$   and $\rho_{\gamma}(X)=\rho_\gamma(x)$,
	then    $\phi_{\gamma}:=\rho_{\gamma}*\phi \in \mathbb{X}$\footnote{  Since   we  are working on  $\T^2\times D,$  we  recall  that    for any  integrable function $g$    on $\T^2$, $g$ can be extended
periodically to a locally integrable function on the whole $\R^2$   and
convolution $\rho_\gamma * g$ is    meaningful   and  $\rho_\gamma * g$ is still a $C^\infty$-periodic function. }, from 
    \autoref{prop-exisence-mean} we get,  for any $t\in[0,T],$    
			\begin{align*}
				&( [V^\varepsilon_N(t)]_\gamma,\phi)_H=-\int_0^t \left( [h_n\cdot \nabla_xV^\varepsilon_N(s)]_\gamma+\varphi_\varepsilon[y_n \cdot\nabla_rV^\varepsilon_N(s)]_\gamma,  \phi \right)_Hds \\
				&        -\dfrac{1}{\beta}\int_0^t\big(\mathbf{M}\nabla_r(\dfrac{[V^\varepsilon_N(s)]_\gamma}{\mathbf{M}}),\nabla_r\left(\dfrac{\phi}{\mathbf{M}}\right)\big) ds-\alpha_N\int_0^t( \nabla_x[V^\varepsilon_N(s)]_\gamma,\nabla_x \phi)_H ds\\& - \dfrac{1}{2}\int_0^t\big(\sum_{k\in K}[\left(\varphi_\varepsilon^2(\nabla \sigma_k^Nr) \otimes (\nabla \sigma_k^Nr)\right) \nabla_rV^\varepsilon_N(s)]_\gamma,\nabla_r\left(\dfrac{\phi}{\mathbf{M}}\right)\big) ds.
			\end{align*}
            Next, $\phi=[V^\varepsilon_N]_\gamma$ is a suitable test function in the last equality and we obtain
            \begin{align*}
			&\Vert[V^\varepsilon_N(t)]_\gamma\Vert^2_H+\dfrac{1}{\beta}\int_0^t[[V^\varepsilon_N(s)]_\gamma]^2_V ds+\alpha_N\int_0^t\Vert  \nabla_x[V^\varepsilon_N(s)]_\gamma\Vert_H^2 ds\\&=\int_0^t \left( h_n\cdot \nabla_x[V^\varepsilon_N(s)]_\gamma-[h_n\cdot \nabla_xV^\varepsilon_N(s)]_\gamma,  [V^\varepsilon_N(s)]_\gamma \right)_Hds \\
            &-\int_0^t \left(\varphi_\varepsilon[y_n \cdot\nabla_rV^\varepsilon_N(s)]_\gamma,  [V^\varepsilon_N(s)]_\gamma \right)_Hds - \dfrac{1}{2}\int_0^t\big(\sum_{k\in K}[\left(\varphi_\varepsilon^2(\nabla \sigma_k^Nr) \otimes (\nabla \sigma_k^Nr)\right) \nabla_rV^\varepsilon_N]_\gamma,\nabla_r(\dfrac{[V^\varepsilon_N(s)]_\gamma}{\mathbf{M}})\big) ds,
			\end{align*}
            where we used that $\Div_x h_n=0.$  Since  $V^\varepsilon_N\in L^\infty(0, T;H)$, we get by using  the commutator estimates   (see \textit{e.g.} \cite[Appendix C]{tahraoui2025small})
            \begin{align*}
                \int_0^t \left( h_n\cdot \nabla_x[V^\varepsilon_N(s)]_\gamma-[h_n\cdot \nabla_xV^\varepsilon_N(s)]_\gamma,  [V^\varepsilon_N(s)]_\gamma \right)_Hds \to 0 \text{ as } \gamma \to 0.
            \end{align*}
            On the other hand, by using that $V^\varepsilon_N\in L^2(0, T;V)\cap L^\infty(0, T;H)$ and the properties of the convolution, we obtain as $\gamma \to 0$
             \begin{align*}
			&\Vert V^\varepsilon_N(t)\Vert^2_H+\dfrac{1}{\beta}\int_0^t[V^\varepsilon_N(s)]^2_V ds\leq 
            -\int_0^t \left(\varphi_\varepsilon y_n \cdot\nabla_rV^\varepsilon_N(s),  V^\varepsilon_N(s) \right)_Hds\\& - \dfrac{1}{2}\int_0^t\big(\sum_{k\in K}\left(\varphi_\varepsilon^2(\nabla \sigma_k^Nr) \otimes (\nabla \sigma_k^Nr)\right) \nabla_rV^\varepsilon_N,\nabla_r(\dfrac{V^\varepsilon_N(s)}{\mathbf{M}})\big) ds,
			\end{align*}
            Argument similar to \eqref{esti-cov-uniq} ensures
            \begin{align*}
                & - \dfrac{1}{2}\int_0^t\big(\sum_{k\in K}\left(\varphi_\varepsilon^2(\nabla \sigma_k^Nr) \otimes (\nabla \sigma_k^Nr)\right) \nabla_rV^\varepsilon_N,\nabla_r(\dfrac{V^\varepsilon_N(s)}{\mathbf{M}})\big) ds \\&\leq  - \dfrac{1}{2}\int_0^t\big(\sum_{k\in K}\Vert \varphi_\varepsilon (\nabla \sigma_k^Nr)\cdot\nabla_rV^\varepsilon_N\Vert_H^2ds+\dfrac{\kappa (5+\kappa)C_A}{4\varepsilon^2}\int_0^t\Vert V^\varepsilon_N(s)\Vert_H^2 ds.
            \end{align*}
        A standard computation gives
       \begin{align*}
         \left\vert\int_0^t \left(\varphi_\varepsilon y_n \cdot\nabla_rV^\varepsilon_N(s),  V^\varepsilon_N(s) \right)_Hds  \right\vert \leq \dfrac{C(a_\tau)}{\varepsilon}\int_0^t (\sum_{    \vert   k\vert  \leq    n}\vert   g_k\vert^2)^{\frac{1}{2}}\Vert V^\varepsilon_N(s)\Vert_H^2 ds, 
       \end{align*}
       with $\mathrm{g}=(\sum_{    \vert   k\vert  \leq    n}\vert   g_k\vert^2)^{\frac{1}{2}}\in L^2(0,T)$. Finally, we get
         \begin{align*}
			&\Vert V^\varepsilon_N(t)\Vert^2_H+\dfrac{1}{\beta}\int_0^t[V^\varepsilon_N(s)]^2_V ds\leq \dfrac{C(a_\tau)}{\varepsilon}\int_0^t \mathrm{g}(s)\Vert V^\varepsilon_N(s)\Vert_H^2 ds +\dfrac{\kappa (5+\kappa)C_A}{4\varepsilon^2}\int_0^t\Vert V^\varepsilon_N(s)\Vert_H^2 ds.
			\end{align*}
          Grönwall's inequality ensures $\Vert V^\varepsilon_N(t)\Vert^2_H,$ which concludes the proof of  \autoref{lem-uniq-Vn}.
    \end{proof}
    Consequently, the  uniqueness  of    quasi-regular
		weak solutions  of  \eqref{Ito-FP_eps-V2} follows by the argument as in \cite[Subsection 6.3.1]{Flandoli-Tahraoui}.
\begin{remark}
In the setting of \autoref{TH1-*_nocutoff}, employing the same arguments as in the proof of \autoref{TH1-*_nocutoff}, it follows that the same proof provides \autoref{prop-exisence-mean} also with  $\varphi_\varepsilon\equiv 1$.
\end{remark}

\section{Diffusion scaling limit  as  $N\uparrow +\infty$ (\autoref{THm-scal-eps})}\label{Section-diffusion-limit}
 Our aim in  this    section is  to  show    that    the unique  solution    of   stochastic  FP  equation    \eqref{Ito-FP_eps-V2},    in  the sense   of  \Cref{Def-sol-QR-*}  converges, as $N\to +\infty$ for fixed $\varepsilon>0,$  to  the unique   solution    of \eqref{limit-PDE-eps-finalV1}. This shows that under the scaling of the noise given by  \eqref{scaling}, the stochastic equation \eqref{Ito-FP_eps-V2} converges to the deterministic equation \eqref{limit-PDE-eps-finalV1}, where the  matrix $A$ emerges.
 \subsection{Proof of \autoref{THm-scal-eps}}\label{sec:limN}
 We recall that under   the    scaling of  the noise coefficients \eqref{scaling},  we have  
  \begin{align}\label{coefficient}    \lim_{N\to +\infty}\sum_{k\in K_{++}}\left(  \theta^{N,\tau}_{k}\right)^{2}=0;\quad 
 \lim_{N\to +\infty}\sup_{k\in K}\left(  \theta^{N,\tau}_{k}\right)^{2}= 0 \text{ and }
 \lim_{N\to +\infty}\sup_{k\in K}\left\vert k\right\vert
 ^{2}\left(  \theta^{N,\tau}_{k}\right)^{2}= 0.
 \end{align}
Let $\varepsilon>0,$ by using \eqref{estimate-solution}, we have 
$(f^\varepsilon_N)_N$ is bounded in $L^2(\Omega ;L^2(0, T;V)) \cap L^2(\Omega ;L^\infty(0, T;H)).$  By similar argument to get  \eqref{cv-appro-4},  there exists    a subsequence of $(f^\varepsilon_N)_N$ (denoted by the same way) and  $f_\varepsilon \in L^2(\Omega ;L^2(0, T;V)) \cap L^2_{w-*}(\Omega ;L^\infty(0, T;H)) $ such that
      \begin{align}
 	f^\varepsilon_N \rightharpoonup f_\varepsilon &\text{		in	}  L^2(\Omega ;L^2(0, T;V)),\label{cv1-eps}\\
    f^\varepsilon_N \overset{*}{\rightharpoonup} f_\varepsilon &\text{		in	}  L^2_{w-*}(\Omega ;L^\infty(0, T;H)) \label{cv2-eps}
 	\end{align}
 
  \subsubsection{Passage to the limit as $N\to+\infty$}
   In the following,   we  will    establish   some    lemmas  to  pass    to  the limit   as  $N\to   +\infty.$\\
   
   Let $\widetilde{\phi}\in C^2(\T^2)\otimes C^1(\overline{D})$ and set $\phi=\mathbf{M}\widetilde{\phi}.$ Then it is easy to check that $\phi   \in \mathbb{X}$. Moreover,      let    $\xi\in  C^\infty_c(]0,T[)$ and $A\in \mathcal{F}$.
   For  $t\in [0,T]$, let us  set 
        \begin{align}\label{time-regu}
            B_N^\varepsilon(t) &:=(f_N^\varepsilon(t), \phi )_H \\& -\sum_{k\in K}\int_0^t\big(f_N^\varepsilon(s),( \sigma_k^N.\nabla_x\phi +(\nabla \sigma_k^Nr).\nabla_r (\varphi_\varepsilon\phi)+(\nabla \sigma_k^Nr).\dfrac{\kappa r}{1-\vert r\vert^2}\varphi_\varepsilon\phi)\big)_HdW^k(s).\notag
        \end{align}
		From    \Cref{Def-sol-QR-*}, point (2),   we  write   (in distributional sense    with    respect to  $t$)
	\begin{align}\dfrac{d}{dt}B_N^\varepsilon=&-\dfrac{1}{\beta}\big(\mathbf{M}\nabla_r(\dfrac{f^\varepsilon_N}{\mathbf{M}}),\nabla_r\left(\dfrac{\phi}{\mathbf{M}}\right)\big)+\alpha_N(f_N^\varepsilon ,\Delta_x\phi )_H \notag\\&- \dfrac{1}{2}\big(\sum_{k\in K}\left(\varphi_\varepsilon^2(\nabla \sigma_k^Nr) \otimes (\nabla \sigma_k^Nr)\right) \nabla_rf^\varepsilon_N,\nabla_r\left(\dfrac{\phi}{\mathbf{M}}\right)\big).\label{dist-limit}\end{align}
Since  P a.s.  $f^\varepsilon_N \in C_w([0, T];H)$, we also have $B_N^\varepsilon(0) =(f_0, \phi)_H.$
By multiplying the      equation \eqref{dist-limit}    by    $\mathbb{I}_A\xi$ and integrating  over $\Omega \times [0,T]$ we derive 
	\begin{align}
		-\int_A\int_0^T\bigl[B_N^\varepsilon\dot\xi \bigr]dsdP&=\int_A\int_0^T\big(-\dfrac{1}{\beta}\big(\mathbf{M}\nabla_r(\dfrac{f^\varepsilon_N}{\mathbf{M}}),\nabla_r\left(\dfrac{\phi}{\mathbf{M}}\right)\big)+\alpha_N(f_N^\varepsilon ,\Delta_x\phi )_H \big) \xi dsdP \label{cv1-N}\\&-\dfrac{1}{2}\int_A\int_0^T\big(\sum_{k\in K}\left(\varphi_\varepsilon^2(\nabla \sigma_k^Nr) \otimes (\nabla \sigma_k^Nr)\right) \nabla_rf^\varepsilon_N,\nabla_r\left(\dfrac{\phi}{\mathbf{M}}\right))\xi dsdP.	\notag\end{align}
        Now, we pass to the limit as $N\to +\infty$ in the last equality \eqref{cv1-N}. The following result concerns the limit of the right hand side of the equality \eqref{cv1-N}.
        	\begin{lemma}\label[lemma]{lem-cv-2}
		The   following  convergences hold:  
        \begin{align}
        & \lim_N\dfrac{1}{\beta}\int_A\int_0^T\left(f_N^\varepsilon, \phi \right)_V \xi dsdP= \dfrac{1}{\beta}\int_A\int_0^T\left(f^\varepsilon, \phi \right)_V \xi dsdP\label{cv-N-C}\\
        & \displaystyle\lim_N\alpha_N\int_A\int_0^T(f_N^\varepsilon ,\Delta_x\phi )_H \xi dsdP=0 \label{cv-N-A},\\ &\lim_N\int_A\int_0^T\big(\varphi_\varepsilon^2A^N(r)\nabla_rf^\varepsilon_N,\nabla_r\left(\dfrac{\phi}{\mathbf{M}}\right))\xi dsdP =\int_A\int_0^T\big(\varphi_\varepsilon^2A(r)\nabla_rf_\varepsilon,\nabla_r\left(\dfrac{\phi}{\mathbf{M}}\right))\xi dsdP,\label{cv-N-B}
        \end{align}
        where    	$A(r)
		=\dfrac{\lambda}{\tau} (3 \left\vert r\right\vert^2I-2r\otimes r).$
	\end{lemma}
\begin{proof} 	
First, notice that \eqref{cv-N-C} is a consequence of \eqref{cv1-eps}. On the other hand, 	we have \begin{align*}
			\vert\alpha_N\int_A\int_0^T(f_N^\varepsilon, \Delta_x\phi)_H \xi  dsdP\vert
			&\leq C\alpha_N\Vert f_N^\varepsilon\Vert_{L^2(\Omega\times[0, T];H))}\Vert \Delta_x\phi\Vert_{H}\Vert \xi\Vert_{\infty}\\&\leq  C\alpha_N\Vert f_N^\varepsilon\Vert_{L^2(\Omega\times[0, T];H))}\Vert \phi\Vert_{\mathbb{X}}\Vert \xi\Vert_{\infty},
		\end{align*}
		where $C>0$ is a constant independent of $N$. On the other hand, we have   $
			0  \leq \alpha_N \lesssim  \dfrac{1}{N^3} $ and   \eqref{cv-N-A} follows. Concerning \eqref{cv-N-B}, thanks to 	 \eqref{limit-A-N} we have 
		\begin{align}\label{cv-3-*}
			& \int_A\int_0^T\big(\varphi_\varepsilon^2A^N(r) \nabla_rf^\varepsilon_N,\nabla_r\left(\dfrac{\phi}{\mathbf{M}}\right))\xi dsdP \notag\\&=
       \int_A\int_0^T\big(\varphi_\varepsilon^2A(r)\nabla_rf_\varepsilon,\nabla_r\left(\dfrac{\phi}{\mathbf{M}}\right))\xi dsdP+R_N,
		\end{align}
		where $R_N$ satisfies
		\begin{align*}
			\vert	R_N\vert  &\leq  \dfrac{C}{\tau N} 	\int_A\int_0^T\int_{\mathbb{T}^2}\int_{D} \varphi_\varepsilon^2\vert P(r)\vert \vert \nabla_r f^\varepsilon_N\vert\cdot\vert \nabla_r\left(\dfrac{\phi}{\mathbf{M}}\right)\vert \vert\xi\vert dr dxdsdP \\
			&\leq\dfrac{C_\varepsilon}{\tau N} 	\Vert f^\varepsilon_N\Vert_{L^2(\Omega\times[0, T];V))}\Vert \phi\Vert_{V}\Vert \xi\Vert_{\infty} \to 0   \text{  as  }   N\to    +\infty.
		\end{align*}
		Now, by passing to the limit in \eqref{cv-3-*}, \eqref{cv-N-B} follows.
         	\end{proof} 	
	\begin{lemma}\label{Lemma-stoch-cv} The   following  convergence holds:  
	    \begin{align}\label{limit-cv-*}
	        	\lim_N	\int_A\int_0^TB_N^\varepsilon \dot\xi dsdP=\int_A\int_0^T 
                (f_\varepsilon, \phi)_H \dot \xi  dsdP.
	    \end{align}
	\end{lemma}
\begin{proof} Notice that
\begin{align}
   \label{cv-3-stoch} \int_A\int_0^T\bigl[ B_N^\varepsilon\dot\xi \bigr]dsdP =\int_A\int_0^T(f_N^\varepsilon(s),\phi)_H \dot\xi dsdP-I_N^\varepsilon,
\end{align}
where 
\begin{align}\label{Martingale-lim}
I_N^\varepsilon=\int_A\int_0^T\bigl[ \sum_{k\in K}\int_0^s\int_{\mathbb{T}^2}\int_{D}f_N^\varepsilon(v)( \sigma_k^N.\nabla_x(\dfrac{\phi}{\mathbf{M}}) +(\nabla \sigma_k^Nr).\nabla_r\big(\dfrac{\phi \varphi_\varepsilon}{\mathbf{M}}\big)) drdxdW^k_v\dot \xi\bigr]dsdP.
\end{align}
Moreover, we have
\begin{align}\label{bound-weight-cut}
    \Vert \nabla_r\big(\dfrac{\phi \varphi_\varepsilon}{\mathbf{M}}\big)\Vert_\infty= \Vert \nabla_r\big(\widetilde{\phi} \varphi_\varepsilon\big)\Vert_\infty \leq  \Vert \nabla_r\widetilde{\phi}\Vert_\infty+C{\varepsilon}^{-1}\Vert \widetilde{\phi}\Vert_\infty =\mathbf{K}_\varepsilon.
\end{align}
Let us prove that $\displaystyle\lim_N I_N^\varepsilon=0.$  
Thanks to Fubini’s theorem and Cauchy-Schwarz, it is clear that this follows if we can show that 
\begin{align}
    \lim_N \EE{\sup_{t\in [0, T]}\left|\int _0^t \sum_{k\in K}\langle f_N^\varepsilon,( \sigma_k^N.\nabla_x(\dfrac{\phi}{\mathbf{M}}) +(\nabla \sigma_k^Nr).\nabla_r\big(\dfrac{\phi \varphi_\varepsilon}{\mathbf{M}}\big)\rangle_{L^2(\T^2\times D)}\right|^2} =0
\end{align}
By, Burkholder-Davis-Gundy we have 
\begin{align*}
    &\EE{\sup_{t\in [0, T]} \left|\int _0^t \sum_{k\in K}\left\langle f_N^\varepsilon,( \sigma_k^N\cdot\nabla_x(\dfrac{\phi}{\mathbf{M}}) +(\nabla \sigma_k^Nr)\cdot\nabla_r\big(\dfrac{\phi \varphi_\varepsilon}{\mathbf{M}}\big)\right\rangle_{L^2(\T^2\times D)}ds\right|^2} \\ 
    &\lesssim \EE{\sum_{k\in K}\int _0^T \left\langle f_N^\varepsilon, \sigma_k^N\cdot \nabla_x(\dfrac{\phi}{\mathbf{M}}) +(\nabla \sigma_k^Nr)\cdot \nabla_r\big(\dfrac{\phi \varphi_\varepsilon}{\mathbf{M}}\big)\right\rangle_{L^2(\T^2\times D)}^2 ds}\\
    &\lesssim I^N_1 + I^N_2.
\end{align*}
		Where, by  using   \eqref{Def-sigma-k},   we  get
		\begin{align*}
			I_1^N&:=2\E\int_0^T\sum_{k\in K}\left[\int_{D}\int_{\mathbb{T}^2}(f_N^\varepsilon(v)\nabla_x(\dfrac{\phi}{\mathbf{M}}))\cdot\sigma_k^N dxdr\right]^2dv\\&=\E\int_0^T\sum_{k\in K_+}(\theta^{N,\tau}_{k})^2\left[\int_{D}\int_{\mathbb{T}^2}(f_N^\varepsilon(v)\nabla_x(\dfrac{\phi}{\mathbf{M}}))\cdot  \dfrac{k^\perp}{\vert   k\vert}\cos{k\cdot  x}dxdr\right]^2dv\\&+\E\int_0^T\sum_{k\in K_-}(\theta^{N,\tau}_{k})^2\left[\int_{D}\int_{\mathbb{T}^2}(f_N^\varepsilon(v)\nabla_x(\dfrac{\phi}{\mathbf{M}}))\cdot  \dfrac{k^\perp}{\vert   k\vert}\sin{k\cdot  x}dxdr\right]^2dv.
		\end{align*}
		By using   that    $(\frac{k^\perp}{\vert   k\vert}\cos{k\cdot  x},\frac{k^\perp}{\vert   k\vert}\sin{k\cdot  x})_{k\in  K}$  is an
		(incomplete) orthonormal system in  $L^2(\T^2;\mathbb{R}^2)$,    we  obtain by  using   \eqref{coefficient}  
      	\begin{align*}
			\vert I_1^N\vert&\leq     \sup_{k\in  K}(\theta^{N,\tau}_{k})^2\E\int_0^T\int_{D}\int_{\mathbb{T}^2}(f_N^\varepsilon(v)\nabla_x(\dfrac{\phi}{\mathbf{M}}))^2dxdrdv\\&\leq\sup_{k\in  K}(\theta^{N,\tau}_{k})^2\E\int_0^T\int_{D}\int_{\mathbb{T}^2}(f_N^\varepsilon(v)\nabla_x\widetilde{\phi})^2dxdrdv\\
			&\leq    \Vert \nabla_x\widetilde{\phi}\Vert_\infty^2\sup_{k\in  K}(\theta^{N,\tau}_{k})^2  \E\int_0^T\int_{D}\int_{\mathbb{T}^2}f_N^\varepsilon(x,r,v)^2dxdrdv\\&\leq   \Vert \nabla_x\widetilde{\phi}\Vert_\infty^2\sup_{k\in  K}(\theta^{N,\tau}_{k})^2\Vert  f_N^\varepsilon\Vert_{L^2(\Omega\times[0,T];H)}^2\to   0   \text{  as  }   N\uparrow    +\infty.
		\end{align*}
		Concerning   $I_2^N$,    we  have
		\begin{align*}
			I_2^N&:= 2\E\int_0^T\sum_{k\in K}\left[\int_{D}\int_{\mathbb{T}^2}\left(f_N^\varepsilon(v)\nabla_r\big(\dfrac{\phi \varphi_\varepsilon}{\mathbf{M}}\big)\right)\cdot(\nabla \sigma_k^Nr) dxdr\right]^2dv\\&=\E\int_0^T\sum_{k\in K_+}(\theta^{N,\tau}_{k})^2\vert   k\vert^2\left[\int_{D}\int_{\mathbb{T}^2}\dfrac{k}{\vert   k\vert}\cdot   r\left(f_N^\varepsilon(v)\nabla_r\big(\dfrac{\phi \varphi_\varepsilon}{\mathbf{M}}\big)\right)\cdot  \dfrac{k^\perp}{\vert   k\vert}\sin{k\cdot  x}dxdr\right]^2dv\\&+\E\int_0^T\sum_{k\in K_-}(\theta^{N,\tau}_{k})^2\vert   k\vert^2\left[\int_{D}\int_{\mathbb{T}^2}\dfrac{k}{\vert   k\vert}\cdot   r\left(f_N^\varepsilon(v)\nabla_r\big(\dfrac{\phi \varphi_\varepsilon}{\mathbf{M}}\big)\right)\cdot  \dfrac{k^\perp}{\vert   k\vert}\cos{k\cdot  x}dxdr\right]^2dv.
		\end{align*}
		Thus,    by  using   that $(\dfrac{k}{\vert   k\vert}\otimes\dfrac{k^\perp}{\vert   k\vert}\cos{k\cdot  x},\dfrac{k}{\vert   k\vert}\otimes\dfrac{k^\perp}{\vert   k\vert}\sin{k\cdot  x})_{k\in  K}$  is an
		(incomplete) orthonormal system in  $L^2(\T^2;\mathbb{R}^2\otimes \mathbb{R}^2)$    we  obtain
		\begin{align*}
			\vert I_1^N\vert&\leq     \sup_{k\in  K}[(\theta^{N,\tau}_{k})^2\vert   k\vert^2]\E\int_0^T\int_{D}\int_{\mathbb{T}^2}   \vert    r\vert^2 \vert \nabla_r\big(\dfrac{\phi \varphi_\varepsilon}{\mathbf{M}}\big)\vert^2\left(f_N^\varepsilon(v)\right)^2dxdrdv\\
			&\leq  \mathbf{K}_\varepsilon^2\sup_{k\in  K}[(\theta^{N,\tau}_{k})^2\vert   k\vert^2]\Vert  f_N^\varepsilon\Vert_{L^2(\Omega\times[0,T];H)}^2\to   0   \text{  as  }   N\uparrow   +\infty.
		\end{align*}
Finally, by using \eqref{cv1-eps} in  \eqref{cv-3-stoch},  we get \eqref{limit-cv-*}.
\end{proof}
Finally, by using \Cref{lem-cv-2} and \Cref{Lemma-stoch-cv}, we can pass to the limit in \eqref{cv1-N} as $N\to +\infty$ to get
	\begin{align}\label{iden-test-class}
		-\int_A\int_0^T 
                (f_\varepsilon, \phi)_H \dot\xi  dsdP=-\int_A\int_0^T\big(\dfrac{1}{\beta}\mathbf{M}\nabla_r(\dfrac{f_\varepsilon}{\mathbf{M}})+\dfrac{1}{2}\varphi_\varepsilon^2A(r)\nabla_rf_\varepsilon,\nabla_r\left(\dfrac{\phi}{\mathbf{M}}\right)\big) \xi dsdP\end{align}
                It is worth noticing that the last equality \eqref{iden-test-class} holds, \textit{a priori}, with test functions $\phi=\mathbf{M} \widetilde{\phi}$ such that $  \widetilde{\phi}\in C^2(\T^2)\otimes C^1(\overline{D})$ but since $    \overline{C^1(\overline{D})}^{\mathcal{H}^1_\mathbf{M}}=\mathcal{H}^1_{\mathbf{M}}$ (see \cite[Proposition B.2]{Masmoudi2013}) where $\mathcal{H}^1_{\mathbf{M}}=\{g: \quad \int_D(\vert g\vert^2+\vert \nabla g\vert^2)\mathbf{M} dr<\infty\}$, we deduce that \eqref{iden-test-class} holds for any $\phi\in V.$  \\ 
                
In particular, $f_\varepsilon:=f_\varepsilon(t,\omega, x,r)$ solves, in the sense of distribution, the following problem
\begin{align}\label{limit-PDE-eps}
    \dfrac{\partial f_\varepsilon}{\partial t}=\Div_r\big ( \dfrac{1}{\beta}\mathbf{M}\nabla_r(\dfrac{f_\varepsilon}{\mathbf{M}})+\dfrac{1}{2}\varphi_\varepsilon^2A(r)\nabla_rf_\varepsilon\big)=\Div_r\big ( \dfrac{\k}{\beta} f_\varepsilon F(r)r+\dfrac{1}{\beta}\nabla_rf_\varepsilon+ \dfrac{1}{2} \varphi_\varepsilon^2A(r) \nabla_rf_\varepsilon\big).
\end{align}
On the other hand, we have ( see \eqref{cv1-eps} and \eqref{cv2-eps}) P-a.s.  $f_\varepsilon \in L^2(0, T;V) \cap L^\infty(0, T;H)$ and  \eqref{limit-PDE-eps} ensures that $\dfrac{\partial f_\varepsilon}{\partial t}\in  L^2(0,T;V^\prime)$, which gives $f_\varepsilon\in C([0,T];H)$.
    Recall that  $B_N^\varepsilon$  given by \eqref{time-regu} satisfies:
$ B_N^\varepsilon\in L^2(\Omega\times[0,T]) \text{ and } \partial_tB_N^\varepsilon\in L^2(\Omega\times[0,T]),$ which guarantees that $ B_N^\varepsilon\in   L^2(\Omega;C([0,T]).$\\
We   finish  the proof   by  showing some    continuous  convergence in  time.   Let  $\xi \in \mathcal{C}^\infty([0,t])$ for $t\in ]0,T]$ and note that the	following	integration	by	parts	formula	for $B_N^\varepsilon$, given by  \eqref{time-regu}, holds
	\begin{align}\label{IPPtimez-1}
	\int_A	\int_0^t	\dfrac{dB_N^\varepsilon}{ds}(s)\xi(s) ds dP&=-\int_A\int_0^tB_N^\varepsilon(s)\dot\xi ds dP+\int_AB_N^\varepsilon(t)\xi(t)dP- \int_A(  f_0 ,\phi)_H\xi(0)dP.
	\end{align}
	Now,    by  standard   arguments  (see \textit{e.g.}  \cite[proof    of  Prop. 3.]{vallet2019well}) and using \eqref{iden-test-class} we  get
	$$ \text{ for any } t\in [0,T]:   \int_A B_N^\varepsilon(t) dP \to    \int_A(f_\varepsilon(t),\phi)_H dP, \text{ as } N\to \infty.$$
	and    $f_\varepsilon(0)=f_0$   in  $H$-sense.
    Finally, recall that from \eqref{time-regu}, we have
    \begin{align*}
   (f_N^\varepsilon(t), \phi )_H=     B_N^\varepsilon(t)  +\sum_{k\in K}\int_0^t\big(f_N^\varepsilon(s),( \sigma_k^N.\nabla_x\phi +(\nabla \sigma_k^Nr).\nabla_r (\varphi_\varepsilon\phi)+(\nabla \sigma_k^Nr).\dfrac{\kappa r}{1-\vert r\vert^2}\varphi_\varepsilon\phi)\big)_HdW^k(s).
    \end{align*}
    Notice that Burkholder-Davis-Gundy  inequality ensures
    \begin{align*}
        &\E\sup_{v\in[0,t]}\vert \sum_{k\in K}\int_0^v\big(f_N^\varepsilon(s),( \sigma_k^N.\nabla_x\phi +(\nabla \sigma_k^Nr).\nabla_r (\varphi_\varepsilon\phi)+(\nabla \sigma_k^Nr).\dfrac{\kappa r}{1-\vert r\vert^2}\varphi_\varepsilon\phi)\big)_HdW^k(s)\vert\\
        &=\E\sup_{v\in[0,t]}\vert \int_0^v\int_{\mathbb{T}^2}\int_{D} \sum_{k\in K}f_N^\varepsilon(v)( \sigma_k^N.\nabla_x(\dfrac{\phi}{\mathbf{M}}) +(\nabla \sigma_k^Nr).\nabla_r\big(\dfrac{\phi \varphi_\varepsilon}{\mathbf{M}}\big)) drdxdW^k(v) \vert\\
        &\leq  C_B \E\Big[ \int_0^t \sum_{k\in K}\big[\int_{\mathbb{T}^2}\int_{D} f_N^\varepsilon(v)( \sigma_k^N.\nabla_x(\dfrac{\phi}{\mathbf{M}}) +(\nabla \sigma_k^Nr).\nabla_r\big(\dfrac{\phi \varphi_\varepsilon}{\mathbf{M}}\big)) drdx\big]^2dv\Big]^{1/2}\\
        &\leq C_BT^{1/2}\Big[\E \int_0^t \sum_{k\in K}\big[\int_{\mathbb{T}^2}\int_{D} f_N^\varepsilon(v)( \sigma_k^N.\nabla_x(\dfrac{\phi}{\mathbf{M}}) +(\nabla \sigma_k^Nr).\nabla_r\big(\dfrac{\phi \varphi_\varepsilon}{\mathbf{M}}\big)) drdx\big]^2dv\Big]^{1/2} \to 0,
    \end{align*}
as $N\uparrow+\infty$, where the convergence to $0$ follows by similar argument as in the proof of  \autoref{Lemma-stoch-cv}. Consequently, we get 
\begin{align*}
  \forall \phi\in H: \quad      \lim_{N\to+\infty}   \int_A(f_N^\varepsilon(t), \phi )_HdP=\int_A (f_\varepsilon(t), \phi )_H dP\text{ for any } t\in [0,T].
\end{align*}
     \begin{proposition}\label{prop-uniq-1}
          For any $\varepsilon>0,$ the solution $f_\varepsilon$ given by \Cref{THm-scal-eps} is unique.     
    \end{proposition}
  \begin{proof}
      Since \eqref{limit-PDE-eps-finalV1} is linear with respect to $f_\varepsilon,$ it is sufficient to prove that $f_\varepsilon(t)\equiv 0$ for any $t\in ]0,T]$ if $f_0\equiv 0$. Let $\varepsilon>0$ and $t\in ]0,T[$ and notice that $f_\varepsilon$ is an admissible test function in \eqref{limit-PDE-eps-finalV1}. Therefore, after integration from $0$ to $t$ we get 
       \begin{align*}
    \int_0^t\langle\dfrac{\partial f_\varepsilon}{\partial s},f_\varepsilon\rangle_{V^\prime,V}ds=-\int_0^t\int_{\mathbb{T}^2}\int_D\big ( \dfrac{1}{\beta}\mathbf{M}\nabla_r(\dfrac{f_\varepsilon}{\mathbf{M}})+ \varphi_\varepsilon^2\dfrac{1}{2} A(r) \nabla_rf_\varepsilon\big)\cdot \nabla_r(\dfrac{f_\varepsilon}{\mathbf{M}}) drdxds.
\end{align*}
Since $2\displaystyle\int_0^t\langle\dfrac{\partial f_\varepsilon}{\partial s},f_\varepsilon\rangle_{V^\prime,V}ds=\Vert f_\varepsilon(t)\Vert_H^2,$ thanks to the  results based on  Lions-Guelfand
triple (see \textit{e.g.} \cite[Theorem 3.1.]{Lions-Magenes}). By using the definition of $A(r)$, given by \eqref{matrix-limit}, we have
\begin{align*}
    \Vert f_\varepsilon(t)\Vert_H^2+\int_0^t \dfrac{2}{\beta}[f_\varepsilon(s)]_V^2 ds \leq -\int_0^t\int_{\mathbb{T}^2}\int_D \varphi_\varepsilon^2 A(r) \nabla_rf_\varepsilon\cdot \nabla_r(\dfrac{f_\varepsilon}{\mathbf{M}}) drdxds.
\end{align*}
On the other hand, notice that
\begin{align*}
-\int_0^t\int_{\mathbb{T}^2}\int_D \varphi_\varepsilon^2 A(r) \nabla_rf_\varepsilon\cdot \nabla_r(\dfrac{f_\varepsilon}{\mathbf{M}}) drdxds \leq& -k_T\int_0^t\int_{\mathbb{T}^2}\int_D  \varphi_\varepsilon^2\vert r\vert^2\mathbf{M} \vert \nabla_r(\dfrac{f_\varepsilon}{\mathbf{M}})\vert^2 drdxds\\&+\int_0^t\int_{\mathbb{T}^2}\int_D \varphi_\varepsilon^2 f_\varepsilon A(r)\cdot \dfrac{\kappa r}{1-\vert r\vert^2} \cdot \nabla_r(\dfrac{f_\varepsilon}{\mathbf{M}}) drdxds.
\end{align*}
Since 
\begin{align*}
    &\int_0^t\int_{\mathbb{T}^2}\int_D \varphi_\varepsilon^2 f_\varepsilon A(r)\cdot \dfrac{\kappa r}{1-\vert r\vert^2} \cdot \nabla_r(\dfrac{f_\varepsilon}{\mathbf{M}}) drdxds \leq k_T\kappa   \int_0^t\int_{\mathbb{T}^2}\int_D\dfrac{\varphi_\varepsilon^2  \vert r\vert^3}{1-\vert r\vert^2} \vert f_\varepsilon\vert \vert \nabla_r(\dfrac{f_\varepsilon}{\mathbf{M}})\vert drdxds\\
    &\leq  \dfrac{k_T}{2}\int_0^t\int_{\mathbb{T}^2}\int_D  \varphi_\varepsilon^2\vert r\vert^2\mathbf{M} \vert \nabla_r(\dfrac{f_\varepsilon}{\mathbf{M}})\vert^2 drdxds+ \dfrac{k_T\kappa^2}{2}\int_0^t\int_{\mathbb{T}^2}\int_D   \dfrac{\varphi_\varepsilon^2  \vert r\vert^4}{(1-\vert r\vert^2)^2} \dfrac{f_\varepsilon^2}{\mathbf{M}} drdxds,
\end{align*} we get
\begin{align*}     \Vert f_\varepsilon(t)\Vert_H^2+\int_0^t \dfrac{2}{\beta}[f_\varepsilon(s)]_V^2 ds+ \dfrac{k_T}{2}\int_0^t\int_{\mathbb{T}^2}\int_D  \varphi_\varepsilon^2\vert r\vert^2\mathbf{M} \vert \nabla_r(\dfrac{f_\varepsilon}{\mathbf{M}})\vert^2 drdxds\leq \dfrac{k_T\kappa^2}{2\varepsilon^2}\int_0^t \Vert f_\varepsilon(s)\Vert_H^2ds,
\end{align*}
 and Gr\"onwall’s lemma concludes the proof.
\end{proof}

\section{Removing the cut-off (\autoref{THm-eps-nocut})}\label{section-remove-cut-off}
    In  this    section,   we  show    that    the unique  solution    of   scaling limit equation    \eqref{limit-PDE-eps-finalV1} converges, as $\varepsilon\to 0$,   to  the unique   solution    of \eqref{limit-PDE-final-NOCUT-V1} given by \autoref{THm-eps-nocut}. To reach the limit $\varepsilon\to 0$, we must work in the appropriate weighted space, where the weight depends on $\varepsilon$ but sits  between the weights $\mathbf{M}$ and $\mathbf{M}_0.$
     \subsection{Proof of \autoref{THm-eps-nocut}}
Let $\varepsilon>0$, from \autoref{THm-scal-eps} there exists a unique $$f_\varepsilon \in L^2(0, T;V) \cap C([0, T];H)\footnote{Since the dependence on $\omega\in \Omega$ does not affect the proof, we ignore referring to that.}$$  such that
        \begin{align}\label{eq-to-test}
    \langle\dfrac{\partial f_\varepsilon}{\partial t},\Phi\rangle_{V^\prime,V}=-\int_{\mathbb{T}^2}\int_D\big ( \dfrac{\k}{\beta} f_\varepsilon F(r)r+\dfrac{1}{\beta}\nabla_rf_\varepsilon+ \dfrac{1}{2}\varphi_\varepsilon^2 A(r) \nabla_rf_\varepsilon\big)\cdot \nabla_r\left(\dfrac{\Phi}{\mathbf{M}}\right) drdx,  \forall \Phi \in V.
\end{align}
Let $r\in D$ and  consider the radial function $\mathbf{M}_\varepsilon$ given by
\begin{align}\label{weight-inter}
    \mathbf{M}_\varepsilon(r)=\exp{\big(-\kappa\int_0^{\vert r\vert} \dfrac{s ds}{(1-s^2)(1+\frac{\lambda\beta}{2 \tau}\varphi_\varepsilon^2 s^2)}  \big)}.
\end{align}
Note that   
$\mathbf{M}_\varepsilon$  solves
\begin{align}\label{eqn-weight-eps}
     \dfrac{\k}{\beta} \mathbf{M}_\varepsilon F(r)r+\dfrac{1}{\beta}\nabla_r\mathbf{M}_\varepsilon+ \dfrac{1}{2}\varphi_\varepsilon^2 A(r) \nabla_r\mathbf{M}_\varepsilon=0.
\end{align}  On the other hand,  denote by $\widetilde{\mathbf{M}}_0(r)=\left(
\dfrac{1 - \vert r\vert^{2}}{1 + \frac{\lambda\beta}{2 \tau}\, \vert r\vert^{2}}
\right)^{\frac{\kappa}{2\left(1 + \frac{\lambda\beta}{2 \tau}\right)}}.$ Since $0\leq \varphi_\varepsilon \leq 1,$ one gets
    \begin{align*}
           (1-\vert r\vert^2)^{\frac{\kappa}{2}} \leq \exp{\big(-\kappa\int_0^{\vert r\vert} \dfrac{s ds}{(1-s^2)(1+\frac{\lambda\beta}{2 \tau}\varphi_\varepsilon^2 s^2)}  \big)} \leq  \left(
\dfrac{1 - \vert r\vert^{2}}{1 + \frac{\lambda\beta}{2 \tau}\, \vert r\vert^{2}}
\right)^{\frac{\kappa}{2\left(1 + \frac{\lambda\beta}{2 \tau}\right)}},
    \end{align*}
and the following holds
\begin{align}\label{relation-weghts}
  C\mathbf{M}(r) \leq \mathbf{M}_\varepsilon(r) \leq  \widetilde{\mathbf{M}}_0(r) \text{ for any } r\in D, \quad \text{ where } C=\int_D (1-\vert r\vert^2)^{\frac{\kappa}{2}} dr.
\end{align}
In order to get a uniform estimate with respect to $\varepsilon,$ we need the following lemma.
\begin{lemma}\label{Lem-appropriate-test-function}
 Let $\varepsilon>0,$ the function $g_\varepsilon=\dfrac{\mathbf{M}f_\varepsilon}{\mathbf{M}_\varepsilon}  \in L^2(0, T;V) \cap C([0, T];H).$ 
\end{lemma}
\begin{proof}
Let $\varepsilon>0,$   thanks to \eqref{relation-weghts} and since $f_\varepsilon\in C([0, T];H),$ one gets $g_\varepsilon \in  C([0, T];H).$ On the other hand, we have
  \begin{align*}
      \nabla_r(\dfrac{g_\varepsilon}{\mathbf{M}})=\nabla_r(\dfrac{f_\varepsilon}{\mathbf{M}_\varepsilon})=\nabla_r(\dfrac{f_\varepsilon}{\mathbf{M}})\dfrac{\mathbf{M}}{\mathbf{M}_\varepsilon}+\dfrac{f_\varepsilon}{\mathbf{M}} \nabla_r(\dfrac{\mathbf{M}}{\mathbf{M}_\varepsilon})= \nabla_r(\dfrac{f_\varepsilon}{\mathbf{M}})\dfrac{\mathbf{M}}{\mathbf{M}_\varepsilon}+\dfrac{f_\varepsilon}{\mathbf{M}} \dfrac{\mathbf{M}}{\mathbf{M}_\varepsilon}h_\varepsilon(r),
  \end{align*}
  where $h_\varepsilon(r)=\varphi_\varepsilon^2\dfrac{-\kappa r}{(1-\vert r\vert^2)} \dfrac{\frac{\lambda\beta}{2 \tau}\vert r\vert^2 }{(1+\frac{\lambda\beta}{2 \tau}\varphi_\varepsilon^2 \vert r\vert^2)}.$ Notice that $\vert h_\varepsilon\vert \leq \dfrac{\lambda\beta\kappa}{2 \tau}\dfrac{1}{\varepsilon} $.  Therefore, \autoref{Lem-appropriate-test-function} is a consequence of    $f_\varepsilon\in L^2(0, T;V)$ and  \eqref{relation-weghts}. 
\end{proof}
    \subsubsection*{Uniform estimate with respect to $\varepsilon$}
  As a consequence of \autoref{Lem-appropriate-test-function}, $g_\varepsilon$ is an admissible test function in \eqref{eq-to-test}. Therefore, we have
  \begin{align}
    \langle\dfrac{\partial f_\varepsilon}{\partial t},g_\varepsilon\rangle_{V^\prime,V}=-\int_{\mathbb{T}^2}\int_D\big ( \dfrac{\k}{\beta} f_\varepsilon F(r)r+\dfrac{1}{\beta}\nabla_rf_\varepsilon+ \dfrac{1}{2}\varphi_\varepsilon^2 A(r) \nabla_rf_\varepsilon\big)\cdot \nabla_r(\dfrac{f_\varepsilon}{\mathbf{M}_\varepsilon}) drdx.
\end{align}
By using \eqref{eqn-weight-eps}, the right hand side becomes
\begin{align*}
    &-\int_{\mathbb{T}^2}\int_D\big ( \dfrac{\k}{\beta} f_\varepsilon F(r)r+\dfrac{1}{\beta}\nabla_rf_\varepsilon+ \dfrac{1}{2}\varphi_\varepsilon^2 A(r) \nabla_rf_\varepsilon\big)\cdot \nabla_r(\dfrac{f_\varepsilon}{\mathbf{M}_\varepsilon}) drdx\\&=-\int_{\mathbb{T}^2}\int_D \mathbf{M}_\varepsilon (\dfrac{1}{\beta}I+\dfrac{1}{2}\varphi_\varepsilon^2A(r))\nabla_r(\dfrac{f_\varepsilon}{\mathbf{M}_\varepsilon})\cdot \nabla_r(\dfrac{f_\varepsilon}{\mathbf{M}_\varepsilon}) drdx\\
    &\leq-\int_{\mathbb{T}^2}\int_D \mathbf{M}_\varepsilon (\dfrac{1}{\beta}+\dfrac{k_T}{2}\varphi_\varepsilon^2\vert r\vert^2)\vert \nabla_r(\dfrac{f_\varepsilon}{\mathbf{M}_\varepsilon})\vert^2 dr dx.
\end{align*}
Now, recall that  $f_\varepsilon \in L^2(0, T;V)$ and $\partial_tf_\varepsilon \in L^2(0, T;V^\prime)$ thanks to   \eqref{eq-to-test}. Let $t\in [0,T]$,  thanks to classical results based on  Lions-Guelfand triple (see \textit{e.g.} \cite[Theorem 3.1.]{Lions-Magenes}), we get
\begin{align*}
  2  \int_0^t \langle\dfrac{\partial f_\varepsilon}{\partial t},g_\varepsilon\rangle_{V^\prime,V} ds= \int_{\mathbb{T}^2}\int_D\dfrac{f_\varepsilon^2(t)}{\mathbf{M}_\varepsilon}drdx-\int_{\mathbb{T}^2}\int_D\dfrac{f_0^2}{\mathbf{M}_\varepsilon}drdx.
\end{align*}
Therefore, after using  $\int_{\mathbb{T}^2}\int_D\dfrac{f_0^2}{\mathbf{M}_\varepsilon}drdx \leq \dfrac{1}{C}\Vert f_0\Vert_H^2$, we get
\begin{align}\label{key-est-eps-weight}
  \sup_{t\in[0,T]}    \int_{\mathbb{T}^2}\int_D\dfrac{f_\varepsilon^2(t)}{\mathbf{M}_\varepsilon}drdx+2\int_0^T\int_{\mathbb{T}^2}\int_D \mathbf{M}_\varepsilon (\dfrac{1}{\beta}+\dfrac{k_T}{2}\varphi_\varepsilon^2\vert r\vert^2)\vert \nabla_r(\dfrac{f_\varepsilon}{\mathbf{M}_\varepsilon})\vert^2 dr dx ds\leq \dfrac{1}{C}\Vert f_0\Vert_H^2.
\end{align}
As a sequence of \eqref{key-est-eps-weight} and \eqref{relation-weghts}, we get
\begin{align}
     &\sup_{t\in[0,T]}    \int_{\mathbb{T}^2}\int_D\dfrac{f_\varepsilon^2(t)}{\mathbf{M}_\varepsilon^2} \mathbf{M}drdx+\dfrac{2 }{\beta}\int_0^T\int_{\mathbb{T}^2}\int_D \mathbf{M}\vert \nabla_r(\dfrac{f_\varepsilon}{\mathbf{M}_\varepsilon})\vert^2 dr dx ds\leq \dfrac{1}{C^2}\Vert f_0\Vert_H^2,\label{est-limit-1}\\
      &\sup_{t\in[0,T]}    \int_{\mathbb{T}^2}\int_D\dfrac{f_\varepsilon^2(t)}{\widetilde{\mathbf{M}}_0^2} \widetilde{\mathbf{M}}_0drdx\leq \dfrac{1}{C}\Vert f_0\Vert_H^2 \text{ and } \int_0^T\int_{\mathbb{T}^2}\int_D \mathbf{M}_\varepsilon\vert \nabla_r(\dfrac{f_\varepsilon}{\mathbf{M}_\varepsilon})\vert^2 dr dx ds\leq \dfrac{\beta}{2C}\Vert f_0\Vert_H^2.  \label{est-limit-2}
\end{align}
The  inequalities \eqref{est-limit-1} and  \eqref{est-limit-2} ensure that 
\begin{align*}
(\dfrac{f_\varepsilon}{\mathbf{M}_\varepsilon})_\varepsilon  &\text{ is bounded in }  L^\infty(0, T;L^2(\mathbb{T}^2,L^2(D;\mathbf{M}dr)))\cap L^2(0, T;L^2(\mathbb{T}^2,H^1(D;\mathbf{M}dr))),\\
(\dfrac{f_\varepsilon}{ \widetilde{\mathbf{M}}_0})_\varepsilon  &\text{ is bounded in }  L^\infty(0, T;L^2(\mathbb{T}^2,L^2(D; \widetilde{\mathbf{M}}_0dr))),\\
(\sqrt{\mathbf{M}_\varepsilon}\nabla_r(\dfrac{f_\varepsilon}{\mathbf{M}_\varepsilon}))_\varepsilon&\text{ is bounded in }  \big(L^2(0, T;L^2(\mathbb{T}^2\times D))\big)^2.
\end{align*}

  \subsubsection*{Passage to the limit as $\varepsilon\to 0$ in \eqref{eq-to-test}}
By using Banach–Alaoglu theorem, there exist a subsequence of  $(G_\varepsilon:=\dfrac{f_\varepsilon}{\mathbf{M}_\varepsilon})_\varepsilon,$ $(\mathbf{h}_\varepsilon:=\dfrac{f_\varepsilon}{ \widetilde{\mathbf{M}}_0})_\varepsilon $  and $(\sqrt{\mathbf{M}_\varepsilon}\nabla_r(\dfrac{f_\varepsilon}{\mathbf{M}_\varepsilon}))_\varepsilon$ ( denoted by the same way), $G\in L^\infty(0, T;L^2(\mathbb{T}^2,L^2(D;\mathbf{M}dr)))\cap L^2(0, T;L^2(\mathbb{T}^2,H^1(D;\mathbf{M}dr)))$, $\mathbf{h}\in L^\infty(0, T;L^2(\mathbb{T}^2,L^2(D; \widetilde{\mathbf{M}}_0dr)))$ and  $\mathcal{G}\in \big(L^2(0, T;L^2(\mathbb{T}^2\times D))\big)^2$  such that, as $\varepsilon \to 0,$ the following convergences hold
	\begin{align}
		G_\varepsilon \rightharpoonup G &\text{		in	}  L^2(0, T;L^2(\mathbb{T}^2,H^1(D;\mathbf{M}dr))),\label{cv-appro-3-eps-G}\\
		G_\varepsilon \overset{*}{\rightharpoonup} G &\text{		in	}  L^\infty(0, T;L^2(\mathbb{T}^2,L^2(D;\mathbf{M}dr))),\label{cv-appro-4-eps-G}\\
        \mathbf{h}_\varepsilon \overset{*}{\rightharpoonup} \mathbf{h} &\text{		in	}  L^\infty(0, T;L^2(\mathbb{T}^2,L^2(D;\widetilde{\mathbf{M}}_0dr)))\label{cv-appro-5-eps-h},\\
        \sqrt{\mathbf{M}_\varepsilon}\nabla_r(\dfrac{f_\varepsilon}{\mathbf{M}_\varepsilon}) \rightharpoonup  \mathcal{G} &\text{  in }  \big(L^2(0, T;L^2(\mathbb{T}^2\times D))\big)^2.\label{cv-appro-6-eps-h}
	\end{align} Notice that \eqref{est-limit-2} ensures 
$  (f_\varepsilon)_\varepsilon  \text{ is bounded in }  L^\infty(0, T;H_0).$
     By the definition of the space $H_0$ (see \eqref{spaces-2}),  there exists $f=\widetilde{\mathbf{M}}_0\mathbf{h}\in L^\infty(0, T;H_0)$ such that
	\begin{align}
		f_\varepsilon= \widetilde{\mathbf{M}}_0 \mathbf{h}_\varepsilon\overset{*}{\rightharpoonup} \widetilde{\mathbf{M}}_0\mathbf{h}=f &\text{		in	}  L^\infty(0, T;H_0).\label{cv-appro-h-KEY}
	\end{align} 
    \subsubsection*{Identification of the limits}
 By using \eqref{relation-weghts},  it is clear  that $ L^2(D;\widetilde{\mathbf{M}}_0dr) \hookrightarrow L^2(D;\mathbf{M}dr).$ First, let us prove  that $G=\dfrac{f}{\widetilde{\mathbf{M}}_0}$. Indeed, we have
 $G_\varepsilon=\dfrac{f_\varepsilon}{ \widetilde{\mathbf{M}}_0}\dfrac{\widetilde{\mathbf{M}}_0}{\mathbf{M}_\varepsilon}=\mathbf{h}_\varepsilon\dfrac{\widetilde{\mathbf{M}}_0}{\mathbf{M}_\varepsilon}$ with  $\mathbf{h}_\varepsilon \overset{*}{\rightharpoonup} \mathbf{h}$		in	 $ L^\infty(0, T;L^2(\mathbb{T}^2,L^2(D;\widetilde{\mathbf{M}}_0dr)))$ and 
 \begin{align*}    
 \dfrac{\widetilde{\mathbf{M}}_0}{\mathbf{M}_\varepsilon}(r)=\exp{\big(\kappa\sigma\int_0^{\vert r\vert} \dfrac{(1-\varphi_\varepsilon^2)s^3 ds}{(1-s^2)(1+\sigma s^2)(1+\sigma\varphi_\varepsilon^2 s^2)}  \big)}, \quad \sigma=\frac{\lambda\beta}{2 \tau}.
 \end{align*}
 By noticing that 
 \begin{align*}
 (\int_0^{\vert r\vert} \dfrac{(1-\varphi_\varepsilon^2)s^3 ds}{(1-s^2)(1+\sigma s^2)(1+\sigma\varphi_\varepsilon^2 s^2)})^2 \leq \vert r\vert^5(\int_0^1 (1-\varphi_\varepsilon^2)^2 ds) (\int_0^{\vert r\vert} \dfrac{sds}{(1-s^2)^2}) \leq \dfrac{\vert r\vert^7}{2(1-\vert r\vert^2)}\varepsilon,     
 \end{align*}
 one gets that $\dfrac{\widetilde{\mathbf{M}}_0}{\mathbf{M}_\varepsilon} \to 1$ (strongly)  in $C_{loc}(D)$, as $\varepsilon\to0.$ Next, let $\phi\in C^\infty_c(\mathbb{T}^2\times D)$ and $\xi\in C^\infty_c(]0,T[)$ and note that 
 \begin{align*}
 \int_0^T \int_{\mathbb{T}^2}\int_D G\phi\xi dxdrds &=\lim_{\varepsilon\downarrow 0}   \int_0^T \int_{\mathbb{T}^2}\int_D G_\varepsilon\phi\xi dxdrds\\&=\lim_{\varepsilon\downarrow 0} \int_0^T \int_{\mathbb{T}^2}\int_D\mathbf{h}_\varepsilon\dfrac{\widetilde{\mathbf{M}}_0}{\mathbf{M}_\varepsilon}\phi\xi dxdrds=\int_0^T \int_{\mathbb{T}^2}\int_D\mathbf{h}\phi\xi dxdrds,
 \end{align*}
after using \eqref{cv-appro-4-eps-G}  and \eqref{cv-appro-5-eps-h}.
Consequently, we get $G=\mathbf{h}=\dfrac{f}{\widetilde{\mathbf{M}}_0},$ where 
   $f=\widetilde{\mathbf{M}}_0\mathbf{h}\in L^\infty(0, T;H_0).$\\
   Secondly, let us prove that $ \mathcal{G}=\sqrt{\widetilde{\mathbf{M}}_0}\nabla_r(\dfrac{f}{\widetilde{\mathbf{M}}_0})$. Indeed, from \eqref{weight-inter} and \eqref{relation-weghts} one gets that 
   \begin{align}\label{cv-ident-1}
    \sqrt{\mathbf{M}_\varepsilon} \to \sqrt{\widetilde{\mathbf{M}}_0} \text{ in } L^p(D) \text{ for any }  1\leq p<+\infty.
   \end{align}Moreover,  we have 
   \begin{align}\label{cv-ident-2}
          \nabla_r(\dfrac{f_\varepsilon}{\mathbf{M}_\varepsilon}) \rightharpoonup \nabla_r(\dfrac{f}{\widetilde{\mathbf{M}}_0}) \text{ in } (L^2(0, T;L^2(\mathbb{T}^2,L^2(D;\mathbf{M}dr))))^2,
   \end{align}
     thanks to  \eqref{cv-appro-3-eps-G}. Let  $\phi\in (C^\infty_c(\mathbb{T}^2\times D))^2$ and $\xi\in C^\infty_c(]0,T[)$. Therefore,  \eqref{cv-appro-6-eps-h}, \eqref{cv-ident-1} and \eqref{cv-ident-2} ensure
     \begin{align*}
         \int_0^T \int_{\mathbb{T}^2}\int_D \mathcal{G}\phi\xi dxdrds &=\lim_{\varepsilon\downarrow 0}   \int_0^T \int_{\mathbb{T}^2}\int_D   \sqrt{\mathbf{M}_\varepsilon}\nabla_r(\dfrac{f_\varepsilon}{\mathbf{M}_\varepsilon})\phi\xi dxdrds=\int_0^T \int_{\mathbb{T}^2}\int_D\sqrt{\widetilde{\mathbf{M}}_0}\nabla_r(\dfrac{f}{\widetilde{\mathbf{M}}_0})\phi\xi dxdrds,
     \end{align*}
which guarantees the claim.
In conclusion, we proved
   as $\varepsilon \to 0,$ it holds
 \begin{align}
			f_\varepsilon&\overset{*}{\rightharpoonup} f \text{		in	}  L^\infty(0, T;H_0),\label{cv-appro-5-eps-h_*}\\
           \sqrt{\mathbf{M}_\varepsilon}\nabla_r(\dfrac{f_\varepsilon}{\mathbf{M}_\varepsilon}) &\rightharpoonup  \sqrt{\widetilde{\mathbf{M}}_0}\nabla_r(\dfrac{f}{\widetilde{\mathbf{M}}_0})  \text{  in }  \big(L^2(0, T;L^2(\mathbb{T}^2\times D))\big)^2.\label{cv-appro-6-eps-h-*}
	\end{align}
    In particular, \eqref{cv-appro-6-eps-h-*} ensures $\int_0^T\int_{\mathbb{T}^2}\int_D \widetilde{\mathbf{M}}_0\vert \nabla_r(\dfrac{f}{\widetilde{\mathbf{M}}_0})\vert^2 dr dx ds\leq \dfrac{\beta}{2C}\Vert f_0\Vert_H^2$.\\
    
       Let $\widetilde{\phi}\in L^2(\T^2)\otimes C^1(\overline{D})$ and set $\Phi=\mathbf{M}\widetilde{\phi}.$ Then, it is easy to check that $\Phi   \in V.$
    Again, let $\xi\in C^\infty_c(]0,T[)$ and 
    set $\Phi=\mathbf{M}\widetilde{\phi}$ in \eqref{eq-to-test}. Next,   multiply by $\xi$ the equation \eqref{eq-to-test} and integrate from $0$ to $T$ to get
    \begin{align*}
   \int_0^T \langle\dfrac{\partial f_\varepsilon}{\partial t},\mathbf{M}\widetilde{\phi}\rangle_{V^\prime,V} \xi dt=-\int_0^T\int_{\mathbb{T}^2}\int_D\mathbf{M}_\varepsilon (\dfrac{1}{\beta}I+\dfrac{1}{2}\varphi_\varepsilon^2A(r))\nabla_r(\dfrac{f_\varepsilon}{\mathbf{M}_\varepsilon})\cdot \nabla_r\widetilde{\phi} \xi drdx dt.
\end{align*}
First, recall that  $\varphi_\varepsilon \to 1$   in  in $L^p(D)$ for any $1\leq p<+\infty.$
Next, we perform an  integration by parts in time, we get after using \eqref{cv-appro-5-eps-h_*}
\begin{align*}
      \int_0^T \langle\dfrac{\partial f_\varepsilon}{\partial s},\mathbf{M}\widetilde{\phi}\rangle_{V^\prime,V} \xi ds&=- \int_0^T (f_\varepsilon,\mathbf{M}\widetilde{\phi})_{H}  \xi^\prime(s) ds=- \int_0^T  \int_{\mathbb{T}^2}\int_D  f_\varepsilon \widetilde{\phi}\xi^\prime(s) drdxds\\
\text{ ( recall that }  H_0\hookrightarrow L^2(\T^2\times D))   &= - \int_0^T (f_\varepsilon,\widetilde{\phi})  \xi^\prime(s) ds  \to  - \int_0^T (f,\widetilde{\phi})  \xi^\prime(s) ds,
\end{align*}
 and, by using   \eqref{cv-appro-6-eps-h-*}, we get
    \begin{align}
   \lim_{\varepsilon \downarrow 0}     &-\int_0^T\int_{\mathbb{T}^2}\int_D\mathbf{M}_\varepsilon (\dfrac{1}{\beta}I+\dfrac{1}{2}\varphi_\varepsilon^2A(r))\nabla_r(\dfrac{f_\varepsilon}{\mathbf{M}_\varepsilon})\cdot \nabla_r\widetilde{\phi} \xi drdx dt\notag\\
   &=-\int_0^T\int_{\mathbb{T}^2}\int_D\widetilde{\mathbf{M}}_0 (\dfrac{1}{\beta}I+\dfrac{1}{2}A(r))\nabla_r(\dfrac{f}{\widetilde{\mathbf{M}}_0})\cdot \nabla_r\widetilde{\phi} \xi drdx dt.\label{eq-regularity}
    \end{align}
   Since $    \overline{C^1(\overline{D})}^{\mathcal{H}^1_{\widetilde{\mathbf{M}}_0}}=\mathcal{H}^1_{\widetilde{\mathbf{M}}_0}$ (see \cite[Proposition B.2]{Masmoudi2013}) and taking into account that $f \in L^\infty(0, T;H_0)\cap   L^2(0, T;V_0)$, we get
   \begin{align*}
       - \int_0^T (f,\phi)_{H_0}  \xi^\prime(s) ds=\int_0^T\int_{\mathbb{T}^2}\int_D\widetilde{\mathbf{M}}_0 (\dfrac{1}{\beta}I+\dfrac{1}{2}A(r))\nabla_r(\dfrac{f}{\widetilde{\mathbf{M}}_0})\cdot \nabla_r\big( \dfrac{\phi}{\widetilde{\mathbf{M}}_0}\big) \xi drdx dt,\quad \forall \phi \in V_0.
   \end{align*}
In particular, $f \in L^\infty(0, T;H_0)\cap   L^2(0, T;V_0)$  solves in the sense of distribution (in time) 
  \begin{align}\label{fom-limit-1}
       \partial_t f= \Div_r\bigl(\widetilde{\mathbf{M}}_0 (\dfrac{1}{\beta}I+\dfrac{1}{2}A(r))\nabla_r(\dfrac{f}{\widetilde{\mathbf{M}}_0})\bigr) \text{ in } (V_0)^\prime.
    \end{align}
 Now, let $\phi \in V_0$. From \eqref{fom-limit-1}, we have also $\partial_tf\in L^2(0, T;(V_0)^\prime).$ Thus by considering the Lions-Guelfand triple $V_0\hookrightarrow H_0\hookrightarrow (V_0)^\prime$ and taking $\xi\in C^\infty([0,T])$, we write
   for any $t\in ]0,T]$
      \begin{align}\label{IPPtimez-eps-1}
	 \quad	\int_0^t	\langle\dfrac{\partial f}{ds}(s),\phi\rangle_{(V_0)^\prime,V_0}\xi(s) ds&=-\int_0^t(f,\phi)_{H_0}\xi^\prime ds+(  f(t) ,\phi)_{H_0}\xi(t)- (  f(0) ,\phi)_{H_0}\xi(0).
	\end{align}
    Now, let  $\widetilde{\phi}\in L^2(\T^2)\otimes C^1(\overline{D})$ and notice that, in particular, for $\phi=\widetilde{\mathbf{M}}_0\widetilde{\phi}$ in  \eqref{IPPtimez-eps-1}, we have 
         \begin{align}\label{IPPtimez-eps}
	 \quad	\int_0^t	\langle\dfrac{\partial f}{ds}(s),\widetilde{\mathbf{M}}_0\widetilde{\phi}\rangle_{(V_0)^\prime,V_0}\xi(s) ds&=-\int_0^t(f,\widetilde{\phi})\xi^\prime ds+(  f(t) ,\widetilde{\phi})\xi(t)- (  f(0) ,\widetilde{\phi})\xi(0).
	\end{align}
    On the other hand,       set $\Phi=\mathbf{M}\widetilde{\phi}  \in V.$
For any 
      $\xi\in C^\infty([0,T])$ we have
        \begin{align}\label{IPPtimez-eps-2}
	 \quad	\int_0^t	\langle\dfrac{\partial f_\varepsilon}{ds}(s),\mathbf{M}\widetilde{\phi}\rangle_{V^\prime,V}\xi(s) ds&=-\int_0^t(f_\varepsilon,\widetilde{\phi})\xi^\prime ds+(  f_\varepsilon(t) ,\widetilde{\phi})\xi(t)- (  f_0,\widetilde{\phi})\xi(0).
	\end{align}
    By using \eqref{IPPtimez-eps}, \eqref{IPPtimez-eps-2}, \eqref{eq-to-test}  and \eqref{fom-limit-1}, we get for any $\widetilde{\phi}\in L^2(\T^2)\otimes C^1(\overline{D})$
     \begin{align}
          \lim_{\varepsilon \downarrow 0}  (  f_\varepsilon(t) ,\widetilde{\phi})\xi(t)=(  f(t) ,\widetilde{\phi})\xi(t)- ( f_0- f(0) ,\widetilde{\phi})\xi(0), \quad \forall t\in ]0,T].
    \end{align}
  Finally, by taking into account that $f_\varepsilon, f\in L^\infty(0,T,H_0); \varepsilon>0$ and  $    \overline{C^1(\overline{D})}^{\mathcal{H}^1_{\widetilde{\mathbf{M}}_0}}=\mathcal{H}^1_{\widetilde{\mathbf{M}}_0}$, we get
    \begin{align}
          \lim_{\varepsilon \downarrow 0}  (  f_\varepsilon(t) ,\phi)_{H_0}\xi(t)=(  f(t) ,\phi)_{H_0}\xi(t)- ( f_0- f(0) ,\phi)_{H_0}\xi(0), \quad  \forall \phi\in H_0, \quad \forall t\in ]0,T].
    \end{align}
    By appropriate choice of $\xi$, we get first $f(0)=f_0$ and then $f_\varepsilon(t)  \rightharpoonup f(t)$  in $H_0$ for any $t\in [0,T].$\\
 Finally,  we prove the following.
  \begin{proposition}(Stability and uniqueness)\label{propo-stab-}
      Let $f_1$ and $f_2$ be two solutions to  \eqref{fom-limit-1} with initial data $f_0^1\in H_0$ and $f_0^2\in H_0$ respectively. Then, we have
      \begin{align*}
          \Vert f_1(t)-f_2(t)\Vert_{H_0}^2+2\int_0^t\int_{\mathbb{T}^2}\int_D\bigl(\widetilde{\mathbf{M}}_0 (\dfrac{1}{\beta}+\dfrac{k_T}{2}\vert r\vert^2)\vert \nabla_r(\dfrac{f_1-f_2}{\widetilde{\mathbf{M}}_0})\vert^2 drdxds \leq   \Vert f_0^1-f_0^2\Vert_{H_0}^2.
      \end{align*}
      In particular, the solution to \eqref{fom-limit-1} is unique.
  \end{proposition}
  \subsubsection*{Proof of \autoref{propo-stab-}}
  Set $f=f_1-f_2$ and $f_0=f_0^1-f_0^2$.
    Let  $t\in ]0,T[$ and notice that $f$ is an admissible test function in \eqref{fom-limit-1}. Therefore, after integration from $0$ to $t$ we get 
       \begin{align*}
        \int_0^t \langle\partial_s f,f\rangle_{(V_0)^\prime,V_0} ds=-\int_0^t\int_{\mathbb{T}^2}\int_D\bigl(\widetilde{\mathbf{M}}_0 (\dfrac{1}{\beta}I+\dfrac{1}{2}A(r))\nabla_r(\dfrac{f}{\widetilde{\mathbf{M}}_0})\bigr)\cdot \nabla_r\left(\dfrac{f}{\widetilde{\mathbf{M}}_0}\right) drdxds 
\end{align*}
Notice that 
\begin{align*}
    -&\int_0^t\int_{\mathbb{T}^2}\int_D\bigl(\widetilde{\mathbf{M}}_0 (\dfrac{1}{\beta}I+\dfrac{1}{2}A(r))\nabla_r(\dfrac{f}{\widetilde{\mathbf{M}}_0})\bigr)\cdot \nabla_r\left(\dfrac{f}{\widetilde{\mathbf{M}}_0}\right) drdxds\\&\leq -\int_0^t\int_{\mathbb{T}^2}\int_D\bigl(\widetilde{\mathbf{M}}_0 (\dfrac{1}{\beta}+\dfrac{k_T}{2}\vert r\vert^2)\vert \nabla_r(\dfrac{f}{\widetilde{\mathbf{M}}_0})\vert^2 drdxds.
\end{align*}
Thanks to classical results based on  Lions-Guelfand
triple (see \textit{e.g.} \cite[Theorem 3.1.]{Lions-Magenes}), we have  $$2  \int_0^t \langle\partial_s f,f\rangle_{(V_0)^\prime,V_0} ds=\Vert f(t)\Vert_{H_0}^2-\Vert f_0\Vert_{H_0}^2,$$ 
which ensures the claim.
 
\section{Singular limit as $\tau \downarrow 0$ and stationary density (\autoref{Thm-singular-limit-main})}\label{Sec-singular-limit}
  From \autoref{section-remove-cut-off},   recall that
  \begin{align*}
  \widetilde{\mathbf{M}}_0(r)=\left(
\dfrac{1 - \vert r\vert^{2}}{1 + \frac{\lambda\beta}{2 \tau}\, \vert r\vert^{2}}
\right)^{\frac{\kappa}{2\left(1 + \frac{\lambda\beta}{2 \tau}\right)}}.
    \end{align*}
  Let $\tau>0,$      in the regime where the relaxation time of polymers and the dominant time-scale of the small scale turbulent flow satisfies $\beta=\zeta \tau, \zeta>0$,  $       \widetilde{\mathbf{M}}_0$ and the matrix $A$ become
        \begin{align*}
       \widetilde{\mathbf{M}}_0(r)=\left(
\dfrac{1 - \vert r\vert^{2}}{1 + \alpha \vert r\vert^{2}}
\right)^{\frac{\kappa}{2\left(1 + \alpha \right)}},\quad     \dfrac{1}{2}A(r)
		=\dfrac{\alpha}{\zeta \tau} (3 \left\vert r\right\vert^2I-2r\otimes r),\quad \alpha=\frac{\zeta\lambda}{2}.
        \end{align*}
        Now, we consider the normalized $\widetilde{\mathbf{M}}_0$ denoted by $\mathbf{M}_0=\dfrac{1}{\mathbf{Z}}\widetilde{\mathbf{M}}_0$ where $\mathbf{Z}=\int_D \widetilde{\mathbf{M}}_0 dr$.\\
        
      Thanks to   \autoref{THm-eps-nocut}, there exists  a unique $f:=f_\tau$ such that   $f_\tau \in L^2(0, T;V_0) \cap C([0, T];H_0)$, $f_\tau(0)=f_{\tau,0} $ and solves
        \begin{align}\label{limit-PDE-final-NOCUT-V1*****}
    \langle\partial_t f_\tau,\Phi\rangle_{(V_0)^\prime,V_0}=- \dfrac{1}{\zeta \tau}\int_{\mathbb{T}^2}\int_D\bigl(\mathbf{M}_0 (I+\mathcal{A}(r))\nabla_r(\dfrac{f_\tau}{\mathbf{M}_0})\bigr)\cdot \nabla_r\left(\dfrac{\Phi}{\mathbf{M}_0}\right) drdx, \quad \forall \Phi\in V_0,
\end{align}
where we use the notation 
\begin{align}\label{limit-matrix}
 \mathcal{A}(r)= \alpha(3 \left\vert r\right\vert^2I-2r\otimes r).   
\end{align}
In the following, we identify the limit as $\tau \downarrow 0$ in \eqref{limit-PDE-final-NOCUT-V1*****}. Before that, let us make the following remark.
\begin{remark}
    From \eqref{limit-PDE-final-NOCUT-V1*****}, we can guess formally  that the penalization $\tau \to 0$ leads to 
$$f_\tau(t,x,r) \simeq \rho(t,x)p(r), \quad p\in Ker(\mathcal{L}),$$
where $\mathcal{L}(g)=-\Div_r\bigl(\mathbf{M}_0 (I+\mathcal{A}(r))\nabla_r(\dfrac{g}{\mathbf{M}_0})\bigr)$ is the unbounded operator defined on  $L^2_{\mathbf{M}_0}$ with domain $\mathcal{D}(\mathcal{L})=\{ g\in H^1_{\mathbf{M}_0};\quad \mathcal{L}(g)\in L^2_{\mathbf{M}_0} \text{ and }  \mathbf{M}_0 (I+\mathcal{A}(r))\nabla_r(\dfrac{g}{\mathbf{M}_0})\vert_{\partial D}=0\}.$ One can check that $-\mathcal{L}$ is self adjoint and positive  on $L^2_{\mathbf{M}_0}$, we refer to \cite[Prop. 3.6.]{Masmoudi2008} for the proof in the case of  similar operator, see also \cite[Prop. 13]{tahraoui2025small}. Furthermore, it is not difficult to verify that
$\ker \mathcal{L} = \{\, c \, \mathbf{M}_0 \; : \; c \in \mathbb{R} \,\}.$
\end{remark}

\subsection*{Statement of the result}
Consider the following assumption.
\begin{itemize}
    \item[(H$_1$)] There exists $f_0\in H_0$ such that      $       \displaystyle\lim_{\tau\to 0}\Vert f_{\tau,0}-f_0\Vert_{H_0} =0.  $
\end{itemize}
Notice that (H$_1$) ensures  the existence of  $\mathbf{\Lambda}>0$ independent of $\tau$  such that $\displaystyle\sup_{\tau>0}\Vert f_{\tau,0} \Vert_{H_0}^2\leq  \mathbf{\Lambda}.$
\begin{remark}
    Since $H_0 \hookrightarrow L^2(\mathbb{T}^2,L^1(D)),$ one has $\rho_0=\int_D f_0(x,r) dr\in L^2(\mathbb{T}^2).$
\end{remark}
First, similar argument to the proof of \autoref{propo-stab-} ensures
\begin{lemma}\label{lem-estimate-tau-limit}
  Let $\tau>0$ and    $f_{\tau,0}\in H_0$ such that H$_1$ holds. Then the  unique solution $f_\tau$ of \eqref{limit-PDE-final-NOCUT-V1*****} satisfies for every $t\ge 0$
     \begin{align}
          \xi\tau\Vert f_\tau(t)\Vert_{H_0}^2+2\int_0^t\int_{\mathbb{T}^2}\int_D \mathbf{M}_0 (1+\alpha\vert r\vert^2)\vert \nabla_r(\dfrac{f_\tau}{\mathbf{M}_0})\vert^2 drdxds \leq   \xi\tau\Vert f_{\tau, 0}\Vert_{H_0}^2\leq \xi\tau\Lambda.
      \end{align}
\end{lemma}
\begin{theorem}\label{Thm-singular-limit-PF}
    Assume that $f_{\tau, 0} \to f_0$ in $H_0$ as $\tau \downarrow 0.$ Then the following  hold.
    \begin{itemize}
        \item $ f_\tau$ converges to 
        $ \rho_0\otimes \M_0$ in $L^2(0,T;H_0)$, 
     where  $\rho_0(x)=\int_D f_0(x,r) dr$ and $f_\tau$ is the unique solution  to \eqref{limit-PDE-final-NOCUT-V1*****}. Moreover, the following estimate (convergence rate) holds
              \begin{align}\label{ineq-converg-equi}
 \int_0^T  \Vert f_{\tau}(t)-\rho_0\otimes \mathbf{M}_0\Vert_{H_0}^2 dt \leq C_P  \xi\tau (\Vert f_{\tau, 0}\Vert_{H_0}^2+\Vert \rho_0 \Vert^2_{L^2(\mathbb{T}^2)})+2T\Vert f_{\tau,0}-f_0\Vert_{H_0}^2.
          \end{align}
          \item Assume that $f_{\tau,0}$ converges to $ \rho_{0}\otimes\mathbf{M}_0$ in $H_0.$ Then,  $ f_\tau$ converges to 
        $ \rho_0\otimes \M_0$ in $C([0,T];H_0)$, 
     where  $\rho_0(x)=\int_D f_0(x,r) dr$. Moreover, the following estimate (convergence rate) holds
     \begin{align}\label{ineq-conv-equi-Cont}
         \displaystyle\sup_{t\in [0,T]}\Vert f_{\tau}(t)-\rho_0\otimes \mathbf{M}_0\Vert_{H_0}^2 \leq   \Vert f_{\tau,0}- \rho_{0}\otimes\mathbf{M}_0\Vert_{H_0}^2.
     \end{align}
    \end{itemize}
  
\end{theorem}
\subsubsection*{Proof of \autoref{Thm-singular-limit-PF}}
Let $\tau>0.$
  Since $\mathbf{M}_0$ solves $\kappa rF(r)\M_0 + (I+\mathcal{A})\nabla_r \M_0=0$, set   $$g_\tau(t,x,r)= f_\tau(t,x,r)- \rho_0(x)\mathbf{M}_0(r) \text{ and } g_{\tau,0}=f_{\tau,0}- \rho_{0}\otimes\mathbf{M}_0.$$ Next, we can easily verify that $g_\tau$ satisfies $g_\tau(0)=g_{\tau,0}$ and 
   \begin{align}\label{limit-PDE-g}
    \langle\partial_t g_\tau,\Phi\rangle_{(V_0)^\prime,V_0}=- \dfrac{1}{\zeta \tau}\int_{\mathbb{T}^2}\int_D\bigl(\mathbf{M}_0 (I+\mathcal{A}(r))\nabla_r(\dfrac{g_\tau}{\mathbf{M}_0})\bigr)\cdot \nabla_r\left(\dfrac{\Phi}{\mathbf{M}_0}\right) drdx, \forall \Phi\in V_0.
\end{align}
Notice that $\rho_{0}\otimes\mathbf{M}_0\in V_0,$ which ensures that $g_\tau$ is an admissible test function in \eqref{limit-PDE-g}. Therefore, similarly to \autoref{propo-stab-}, we get
\begin{align}\label{ineq-limit}    \xi\tau \sup_{t\in [0,T]}\Vert g_\tau(t)\Vert_{H_0}^2+2\int_0^T\int_{\mathbb{T}^2}\int_D \mathbf{M}_0 (1+\alpha\vert r\vert^2)\vert \nabla_r(\dfrac{g_\tau}{\mathbf{M}_0})\vert^2 drdxds \leq   \xi\tau\Vert g_{\tau, 0}\Vert_{H_0}^2.
      \end{align}
      Let us discuss the consequence of the last inequality \eqref{ineq-limit}.\\
     From \eqref{ineq-limit}, we get $\displaystyle\sup_{t\in [0,T]}\Vert g_\tau(t)\Vert_{H_0}^2 \leq   \Vert g_{\tau, 0}\Vert_{H_0}^2,$ which ensures that $\displaystyle\lim_{\tau\to 0}f_\tau= \rho_0\otimes \mathbf{M}_0$ in $C([0,T],H_0)$ if $ \displaystyle\lim_{\tau\to 0}f_{\tau,0}= \rho_{0}\otimes\mathbf{M}_0$ in $H_0.$
          On the other hand, the same inequality \eqref{ineq-limit} ensures
          \begin{align}\label{ineq-key-proof-sing}
              \int_0^T\int_{\mathbb{T}^2}\int_D \mathbf{M}_0 (1+\alpha\vert r\vert^2)\vert \nabla_r(\dfrac{g_\tau}{\mathbf{M}_0})\vert^2 drdxds \leq   \dfrac{\xi\tau}{2}\Vert g_{\tau, 0}\Vert_{H_0}^2 \leq \xi\tau (\Vert f_{\tau, 0}\Vert_{H_0}^2+\Vert \rho_0 \Vert^2_{L^2(\mathbb{T}^2)}).
          \end{align}
          Let us prove that \eqref{ineq-key-proof-sing} ensures \eqref{ineq-converg-equi}. Notice that \eqref{ineq-limit} ensures that $\dfrac{g_\tau}{\mathbf{M}_0} \in L^2(0, T; L^2(\T^2; H^1(D; \M_0dr)))$ for any $\tau >0.$
          By using  Poincaré-Wirtinger inequality  with the weight $\mathbf{M}_0$ (see \textit{e.g.} \cite[Prop. 3.4.]{Masmoudi2008}), we write 
       \begin{align}\label{Poincare}
\int_0^T \int_{\mathbb T^2} \int_D 
\frac{|g_\tau(x,r,t)|^2}{\mathbf M_0(r)}
\,dr\,dx\,dt
&\le C_P
\int_0^T \int_{\mathbb T^2} \int_D
\mathbf M_0(r)
\left|
\nabla_r \!\left(
\frac{g_\tau(x,r,t)}{\mathbf M_0(r)}
\right)
\right|^2
dr\,dx\,dt \\
&\quad
+2
\int_0^T \int_{\mathbb T^2}
\left(
\int_D g_\tau(x,r,t)\,dr
\right)^2
dx\,dt,\notag
       \end{align}
where $C_P>0$, independent of $g_\tau$ and $\tau$.  Next, for any $t\in [0,T],$ let $\xi\in C^\infty([0,t]),$ $\psi\in C^\infty(\T^2)$ and set $\Phi=\psi\mathbf{M}_0$ in  \eqref{limit-PDE-g}. Then,  multiply the resulting equation by $\xi$ and integrate from $0$ to $t$  to get 
\begin{align*}
 0&=   \int_0^t\langle\partial_s g_\tau,\psi\mathbf{M}_0\rangle_{(V_0)^\prime,V_0}\xi ds\\&=-\int_0^t\langle g_\tau,\psi\mathbf{M}_0\rangle_{(V_0)^\prime,V_0}\xi^\prime ds+ (  g_\tau(t) ,\psi\mathbf{M}_0)_{H_0}\xi(t)- ( g_{\tau,0},\psi\mathbf{M}_0)_{H_0}\xi(0).
\end{align*}
In particular, set $\xi\equiv 1$ and using the definition of $H_0$ we obtain
\begin{align}\label{limit-PDE-g-rho}
\text{for any } t\in [0,T]:   0=\int_{\mathbb T^2} \int_D  g_\tau(t,x,r) \psi(x) drdx-\int_{\mathbb T^2} \int_D g_{\tau,0}(x,r)\psi(x) drdx,
\end{align}
which gives $
   \displaystyle \int_{\mathbb T^2} \bigl(\int_D  g_\tau(t,x,r)dr- \int_D g_{\tau,0}(x,r) dr\bigr)\psi(x) dx=0.
$ Since $\psi$ is arbitrary and by density  of smooth functions in $L^2(\T^2)$, we obtain 
\begin{align}\label{eq-limit-rho-1}
    \text{for any } t\in [0,T], \text{ a.e. } x\in\T^2: \quad \int_D  g_\tau(t,x,r)dr= \int_D g_{\tau,0}(x,r) dr.
\end{align}
Moreover, notice that
\begin{align}\label{eq-limit-rho-2}
    \int_D g_{\tau,0}(x,r)\,dr=\int_D f_{\tau,0}(x,r)\,dr-\rho_0(x)=\rho_\tau(x)-\rho_0(x).
\end{align}
Consequently, by using \eqref{eq-limit-rho-1} and \eqref{eq-limit-rho-2} we have
\begin{align*}
    \int_0^T \int_{\mathbb T^2}
\left(
\int_D g_\tau(x,r,t)\,dr
\right)^2
dx\,dt&= \int_0^T \int_{\mathbb T^2}
\left(\int_D g_{\tau,0}(x,r)\,dr
\right)^2
dx\,dt\\&=\int_0^T \int_{\mathbb T^2}
\left(\rho_\tau(x)-\rho_0(x)
\right)^2
dx\,dt=T\Vert \rho_\tau-\rho_0 \Vert^2_{L^2(\mathbb T^2)}. 
\end{align*}
Finally, by using Cauchy Schwarz inequality, we get
\begin{align*}
    \Vert \rho_\tau-\rho_0 \Vert^2_{L^2(\mathbb T^2)}=\int_{\mathbb T^2}
\left(\int_D [f_{\tau,0}(x,r)-f_0(x,r)]\,dr
\right)^2
dx\leq  \Vert f_{\tau,0}-f_0\Vert_{H_0}^2.
\end{align*}
Thus, by inserting the last inequality in \eqref{Poincare}, we proved
  \begin{align}\label{Poincare-1}
\int_0^T  \Vert g_{\tau}(t)\Vert_{H_0}^2 dt
&\le C_P
\int_0^T \int_{\mathbb T^2} \int_D
\mathbf M_0
\left|
\nabla_r \!\left(
\frac{g_\tau(t)}{\mathbf M_0}
\right)
\right|^2
dr\,dx\,dt 
+2T\Vert f_{\tau,0}-f_0\Vert_{H_0}^2.
       \end{align}
By combining the last inequality \eqref{Poincare-1} and \eqref{ineq-key-proof-sing}, we get 
         \begin{align}\label{ineq-key-proof-sing++}
 \int_0^T  \Vert g_{\tau}(t)\Vert_{H_0}^2 dt \leq C_P  \xi\tau (\Vert f_{\tau, 0}\Vert_{H_0}^2+\Vert \rho_0 \Vert^2_{L^2(\mathbb{T}^2)})+2T\Vert f_{\tau,0}-f_0\Vert_{H_0}^2,
          \end{align}
          which concludes the proof.
          
         \begin{acknowledgements}  The research of  Y.T. is
funded by the European Union (ERC, NoisyFluid, No. 101053472). Views and opinions expressed are however those of the author only and do not necessarily reflect those of the European Union or the European Research Council. Neither the European Union nor the granting authority can be held responsible for them. The authors  would like to thank  Franco Flandoli for   valuable discussion. 
\end{acknowledgements}
\bibliographystyle{plain} 
\bibliography{FENE-Tahraoui-Butori}
\end{document}